\documentclass{article}

\usepackage[backend=bibtex, 
sortcites=true,
firstinits=true,
hyperref,
maxbibnames=99,
]{biblatex}

\bibliography{Bibliography}{}

\usepackage{enumerate}

\usepackage[centertags]{amsmath}
\usepackage{amsfonts}
\usepackage{amssymb}
\usepackage{amsthm}
\usepackage{newlfont}
\usepackage{mathtools}
\usepackage[top=1in, bottom=1in, left=1in, right=1in]{geometry}
\usepackage{hyperref}
\usepackage{tikz}
\usepackage{algpseudocode}

\theoremstyle{plain}
\newtheorem{theorem}{Theorem}[section]
\newtheorem{proposition}[theorem]{Proposition}
\newtheorem{corollary}[theorem]{Corollary}
\newtheorem{lemma}[theorem]{Lemma}

\theoremstyle{definition}
\newtheorem{definition}[theorem]{Definition}

\newtheorem*{remark}{Remark}

\newcommand{\R}{\mathbb{R}}
\newcommand{\N}{\mathbb{N}}
\newcommand{\Z}{\mathbb{Z}}

\newcommand{\SPAN}{\operatorname{span}}

\newcommand{\diam}{\operatorname{diam}}
\newcommand{\sgn}{\operatorname{sgn}}

\newcommand{\eps}{\varepsilon}

\newcommand{\Imm}{\operatorname{Im}}

\newcommand{\Hidden}[1]{}

\newcommand{\ON}{S_1^{comb}}
\newcommand{\CN}{B_1^{comb}}

\newcommand{\SG}{S}
\newcommand{\AUT}{\operatorname{Aut}}
\newcommand{\CAY}{\operatorname{Cay}}

\newcommand{\MBE}{M_2^{\mathrm{path}}}

\newcommand{\dst}{\delta^{(*)}}
\newcommand{\pr}{q}
\newcommand{\eW}{E(W,V)}

\DeclarePairedDelimiter{\abs}{\lvert}{\rvert}

\begin{document}

\title{Betti number estimates for non-negatively curved graphs}
\author{
Moritz Hehl \footnote{Leipzig University,  SECAI, moritz.hehl@uni-leipzig.de}~~~~~~~
Florentin M\"unch\footnote{Leipzig University, cfmuench@gmail.com}
}
\date{\today}
\maketitle

\begin{abstract}
In this paper, we establish Betti number estimates for graphs with non-negative Ollivier curvature, and for graphs with non-negative Bakry-\'Emery curvature, providing a discrete analogue of a classical result by Bochner for manifolds \cite{Bochner1946VectorFA}.
Specifically, we show that for graphs with non-negative Ollivier curvature, the first Betti number is bounded above by half of the minimum combinatorial vertex degree.
In contrast, for graphs with non-negative Bakry-\'Emery curvature, we prove that the first Betti number of the path homology is bounded above by the minimum combinatorial vertex degree minus one.
We further present various rigidity results, characterizing graphs that attain the upper bound on the first Betti number under non-negative Ollivier curvature. Remarkably, these graphs are precisely the discrete tori, similar to the Riemannian setting. Furthermore, we show that the results obtained using the Ollivier curvature extend to the setting of potentially non-reversible Markov chains.
Additionally, we explore rigidity cases depending on the idleness parameter of the Ollivier curvature, i.e., we characterize rigidity for bone-idle graphs with non-negative Ollivier curvature that attain the upper Betti number bound.
We further establish an upper bound on the first Betti number under a more general assumption, where non-negative Ollivier curvature is required only outside a finite subset.
Finally, we provide several examples, e.g., we prove that for a potentially non-reversible Markov chain on a cycle of length at least five, there always exists a unique path metric with constant Ollivier curvature. Moreover, this metric has non-negative Ollivier curvature, and the upper Betti number bound is attained if and only if the curvature is zero.
\end{abstract}

\tableofcontents

\section{Intorduction}
The relation between homology and curvature plays an important role in geometry.
A seminal paper by Bochner shows that for manifolds with non-negative Ricci curvature, the first Betti number is bounded from above by the dimension \cite{Bochner1946VectorFA}. Moreover, the equality only holds for the flat torus.

In the last 20 years, there has been increasing interest in finding Betti number estimates for cell complexes with certain curvature constraints.
Firstly, Forman introduced a Ricci curvature notion on cell complexes via the Bochner Weitzenb\"ock decomposition of the discrete Hodge Laplacian, and gave an upper bound for the Betti number in terms of the vertex degree in case the Forman curvature was non-negative \cite[Theorem~4.5]{forman2003bochner}.

In the following years, various different discrete Ricci curvature notions have been introduced.
In 2009, Ollivier gave a Ricci curvature notion of Markov chains via discrete transport \cite{ollivier2009ricci}.
In 2012, Lin and Yau defined a Ricci curvature for graphs via the Bakry \'Emery calculus \cite{lin2010ricci}. In the same year, Erbar and Maas gave a discrete Ricci curvature definition based on the displacement convexity of the entropy \cite{erbar2012ricci} (also see \cite{mielke2013geodesic}). One connecting theme between the different discrete Ricci curvature notions is that lower bounds can be characterized via gradient estimates for the heat semigroup of the form
\[
\|\nabla P_t f\| \leq e^{-Kt} \|\nabla f\|
\]
for suitable norms where $P_t$ denotes the heat semigroup, see \cite{lin2015equivalent,munch2017ollivier,erbar2018poincare}.

For graphs with strictly positive Bakry-\'Emery curvature, it was shown in  \cite{kempton2021homology} that the first Betti number of the path homology is vanishing. For more details about the path homology, see \cite{grigor2012homologies,grigor2019homology}. It turns out that for undirected graphs, the first homology class is equivalent to the one arising from the cell complex where 2-cells are attached to all cycles of length at most four.
The vanishing Betti number result was improved in \cite{munch2020spectrally} where the assumption of strictly positive Bakry-\'Emery curvature was relaxed, now allowing for some negative curvature.
In contrast, no Betti number estimates have been published for Ollivier or entropic curvature yet. Indeed, for entropic curvature, no Betti number number estimate can be expected as the curvature is non-local, meaning that long cycles have strictly positive entropic curvature, with the curvature converging to zero when the size of the cycle tends to infinity, \cite[Section~3]{kamtue2024entropic}.

Independent of curvature, there has been significant interest in calculating Betti numbers of metric spaces in data analysis via the so called persistent homology, see \cite{zomorodian2005Computing,Otter2015ARF}.
On the other hand, discrete Ricci curvature also proved to be an excellent tool for data analysis \cite{samal2018comparative,
sandhu2015analytical,ni2018network,
ni2019community,saucan2018discrete}.

In this paper, we give Betti number estimates for graphs with non-negative Ollivier curvature, and for graphs with non-negative Bakry-\'Emery curvature. More precisely, we prove that the first Betti number can be upper bounded in terms of the combinatorial vertex degree.
As the Betti number estimate immediately implies an explicit polynomial volume growth bound of the homology group which is quasi-isometric to a certain cover of the graph, we hope that our methods might also be useful for tackling the open question whether graphs with non-negative Ollivier curvature must have polynomial volume growth.
One of the key problems regarding this question is to define and make use of a good notion of the dimension of a graph. The maximum combinatorial vertex degree is believed to be a good choice as a graph dimension, however, this quantity has not yet been used to give dimension dependent estimates analogous to the ones from Riemannian geometry.
In the case of Bakry-\'Emery curvature, there is an explicit notion of dimension coming from the Bochner formula. Making use of a finite dimension, it was shown in \cite{munch2019li} that non-negative Bakry-\'Emery curvature indeed implies polynomial volume growth.

Besides the Betti number estimates, we also prove various rigidity results, i.e., we characterize the graphs for which the Betti number bound is attained.
It turns out that these graphs are precisely the discrete tori, similar to the Riemannian setting.
In a broader scheme, this result is a novel contribution to various recent rigidity results for graphs with lower Ollivier Ricci curvature bounds, such as the rigidity result for the diameter \cite{chen2024diameter,cushing2020rigidity,
munch2022ollivier,Cushing2024ANO,Chen2024RicciCD,
Chen2024CurvatureAL}, the classification of low degree high girth Ricci flat graphs \cite{cushing2018ricci,
cushing2021curvatures,lin2013ricci,cushing2018erratum,
lin2024graphs}, and topological rigidity results for the number of ends \cite{hua2025every,hua2024graphs}.
A Cheeger Gromoll splitting type theorem for groups under a Ollivier type curvature for groups was established in \cite{bar2022conjugation,nguyen2023cheeger}.
We now summarize our main results.

\subsection{Main results}
As a preliminary result, we establish that the first Betti number of a graph vanishes if the graph has non-negative Ollivier curvature and contains a vertex for which every incident edge has positive Ollivier curvature.

\begin{theorem}\label{th:posCurvZeroBettoIntro}
    Let $G = (V, w, \mu)$ be a finite graph with non-negative Ollivier curvature. Suppose there exists a vertex $x \in V$ such that $\kappa(x,y) > 0$ for all $y \in \ON(x)$. Then, the first Betti number satisfies \( \beta_1(M_2(G)) = 0 \).
\end{theorem}

This theorem reappears as Theorem~\ref{th:posCurvZeroBetto}. Our first main result is the Betti number estimate for graphs with non-negative Ollivier curvature.
\begin{theorem}\label{th:MainResIntro}
	Let $G=(V,w,\mu)$ be a finite graph with non-negative Ollivier curvature. Then
	\begin{equation*}
		\beta_{1}(M_{2}(G)) \leq \frac{\deg_{\min}}{2}.
	\end{equation*}
\end{theorem}
This theorem reappears as Theorem~\ref{th:MainRes}.
We prove a similar result for non-negative Bakry-\'Emery curvature, see Theorem~\ref{th:BEBetti}, however with a worse constant.
The Betti number estimate for non-negative Ollivier curvature turns out to be sharp, and we characterize equality.

\begin{theorem}\label{th:RigidityIntro}
		Let $G=(V,w,\mu)$ be a finite graph. The following statements are equivalent.
		\begin{enumerate}[(i)]
			\item $G$ has non-negative Ollivier curvature and $\beta_{1}(M_{2}(G)) = \frac{\deg_{\max}}{2}$.
			\item $G$ is a discrete flat torus.
		\end{enumerate}
\end{theorem}
This theorem reappears as Theorem~\ref{th:CharacBS}.
For the definition of a discrete flat torus, see Definition~\ref{def:Torus}.
We call a graph Ollivier Betti sharp if satisfies $(i)$ of the above theorem.
In the unweighted case, a discrete torus corresponds to an abelian Cayley graph for which all cycles of length at most five are induced by commutators, see Corollary~\ref{cor:BettiSharpCombDist}.

We remark that for the rigidity result, we refer to the maximal and not the minimal vertex degree. The reason is that there is a variety of examples for which $\beta_1 = \deg_{\min}/2$ with quite different combinatorial structures which seem out of reach to characterize, see Section~\ref{sec:Examples}.

Furthermore, we show that these results extend to the broader setting of potentially non-reversible Markov chains; see Subsection~\ref{Sec:NonRevCase} for a detailed discussion. The generalizations of Theorem~\ref{th:posCurvZeroBettoIntro}, Theorem~\ref{th:MainResIntro}, and Theorem~\ref{th:RigidityIntro} to the non-reversible case can be found in Theorem~\ref{th:Betti_vanish_non_rev}, Theorem~\ref{th:Main_res_non_rev}, and Theorem~\ref{th:Rigidity_Non_Rev}, respectively.

We also investigate the rigidity cases depending on the idleness parameter of the Ollivier curvature. In particular, we characterize rigidity for the case of curvature zero for every choice of the idleness parameter which is called bone-idle in the literature, see \cite{bourne2018ollivier}.

\begin{theorem}
    Let $G=(V,w,\mu)$ be a finite graph with $\sum_y w(x,y) = \mu(x)$ for all vertices $x\in V$. Then, the following are equivalent. 
    \begin{itemize}
        \item[$(i)$] $G$ is Ollivier-Betti sharp and bone-idle.
        \item[$(ii)$] $G$ is Ollivier-Betti sharp and has non-negative Ollivier curvature without idleness.
        \item[$(iii)$] $G$ is a discrete flat Torus with constant measure and parallel edges having same length.
    \end{itemize}
\end{theorem}

This theorem reappears as Theorem~\ref{th:OBSandBI}. Additionally, we establish an upper bound on the first Betti number under a more general assumption, where non-negative Ollivier curvature is required only outside a finite subset.

\begin{theorem}
    Let $G=(V,w,\mu)$ be a finite graph and let $W \subseteq V$ be a non-empty set such that $\kappa(x,y) \geq 0$ if $x,y\in V\setminus W$. Then
    \[
        \beta_1(M_2(G)) \leq \vert \eW \vert,
    \]
    where $\eW = \{ x \sim y: x \in W\}$ denotes the set of edges with at least one endpoint in $W$.
\end{theorem}

This theorem reappears as Theorem~\ref{th:BettiNumEstSomeNegCurv}. Moreover, we give several examples.
For a potentially non-reversible Markov chain on a cycle of at least five vertices, we show the following:
\begin{itemize}
\item 
There exists a path metric on the cycle with constant Ollivier curvature.
\item This metric is unique up to scaling.
\item This metric has non-negative Ollivier curvature.
\item This metric has Ollivier curvature zero if and only if the graph is Ollivier-Betti sharp. 
\end{itemize}
These results can be found in Section~\ref{sec:CycleGraphs}.


\section{Setup and notation}
We first introduce weighted graphs equipped with a path metric and a method to assign a two-dimensional weighted CW complex to the graph such that taking universal cover preserves curvature.
Then, we discuss how to lift harmonic 1-forms to the universal cover and that the lifted one-forms are actually the gradient of harmonic functions.
Finally, we introduce the concept of Ollivier Ricci curvature.

\subsection{Weighted graphs and CW complexes}

A weighted graph $G=(V,w,\mu)$ consists of a countable vertex set $V$, a symmetric edge weight $w:V\times V\to [0,\infty)$ which is zero on the diagonal, and a vertex weight $\mu:V \to (0,\infty)$. We will write $x\sim y$ if $w(x,y) > 0$. A graph is called connected if for any $x,y \in V$, there exists a path $\{x_i\}_{i=0}^n \subseteq V$ connecting $x$ and $y$, i.e.,
\[
    x = x_0 \sim \ldots \sim x_n =y.
\] 
In the following, we always assume that $G$ is connected.
Let $\ON(x) = \{y\in V: x\sim y\}$ denote the open neighbourhood of a vertex $x\in V$, and denote by $\CN(x) = \ON(x)\cup\{x\}$ the closed neighbourhood.
We always assume that $G$ is locally finite, i.e., for all $x \in V$, the neighbourhood $\ON(x)$ is finite. The combinatorial degree of a vertex $x \in V$ is denoted by $\deg(x) = \vert \ON(x)\vert$. The minimum degree of a graph $G$ is denoted by $\deg_{\min} = \min_{x\in V}\deg(x)$ and the maximum degree is denoted by $\deg_{\max} = \max_{x\in V}\deg(x)$. A symmetric function $d:V\times V \to [0,\infty)$ with $d(x,y) = 0$ if and only if $x=y$ is called a general path distance if 
\begin{equation*}
    d(x,y) =\inf\left\{\sum_{k=1}^{n}d(x_k,x_{k-1}): x=x_0\sim x_1\sim \dots\sim x_n = y\right\},
\end{equation*} 
for all $x,y\in V$. For a graph $G=(V,w,\mu)$ equipped with a general path distance $d$, we also refer to the quadruple $(V,w,\mu,d)$ as a graph. The diameter of $G$ is denoted by $\diam(G) = \max_{x,y\in V}d(x,y)$.

Let $M_1(G)$ be the associated 1-dimensional CW complex. This CW complex comes with a set of cells $X = X_0 \dot\cup X_1 $, where $X_0=V$ and $X_1=\{\{x,y\} \subset V: x\sim y\}$, a linear coboundary operator $\delta: C(X) \to C(X)$, where $C(X)=\R^{X}$, mapping $C(X_0)$ to $C(X_1)$, and a positive weight function $m:X \to (0,\infty)$.
The weight function $m$ is given by the vertex weight $\mu$ in $X_{0}$ and the edge weight $w$ on $X_{1}$. The linear coboundary operator is obtained as follows: First, we assign a fixed orientation to each edge, i.e., for $e=\{x,y\}$, we choose a source vertex $s(e) \in e$ and a target vertex $t(e)$ such that $e=\{s(e),t(e)\}$. By abuse of notation, we write $x = 1_{x} \in C(X)$ for $x\in X$. Then, for $e \in X_{1}$, $x\in X_{0}$, we define
\begin{equation*}
    \delta x(e) = x(s(e)) - x(t(e)).
\end{equation*} 
Thus, the operator $\delta$ can be seen as a $\{\pm 1,0\}$-valued incidence matrix. Using the linearity of $\delta$, we naturally extend it to arbitrary functions $f\in C(X_{0})$ by
\begin{equation*}
    \delta f(e) = f(s(e)) - f(t(e)),
\end{equation*}
for $e\in X_{1}$. For $\alpha \in C(X_{1})$ and $x,y\in X_{0}$, we introduce the notation
\begin{equation*}
    \alpha(x,y) = \begin{cases}
        0&: \{x,y\} \notin X_1, \\
        \alpha(e)&: x=s(t),y=t(e) \mbox{ for some } e\in X_1, \\
        -\alpha(e)&: x=t(e), y=s(e) \mbox{ for some } e \in X_1.
        \end{cases}
\end{equation*}
Thus, we write $\delta f(x,y) = f(x) - f(y)$ whenever $e=\{x,y\} \in E$.

The weight function $m$ induces a scalar product on $C(X)$ by
\begin{equation*}
    \langle f,g\rangle = \sum_{x \in X} f(x)g(x)m(x) \quad \mbox{for }f,g \in C(X).
\end{equation*}
Via the scalar product, the coboundary operator $\delta$ has an adjoint $\delta^*: C(X) \to C(X)$. It is immediate to verify that
\begin{equation*}
    \delta^* e(x) = \frac{m(e)}{m(x)} \delta x(e) \mbox{ for } x\in X_{0}, e\in X_{1},
\end{equation*}
and by linear extension,
\[
\delta^* \alpha(x) = \sum_{y: x\sim y} \frac{w(x,y)}{\mu(x)}\alpha(x,y) \mbox{ for all } \alpha \in \R^{X_1}, x \in X_0.
\]
Restricting to $C(X_{1})$ gives $\delta^*: C(X_{1}) \to C(X_{0})$.

Next, we obtain the standard (positive) graph Laplacian as $L:=\delta^*\delta$, given by
\[
\delta^* \delta f(x) =  \sum_{y: x\sim y} \frac{w(x,y)}{\mu(x)}(f(x)-f(y)) \mbox{ for all } f \in \R^{X_0}, x \in X_0,
\]
and denote by $\Delta = -L$ the negative graph Laplacian. We call a function $f\in \R^{X_0}$ harmonic if $\Delta f =0$. 

A cycle is an injective path $(x_{0} \sim \ldots \sim x_{n-1})$ of vertices $x_{i} \in X_{0}$ with $x_{0} \sim x_{n-1}$ and $n \geq 3$. We identify two cycles $(x_{0} \sim \ldots \sim x_{n-1})$ and $(y_{0} \sim \ldots \sim y_{n-1})$ if there exists a $k \in \N$ such that $x_{i} = y_{i\pm k \mod n}$ for all $i\in \{0, \ldots, n-1\}$ and a fixed choice of plus or minus.

Denote by $X_2$ the set of all cycles $(x_{0}\sim x_1\sim\dots\sim x_{n-1})$ with $n=3$ or with
\begin{equation*}
    d(x_0,x_3) < d(x_0,x_1) + d(x_1,x_2) + d(x_2,x_3),
\end{equation*} 
and $x_3 \sim \dots\sim x_{n-1}\sim x_0$ is a geodesic. We remark that $X_2$ always contains all $3$- and $4$-cycles, independent of the choice of the path metric.

\begin{proposition}
    Let $G=(V,w,\mu,d)$ be a locally finite graph and $(x_0 \sim \dots \sim x_{n-1})$ be a cycle in $G$. If $n \in \{3,4\}$, then $(x_0 \sim \dots \sim x_{n-1}) \in X_2$.
\end{proposition}

\begin{proof}
    If $n=3$, the claim follows directly from the definition of $X_2$. Next, we assume $n=4$. If 
    \[
        d(x_0, x_3) < d(x_0,x_1) + d(x_1,x_2) + d(x_2,x_3),
    \]
    then the cycle $(x_0 \sim \dots \sim x_{3}) \in X_2$. Hence, assume 
    \[
        d(x_0,x_3) \geq d(x_0, x_1) + d(x_1,x_2) + d(x_2,x_3).
    \]
    We obtain $d(x_1,x_2) < d(x_0,x_3) < d(x_1, x_0) + d(x_0, x_3) + d(x_3, x_2)$, and therefore $(x_0 \sim \dots \sim x_{3}) \in X_2$. 
\end{proof}

\begin{remark}
    Observe that 5-cycles are not necessarily contained in $X_2$. Given any 5-cycle, we can always label its vertices as $(x_0\sim x_1 \sim x_2 \sim x_3 \sim x_4)$ such that the inequality
    \[
        d(x_0,x_3) < d(x_0,x_1) + d(x_1,x_2) + d(x_2,x_3)
    \]
    holds.  However, this does not ensure that the path $x_0 \sim x_4 \sim x_3$ is a geodesic. As a counterexample, consider a 5-cycle $(x_0\sim x_1 \sim x_2 \sim x_3 \sim x_4)$ with additional edges 
    $x_0\sim x_2$ and $x_1 \sim x_3$, where the distance function is defined by $d(x_i, x_{i+1 \mod 5}) = 1$ and $d(x_0, x_2) = d(x_1,x_3) = 1/2$.

    Furthermore, we note that cycles of arbitrary length can be contained in $X_2$, depending on the specific path distance.
\end{remark}

Set $m(z) = 1$ for all $z\in X_2$ and extend the coboundary operator $\delta$ mapping $C(X_1) \to C(X_2)$, given by 
\begin{equation*}
    \delta \alpha((x_{0}\sim \ldots \sim x_{n-1})) = \sum_{i=0}^{n-1} \alpha(x_{i}, x_{i+1\mod n})
\end{equation*}
for $\alpha \in C(X_{1})$ and $(x_{0}\sim \ldots \sim x_{m-1}) \in X_{2}$. Then, $M_{2}(G) = (X \cup X_{2}, \delta, m)$ is a 2-dimensional CW complex.

Let $H_1(M_2(G)) := \{\alpha \in C(X_1):\delta \alpha = 0, \delta^* \alpha = 0\}$ be the first homology group and let $\beta_1(M_2(G)) := \dim H_1(M_2(G)) $ be the first Betti number of $M_2(G)$.

For a regular CW complex $M$ with cell set $X$ and coboundary operator $\delta$, a covering space of $M$ is a regular CW complex $\widetilde M$ with cell set $\widetilde X$ and coboundary operator $\widetilde \delta$, along with a surjective map $\pr: \widetilde X \to X$ satisfying the following compatibility condition:
\begin{equation*}
    \forall f \in C(X), \quad (\delta f)\circ \pr = \widetilde\delta (f\circ \pr).
\end{equation*}
A covering space $\widetilde M$ is called a universal cover of $M$ if every covering space of $\widetilde M$ is isomorphic to $\widetilde M$. This is equivalent to $\widetilde M$ being a simply-connected cover. A space is called simply-connected if it is path-connected and has trivial fundamental group.

\subsection{Lifting one-forms to the universal cover}

As Ollivier and Bakry-\'Emery curvature deal with functions, and the first Betti-number deals with 1-forms, we convert  1-forms $\alpha \in \R^{X_1}$ with $\delta \alpha = 0$  to functions by lifting $\alpha$ to the universal cover. As the universal cover is simply connected, we can use the Hodge decomposition to deduce that the lifted one form can be written as $\widetilde \delta f_\alpha$ for some function $f_\alpha$ on the universal cover. More precisely, we define $f_\alpha \in \R^{\widetilde X_0}$ via $\widetilde \delta f_\alpha = \alpha \circ \pr$ where $\widetilde \delta$ and $\widetilde X_k$ are the coboundary and the set of $k$ cells respectively on the universal cover.
The function $f_\alpha$ is uniquely determined by $\alpha$ up to an additive constant.
More explicitly, given $\alpha \in \R^{X_1}$ with $\delta \alpha = 0$, the lifted one-form $\alpha \circ \pr$ also satisfies $\widetilde \delta(\alpha \circ \pr) = 0$. Given a base vertex $x_0 \in \widetilde X_0$, we can define
\[
f_\alpha(x) = \sum_{k=1}^n (\alpha\circ \pr)(x_k,x_{k-1})
\]
for an arbitrary path $x_0 \sim \ldots \sim x_n = x$ on the universal cover. This is independent of the choice of the path as $\widetilde \delta(\alpha \circ \pr) = 0$. Moreover, by construction, we have $\widetilde \delta f_\alpha = \alpha \circ \pr$ as desired.

Another useful property is that if additionally $\delta^* \alpha =0$, then $Lf_\alpha = 0$. Hence, we can convert every $\alpha \in H_1(M_2(G))$ to a harmonic function $f_\alpha$ on the universal cover. 

We equip the 1-skeleton of the universal cover with the path distance $\widetilde{d}: \widetilde{X_0}\times\widetilde{X_0} \to [0,\infty)$ . For adjacent vertices $\widetilde{x} \sim \widetilde{y}$, the distance is defined as $\widetilde{d}(\widetilde{x}, \widetilde{y}) = d(\pr(\widetilde{x}), \pr(\widetilde{y}))$. For arbitrary $\widetilde{x}\not= \widetilde{y}$, we define
\begin{equation*}
    \widetilde{d}(\widetilde{x},\widetilde{y}) =\inf\left\{\sum_{k=1}^{n}d(\widetilde{x}_k,\widetilde{x}_{k-1}): \widetilde{x}=\widetilde{x}_0\sim \widetilde{x}_1\sim \dots\sim \widetilde{x}_n=  \widetilde{y}\right\}.
\end{equation*}

The following lemma establishes an initial property of the universal cover: If the image of a path in the universal cover forms a cycle that corresponds to a 2-cell in the original graph, then the path in the universal cover must itself be a cycle.

\begin{lemma}\label{lem:x0eqxn}
    Let $G=(V,w,\mu,d)$ be a finite graph. Let $\widetilde{x}_0 \sim \dots \sim \widetilde{x}_n$ be a path in the universal cover of $G$, and let $x_k = \pr(\widetilde{x}_k)$ for $k=0,\dots,n$. Suppose $x_0 = x_n$. If $n=2$, or if $x_3 \sim \dots \sim x_n$ is a geodesic from $x_0$ to $x_3$ and 
    \[
        d(x_0,x_3) < d(x_0,x_1) + d(x_1,x_2) + d(x_2,x_3),
    \]
    then $\widetilde{x}_0 = \widetilde{x}_n$ must hold.
\end{lemma}

\begin{proof}
    Assume, for the sake of contradiction, that the Lemma does not hold. Let $\widetilde{x}_0 \sim \dots \sim \widetilde{x}_n$ be a path that violates the lemma with $n$ minimal. 
    Hence, the lemma is true for all paths of length less than $n$.        
    Let $x_k = \pr(\widetilde{x}_k)$ for $k=0,\dots,n$.  
    
    We first claim that       
    \[
        \{x_0,x_1,x_2,x_3\} \cap \{x_4, \dots, x_{n-1}\} \not=\emptyset.
    \]  
    Assume not.
    Then, there exists $i\in \{0,\dots,3\}$ and $j\in \{4,\dots,n-1\}$ such that $x_i = x_j$.
    
    Then, we can apply the lemma for the shorter path $\widetilde x_i \sim \ldots \sim \widetilde x_j$ giving $\widetilde x_i = \widetilde x_j$. Here, we used that 
    \begin{align*}
    d(x_i,x_{i+3}) <d(x_i,x_3)\leq d(x_{i},x_{i+1}) + d(x_{i+1},x_{i+2}) + d(x_{i+2},x_{i+3})
    \end{align*}
    if $j\geq i+3$, as 
    $x_{i+3}$ lies on the geodesic  $x_3\sim \ldots\sim x_j$.
    
    Similar to before, we apply the lemma again for the shorter path $\widetilde x_1 \ldots \sim \widetilde x_i = \widetilde x_j \sim \ldots \sim \widetilde x_n$ giving $\widetilde x_n = \widetilde x_0$. Here, we used that $x_j \sim \ldots \sim x_n$ is a geodesic and therefore, the assumption of the lemma is satisfied.   This proves our claim 
        \[
            \{x_0,x_1,x_2,x_3\} \cap \{x_4, \dots, x_{n-1}\} =\emptyset,
        \]
        and thus, $(x_0 \sim \dots \sim x_{n-1}) \in X_2$ if $n>2$. 
    As the lemma is trivial in the case $n=2$,   
        this shows that $\widetilde{x}_0 = \widetilde{x}_n$ must hold. Now the proof is finished.
\end{proof}


The next lemma shows that the distances on $\widetilde X$ and $X$ coincide, assuming that the vertices are at most three hops away from each other.

\begin{lemma}\label{lem:Propertiesd}
    Let $G=(V,w,\mu,d)$ be a finite graph and $x_0\sim \ldots \sim x_k$ be a path in $G$ with $k \leq 3$ where all vertices are distinct. Let $\widetilde{x}_i\in \pr^{-1}(x_i)$ such that $\widetilde{x}_{i}\sim \widetilde{x}_{i+1}$ for $i=0,\dots,k-1$. Then, $d(x_0,x_k) = \widetilde{d}(\widetilde{x}_0,\widetilde{x}_k)$.
\end{lemma}

\begin{proof}
    Let  $\widetilde{x}_0 = \widetilde{y}_0 \sim \widetilde{y}_1 \sim \dots \sim \widetilde{y}_n = \widetilde{x}_k$ be a geodesic from $\widetilde{x}_0$ to $\widetilde{x}_k$. Then $x_0 = \pr(\widetilde{y}_0) \sim \pr(\widetilde{y}_1) \sim \dots \sim \pr(\widetilde{y}_n) = x_k$, and therefore
    \[
        d(x_0,x_k) \leq \sum_{i=1}^{n}d(\pr(\widetilde{y}_i),\pr(\widetilde{y}_{i-1}))=\sum_{i=1}^{n}\widetilde{d}(\widetilde{y}_i,\widetilde{y}_{i-1}) = \widetilde{d}(\widetilde{x}_0,\widetilde{x}_k).
    \]
    To prove the other inequality, we first assume that $d(x_0,x_k) = d(x_0,x_1) + \ldots +d(x_{k-1},x_{k})$. Then,
    \[
        \widetilde{d}(\widetilde{x}_0,\widetilde{x}_k) \leq \widetilde{d}(\widetilde{x}_0,\widetilde{x}_1) + \ldots + \widetilde{d}(\widetilde{x}_{k-1},\widetilde{x}_k) = d(x_0,x_1) + \ldots + d(x_{k-1},x_k) = d(x_0,x_k).
    \]
    Otherwise, let $x_k = y_0 \sim y_1 \sim \dots\sim y_n = x_0$ be a geodesic from $x_0$ to $x_k$. Let $\widetilde{y}_0 = \widetilde{x}_k$ and choose $\widetilde{y}_{i} \in \pr^{-1}(y_i)$ with $\widetilde{y}_{i} \sim \widetilde{y}_{i-1}$ for $i=1, \dots, n$.
    
    As $d(x_0,x_k) < d(x_0,x_1) + \ldots +d(x_{k-1},x_{k})$,    
    we can apply Lemma~\ref{lem:x0eqxn} to the path $\widetilde x_0\sim \ldots \sim \widetilde x_k = \widetilde y_0\sim\ldots \sim \widetilde y_n$, giving $\widetilde{x}_0 = \widetilde{y}_n$, and therefore,
    \[
        \widetilde{d}(\widetilde{x}_0,\widetilde{x}_k) \leq \sum_{i=1}^{n} \widetilde{d}(\widetilde{y}_i, \widetilde{y}_{i-1}) = \sum_{i=1}^{n} d(y_i,y_{i-1}) = d(x_0,x_k).
    \]
    This concludes the proof.
\end{proof}

As $\alpha$ is bounded due to finiteness of $G$, we additionally get that $f_\alpha$ is Lipschitz and that the maximal gradient is attained somewhere.
This will turn out to be very convenient, as, together with the non-negative curvature bound, we can show that the maximal gradient of $f_\alpha$ is attained everywhere.

\subsection{Ollivier curvature}

Ollivier introduced the Ollivier curvature, a discrete analog of Ricci curvature formulated for discrete time random walks \cite{ollivier2007ricci, ollivier2009ricci}.  By taking a limit of lazy random walks, Lin, Lu, and Yau modified Ollivier's notion of Ricci-curvature \cite{lin2011ricci}. In \cite{bourne2018ollivier}, the connection between the laziness of random walks and the Ollivier curvature was studied. In this work, we use a generalization of the Ollivier curvature, introduced in \cite{munch2017ollivier}, which is applicable to all weighted graph Laplacians. To this end, let $G=(V,w,\mu,d)$ be a weighted graph. We define the gradient 
\begin{equation*}
    \nabla_{xy} f = \frac{f(x) -f(y)}{d(x,y)}
\end{equation*}
for $x\not=y \in V$ and $f\in C(V)$, and the associated Lipschitz constant
\begin{equation*}
    \lVert \nabla f \rVert_{\infty} = \sup_{x\not= y} \vert \nabla_{xy} f \vert = \sup_{x \sim y} \vert \nabla_{xy} f \vert.
\end{equation*}
For $K \geq 0$, we denote the set of $K$-Lipschitz functions by 
\begin{equation*}
    Lip(K) = \{f \in C(X_{0}): \lVert \nabla f\rVert _{\infty} \leq K\}.
\end{equation*}
According to \cite{munch2017ollivier}, we define the Ollivier curvature $\kappa(x,y)$ for $x\not=y\in V$ by 
\begin{equation*}
    \kappa(x,y) = \inf_{\substack{\nabla_{yx}f = 1\\ \lVert \nabla f \rVert_\infty = 1}} \nabla_{xy} \Delta f.
\end{equation*}
In \cite[Theorem 2.1]{munch2017ollivier}, the authors show that this definition coincides with the curvature notion introduced by Lin, Lu, and Yau whenever the latter definition applies. The curvature can also be obtained via transport plans, a feature that will play an important role in the following.

\begin{lemma}[{{\cite[Proposition~2.4]{munch2017ollivier}}}]\label{Prop:TransPlan}
    Let $G=(V,m,\mu,d)$ be a graph and let $x_{0}\not=y_{0}$ be vertices. Then,
    \begin{align}
        \kappa(x_0,y_0) &= \sup_{\rho}  \sum_{\substack{x \in \CN(x_0) \\ y \in \CN(y_0)}}\rho(x,y) \left[1 - \frac{d(x,y)}{d(x_0,y_0)}\right] \label{eq:OTP} 
    \end{align}
    where the supremum is taken over all $\rho: \CN(x_0) \times \CN(y_0) \to [0,\infty)$ such that
    \begin{align*}
    \sum_{y \in \CN(y_0)} \rho(x,y) &= \frac{w(x_0, x)}{\mu(x_0)}  \qquad \mbox{ for all } x \in \ON(x_0) \mbox{ and}\\
    \sum_{x \in \CN(x_0)} \rho(x,y) &= \frac{w(y_0, y)}{\mu(y_0)} \qquad \mbox{ for all } y \in \ON(y_0).
    \end{align*}
\end{lemma}
Observe that $\rho$ is defined on balls, but the coupling property is only required on spheres. Moreover, we do not assume anything concerning $\sum_{x,y}\rho(x,y)$. Nevertheless, we call $\rho$ an optimal transport plan if the supremum in (\ref{eq:OTP}) is attained. Due to compactness, there always exists an optimal transport plan.

Interestingly, when maximizing over the weights of the 2-cells on $M_2(G)$, this definition also coincides with the Forman curvature of 
\[
F(e) = \square_1 (e,e) - \sum_{e'} |\square_1(e,e')|
\]
where $\square_1 = (\delta \delta^* + \delta^* \delta) |_{\R^{X_1}}$ is the Hodge Laplacian on one-forms and $\square_1(e,e') = \square_1 1_{e'} (e)$ is the corresponding matrix entry when interpreting the Hodge Laplacian as a matrix \cite[Theorem~1.2]{jost2021characterizations}.
It is important to note that the weights on $X_2$ can generally \emph{not} be chosen such that the optimal lower bounds of Forman and Ollivier curvature coincide \cite[Section~6.3]{jost2021characterizations}, as for different edges, the Forman curvature maximizing weights on $X_2$ can be different. 
However, the Ollivier curvature can always be lower bounded by the Forman curvature.
Hence, non-negative Ollivier curvature is a strictly weaker assumption than non-negative Forman curvature so that Forman's results cannot be applied for deriving Betti number estimates for graphs with non-negative Ollivier curvature.

\section{Betti number estimates for non-negative Ricci curvature}

We first discuss that the first Betti number vanishes in case of positive Ricci curvature.
Afterwards, we investigate the non-negative curvature case. The key idea to obtain the Betti number estimate is to lift harmonic one-forms to the universal cover where they can be written as the gradient of a harmonic function.
This harmonic function attains its Lipschitz constant, and by non-negative curvature, we will show that the Lipschitz constant is attained at every vertex.
This implies that the harmonic 1-form must take
its supremum norm at every vertex. Hence, if a harmonic 1-form vanishes around a vertex, then it vanishes everywhere. This will lead to the Betti number estimate in terms of the minimal vertex degree.

\subsection{Improved transport plans and Lipschitz harmonic functions}
A crucial inside for the study of non-negative Ollivier curvature is that an  optimal transport plan can always be chosen in such a way that with positive probability the distance of the coupled random walks gets strictly smaller. This was established in \cite{bordewich2007path,dyer1998more} under the name ``path coupling method'' in order to prove mixing time estimates, and was used in \cite{munch2023non} to prove Buser inequalities and reverse Poincare estimates.

We now give a dual formulation of the existence of such an improved transport plan in terms of Lipschitz functions.

\begin{lemma}\label{lem:LipFun}
    Let $G=(V,w,\mu,d)$ be a connected locally finite graph with non-negative Ollivier curvature. Let $f\in Lip(1)$ satisfy $\Delta f(x) =c$ for every $x\in V$ and some constant $c\in \R$. Let $x_0\not=y_0 \in V$ be two distinct vertices such that $f(x_0) - f(y_0) = d(x_0,y_0)$. Then,
    \begin{itemize}
        \item[$(i)$] $\exists \widetilde{x}\in \CN(x_0)$ and $\exists  \widetilde{y}\in \CN(y_0)$ s.t. $f(\widetilde{x}) - f(\widetilde{y}) = d(\widetilde{x},\widetilde{y})> d(x_0,y_0)$,
        \item[$(ii)$] $\forall x\in \ON(x_0)$ $\exists y \in \CN(y_0)$ s.t. $f(x) - f(y) = d(x,y)$.
    \end{itemize}
\end{lemma}

\begin{proof}
    Let $\rho_0$ be an optimal transport plan and let $x' \in \ON(x)$ lie on a geodesic from $x_0$ to $y_0$. Namely, we have $d(x_0,y_0) = d(x_0,x') + d(x',y_0)$. Define a new transport plan $\rho: \CN(x_0)\times \CN(y_0) \to [0,\infty)$ by
    \begin{equation*}
        \rho(x,y) = \begin{cases}
            \rho_0(x,y_0) + \sum_{\substack{z \in \CN(y_0)\\d(x',z) \geq d(x_0,y_0)}} \rho_0(x',z)
            &: x=x',y=y_0,\\
            0&: x=x', d(x',y)\geq d(x_0,y_0) ,\\
            \rho_0(x_0,y) + \rho_0(x',y) &: x=x_0, d(x',y) \geq d(x_0,y_0),\\
            \rho_0(x,y)&: \mbox{otherwise}.
            \end{cases}
    \end{equation*}
    We aim to show that $\rho$ is also an optimal transport plan. First, we show that $\rho$ is indeed a transport plan. For $x=x'$, we have 
    \begin{align*}
        \sum_{y\in \CN(y_0)}\rho(x',y) &= \rho(x',y_0) + \sum_{d(x',y) \geq d(x_0,y_0)} \rho(x',y) + \sum_{\substack{d(x',y)<d(x_0,y_0)\\y\not=y_0}}\rho(x',y) \\
        &= \left(\rho_0(x',y_0) +  \sum_{d(x',y) \geq d(x_0,y_0)} \rho_0(x',y)\right) + \sum_{\substack{d(x',y)<d(x_0,y_0)\\y\not=y_0}}\rho_0(x',y) \\
        &= \sum_{y\in \CN(y_0)}\rho_0(x',y)
        = \frac{w(x_0,x')}{\mu(x_0)}.
    \end{align*}
    Observe that for $x \in \ON(x_0)\setminus \{x'\}$, we have
    $\rho(x,y) = \rho_0(x,y)$ for every $y \in \CN(y_0)$, and therefore,
    \begin{equation*}
        \sum_{y\in \CN(y_0)} \rho(x,y) = \sum_{y\in \CN(y_0)} \rho_0(x,y) = \frac{w(x_0,x)}{\mu(x_0)}.
    \end{equation*}
    On the other hand, for $y \in \ON(y_0)$ satisfying $d(x',y) < d(x_0,y_0)$, we have $\rho(x,y) = \rho_0(x,y)$ for all $x \in \CN(x_0)$. Hence, 
    \begin{equation*}
        \sum_{x\in \CN(x_0)} \rho(x,y) = \sum_{x\in \CN(x_0)} \rho_0(x,y) = \frac{w(y_0,y)}{\mu(y_0)}.
    \end{equation*}
    Finally, for $y \in \ON(y_0)$ s.t. $d(x',y) \geq d(x_0,y_0)$, we obtain
    \begin{align*}
    \sum_{x\in \CN(x_0)}\rho(x,y) &= \rho(x_0,y)+ \rho(x',y) + \sum_{x \in \ON(x_0)\setminus \{x'\}} \rho(x,y) \\
    &=\Big(\rho_0(x_0,y) + \rho_0(x',y) \Big) + 0 + \sum_{x \in \ON(x_0)\setminus \{x'\}} \rho_0(x,y) \\
    &=\sum_{x\in \CN(x_0)}\rho_0(x,y) = \frac{w(y_0,y)}{\mu(y_0)}.
    \end{align*}
    Hence, $\rho$ satisfies the marginal conditions and is a transport plan. It remains to show that $\rho$ is optimal. Since $\rho_0$ is optimal, it suffices to prove that
    \begin{equation*}
        \sum_{\substack{x\in \CN(x_0)\\y \in \CN(y_0)}} \rho(x,y)\Big(d(x_0,y_0) - d(x,y) \Big) \geq \sum_{\substack{x\in \CN(x_0)\\y \in \CN(y_0)}} \rho_0(x,y)\Big(d(x_0,y_0) - d(x,y) \Big).
    \end{equation*}
    Define $C(x,y):= \Big(\rho(x,y)-\rho_0(x,y)\Big)\Big(d(x_0,y_0) - d(x,y) \Big)$.
    Using this definition, it suffices to prove
    \[
    \sum_{\substack{x\in \CN(x_0)\\y \in \CN(y_0)}} C(x,y) \geq 0.
    \]
    First, observe that $C(x,y)=0$ whenever $\rho(x,y)=\rho_0(x,y)$. Hence, we obtain
    \begin{align*}
    \sum_{\substack{x\in \CN(x_0)\\y \in \CN(y_0)}} C(x,y) = C(x',y_0) + \sum_{d(x',y) \geq d(x_0,y_0)} \Big( C(x',y) + C(x_0,y) \Big).
    \end{align*}
    Next, observe that $\rho(x',y)=0$ if $d(x',y) \geq d(x_0,y_0)$. Therefore, we have by definition of $\rho$
    \[
    C(x',y_0) = \left(\sum_{d(x',y) \geq d(x_0,y_0)} \rho_0(x',y)\right) \cdot \Big(d(x_0,y_0)-d(x',y_0) \Big) 
    \]
    If $d(x',y) \geq d(x_0,y_0)$, we have
    \[
    C(x',y) = -\rho_0(x',y) \Big(d(x_0,y_0)-d(x',y) \Big)
    \]
    and
    \[
    C(x_0,y)= \rho_0(x',y) \Big(d(x_0,y_0) - d(x_0,y)\Big).
    \]
    Hence, we obtain
    \[
    C(x',y) + C(x_0,y) = \rho_0(x',y)\Big(d(x',y) - d(x_0,y)\Big) \geq - d(x_0,x')\rho_0(x',y)
    \]
    since $|d(x',y) - d(x_0,y)| \leq d(x_0,x')$.
    Putting together gives
    \begin{align*}
    \sum_{\substack{x\in \CN(x_0)\\y \in \CN(y_0)}} C(x,y) &= C(x',y_0) + \sum_{d(x',y) \geq d(x_0,y_0)} \Big( C(x',y) + C(x_0,y) \Big) \\
    &\geq \sum_{d(x',y) \geq d(x_0,y_0)} \rho_0(x',y)\left(d(x_0,y_0) - d(x',y_0) - d(x_0,x')\right)= 0,
    \end{align*}
    where we used that $x'$ lies on a geodesic from $x_0$ to $y_0$.
    This proves that $\rho$ is an optimal transport plan.
    Observe that $d(x',y) < d(x_0,y_0)$ for all $y\in \CN(y_0)$ with $\rho(x',y) >0$. Using that $\sum_{y\in \CN(y_0)}\rho(x',y) = \frac{w(\{x_0, x'\})}{\mu(x_0)}> 0$, we conclude that $\exists y'\in \CN(y_0)$ such that $\rho(x',y') > 0$, and thus, $d(x',y') < d(x_0,y_0)$.
    Next, by non-negative Ollivier curvature, we have
    \begin{align*}
        0 &\leq d(x_0,y_0) \kappa(x_0,y_0)\\
        &= \sum_{\substack{x \in \CN(x_0) \\ y \in \CN(y_0)}} \rho(x,y)(d(x_0,y_0) - d(x,y)) \\
        &\leq \sum_{\substack{x \in \CN(x_0) \\ y \in \CN(y_0)}} \rho(x,y)(f(x_0)-f(y_0) - (f(x)-f(y)))\\
        &=\sum_{x \in \CN(x_0)}(f(x_0)- f(x)) \sum_{y \in \CN(y_0)} \rho(x,y) - \sum_{y \in \CN(x_0)}(f(y_0) -f(y)) \sum_{x\in \CN(x_0)} \rho(x,y)\\
        &= \sum_{x \in \ON(x_0)}(f(x_0)- f(x))\frac{w(\{x_0,x\})}{\mu(x_0)} - \sum_{y \in \ON(x_0)}(f(y_0) -f(y)) \frac{w({y_0,y})}{\mu(y_{0})}\\
        &= \Delta f(y_0) -\Delta f(x_0) = 0.
    \end{align*}
    Thus, we obtain
    \begin{equation*}
        0 = \sum_{\substack{x \in \CN(x_0) \\ y \in \CN(y_0)}} \rho(x,y)(d(x_0,y_0) - d(x,y)),
    \end{equation*}
    and using $d(x',y') < d(x_0,y_0)$ and $\rho(x',y') > 0$, we conclude that $\exists \widetilde{x}\in \CN(x_0)$ and $\exists \widetilde{y}\in \CN(y_0)$ such that $\rho(\widetilde{x},\widetilde{y}) >0$ and $d(x_0,y_0) < d(\widetilde{x},\widetilde{y})$.
    Furthermore, observe that $f(x)-f(y) = d(x,y)$ must hold for every $x\in \CN(x_0), y\in \CN(y_0)$, whenever $\rho(x,y) > 0$.
    Therefore, we obtain 
    \begin{equation*}
        f(\widetilde{x}) -f(\widetilde{y}) = d(\widetilde{x},\widetilde{y}) > d(x_0,y_0),
    \end{equation*} and thus, we have proven $(i)$.
    For $(ii)$, observe that for every $x \in \ON(x_0)$, there exists $y\in \CN(y_0)$ such that $\rho(x,y) > 0$, and therefore $f(x) -f(y) = d(x,y)$ must hold.
\end{proof}

\begin{remark}
In the case of the combinatorial graph distance, property $(i)$ can also be derived from \cite[Lemma 2.3]{jost2019Liouville} by taking the limit $\eps \to 0$.
\end{remark}

\subsection{Harmonic one-forms take supremum everywhere -- twice}
By using the properties of harmonic Lipschitz functions on the universal cover established in the previous section, we show that for every vertex, harmonic one-forms attain their supremum norm at two edges containing this vertex.
To this end, for $\alpha \in C(X_1)$, we define the supremum norm
\begin{equation*}
    \lVert \alpha \rVert_\infty = \sup_{x\sim y}\frac{\alpha(x,y)}{d(x,y)}.
\end{equation*}
Using this definition, we obtain the following statement.

\begin{lemma}\label{lem:AlphaEins}
    Let $G=(V,w,\mu,d)$ be a finite graph with non-negative Ollivier curvature. Let $\alpha \in H_{1}(M_{2}(G))$ with $\lVert \alpha \rVert_{\infty} = 1$. For every $x \in X_{0}$ exist $y,z \in X_{0}$ such that $x\sim y, x\sim z$ and 
    \begin{equation*}
        \frac{\alpha(x,y)}{d(x,y)} = \frac{\alpha(z,x)}{d(z,x)} = 1.
    \end{equation*}     
\end{lemma}

\begin{proof}
    We assume without loss of generality that $d(x,y) \geq 1$ for $x\not=y$. Let $\widetilde{M_{2}(G)}$ be the universal cover of $M_{2}(G)$ with covering map $\pr: \widetilde{M_{2}(G)} \to M_{2}(G)$, coboundary operator $\widetilde{\delta}$ and sets of k-cells $\widetilde{X}_{k}$. Then, as $\delta \alpha =0$, there exists $f_{\alpha} \in C(\widetilde{X_{0}})$ with $\widetilde{\delta} f_\alpha = \alpha \circ \pr$. Since $\alpha \in H_{1}(M_{2}(G))$, it follows that $f_{\alpha}$ is a harmonic function. Furthermore, $f_{\alpha}$ is 1-Lipschitz as $\lVert \alpha \rVert_\infty = 1$. 
    Therefore, we can apply Lemma~\ref{lem:LipFun} $(i)$, which implies the existence of $\widetilde{y_0}, \widetilde{z_0} \in \widetilde{X_{0}}$ with \begin{equation*}
        f_\alpha(\widetilde{y_0}) -f_\alpha(\widetilde{z_0}) = \widetilde{d}(\widetilde{y_0}, \widetilde{z_0}) \geq 3m\cdot \diam(G),
    \end{equation*}
    where $m = \max_{x\sim y}d(x,y)$.

    Now, let $x\in X_0$ be arbitrary. As $\widetilde{M_{2}(G)}$ is a covering space, there exists $\widetilde{x} \in \widetilde{X_0}$ such that $\pr(\widetilde{x}) = x$ and $\widetilde{d}(\widetilde{x}, \widetilde{y_0}) \leq \diam(G)$. Therefore, we can find a geodesic $ \widetilde{y_0} \sim \widetilde{y_1} \sim \dots \sim \widetilde{y_n}=\widetilde{x}$ with
    \begin{equation*}
        \sum_{k=1}^{n}\widetilde{d}(\widetilde{y}_k, \widetilde{y}_{k-1}) \leq \diam(G).
    \end{equation*}

    By inductively applying Lemma~\ref{lem:LipFun} $(ii)$, we obtain a sequence of vertices $\widetilde{z_i} \in \CN(\widetilde{z_{i-1}})$, for $i = 1,\dots, n$ such that
    \begin{equation*}
        f_\alpha(\widetilde{y_l}) -f_\alpha(\widetilde{z_l}) = \widetilde{d}(\widetilde{y_l}, \widetilde{z_l})\geq \widetilde{d}(\widetilde{y_0}, \widetilde{z_0}) - 2lm.
    \end{equation*}
    Using $n \leq \diam(G)$ and $\widetilde{d}(\widetilde{y_0}, \widetilde{z_0}) \geq 3m\cdot\diam(G)$, we conclude
    \begin{equation*}
        f_\alpha(\widetilde{x}) -f_\alpha(\widetilde{z_n}) = \widetilde{d}(\widetilde{x}, \widetilde{z_n}) > 0.
    \end{equation*}
    As $f_{\alpha}$ is 1-Lipschitz, there must exist $\widetilde{y} \in \widetilde{X_0}$ such that $\widetilde{x} \sim \widetilde{y}$ and
    \begin{equation*}
        f_{\alpha}(\widetilde{x}) - f_{\alpha}(\widetilde{y}) = \widetilde{d}(\widetilde{x}, \widetilde{y}) = d(x,y).
    \end{equation*}
    By definition of $f_\alpha$, we conclude that $\alpha(x,y) = d(x,y)$, where $y = \pr(\widetilde{y})$. 
    Applying the same arguments to $-\alpha$ ensures the existence of a $z\in X_0$ with $x\sim z$ such that
    \begin{equation*}
        -\alpha(x,z) = \alpha(z,x) = d(x,z).
    \end{equation*}
    This concludes the proof.
\end{proof}

Next, we require a fundamental lemma from linear algebra, which states that if each vector in a subspace of $\R^n$ attains its supremum norm in at least $k$ entries, then the dimension of this subspace can be bounded above by $n/k$. Combining this result with the previous lemma will allow us to derive our Betti number estimate under non-negative Ollivier curvature.

\begin{lemma}\label{lem:LinAlg}
    Let $E \subseteq \R^n$ be a subspace such that for every $v \in E$, there exist indices $i_{1}< \dots< i_{k}$ satisfying $\vert v_{i_{j}} \vert = \lVert v \rVert_{\infty}$ for $j =1, \dots, k$. Then, $\dim(E) \leq \frac{n}{k}$.
\end{lemma}

\begin{proof}
    For simplicity, we write $d = \dim(E)$. Choose a vector $v^{1}\in E$ such that there exist exactly $k$ indices $i_{1} < \dots < i_{k}$ with $\vert v^{1}_{i_{j}} \vert = \lVert v^{1} \rVert_{\infty}$ for $j=1,\dots,k$. Without loss of generality, we assume that $i_{j} = j$ for $j=1,\dots,k$. By the basis extension theorem, there exists a basis $(v^{1}, \dots, v^{d})$ of $E$ containing $v^{1}$. We may assume that $v^{j}_{1} = 0$ for $j \not=1$, because if this was not the case, we could replace $v^{j}$ by $w^{j}=v^{j} - \frac{v^{j}_{1}}{v^{1}_{1}} v^{1}$. Then, $(v^{1}, \dots, v^{j-1}, w, v^{j+1}, \dots, v^{d})$ would still form a basis by the Steinitz exchange lemma and $w^{j}_{1} = 0$. 

    Define $W = \SPAN(\{v^{2}, \dots, v^{d}\})$, then every $w \in W$ satisfies $w_{1} = 0$ and 
    \begin{equation*}
        E = \SPAN(\{v^{1}\}) \oplus W.
    \end{equation*}
    We aim to show that for every $w\in W$, we have $w_{i} = 0$ for $i=1,\dots,k$. We argue by contradiction. Assume there exists $w\in W$ and some $j \leq k$ such that $w_{j} > 0$. Let $0< \eps < \frac{\lVert v^{1} \rVert_{\infty} - \max_{j > k}\vert v^{1}_{j} \vert}{\lVert w \rVert_{\infty}}$. Consider the vector $z = v^{1} + \eps \sgn(v^{1}_{j}) w$. We have $\vert z_{1} \vert = \lVert v^{1} \rVert_{\infty}< \lVert z \rVert_{\infty}$ as $w_{1}=0$ and by the choice of $\eps$, we have $\vert z_{j} \vert <  \lVert z \rVert_{\infty}$ for $j > k$. Therefore $\lVert z \rVert_{\infty}$ is attained in at most $k-1$ coordinates, contradicting our assumption. Thus, every $w\in W$ satisfies $w_{i} = 0$ for $i=1,\dots,k$ and therefore, $d -1 = \dim(W) \leq n-k$.   
    Repeating this arguments leads to $d \leq \frac{n}{k}$. This concludes the proof.
\end{proof}

\subsection{Betti number vanishes under positive Ollivier curvature}

As a first result, we now show that the Betti number of a graph vanishes if the graph has non-negative Ollivier curvature everywhere and contains at least one vertex for which every incident edge has positive Ollivier curvature. 

A similar result holds for other curvature notions: In \cite{kempton2021homology}, it was shown that for graphs with strictly positive Bakry-\'Emery curvature the first Betti number of the path homology vanishes, while in \cite{forman2003bochner} Forman shows that the first Betti number of a CW-complex vanishes under positive Forman curvature. For the Ollivier curvature, we obtain the following result.

\begin{theorem}\label{th:posCurvZeroBetto}
    Let $G = (V, w, \mu, d)$ be a finite graph with non-negative Ollivier curvature. Suppose there exists a vertex $x \in V$ such that $\kappa(x,y) > 0$ for all $y \in \ON(x)$. Then, the first Betti number satisfies \( \beta_1(M_2(G)) = 0 \).
\end{theorem}

\begin{proof}
    We argue by contraposition. Suppose that the first Betti number is nonzero. Hence, there exists an $\alpha \in H_1(M_2(G))$ with $\lVert \alpha \rVert_{\infty} = 1$. The associated function $f_{\alpha}$ on the universal cover is 1-Lipschitz and satisfies
    \[
        \Delta f_{\alpha}(x) = 0,
    \]
    for every $x \in V$. According to Lemma~\ref{lem:AlphaEins}, there exists a vertex $y \in \ON(x)$ such that 
    \[
        \frac{\alpha(x,y)}{d(x,y)} = 1.
    \]
    Hence, there exists an edge $\widetilde{x} \sim \widetilde{y}$ in the universal cover with $p(\widetilde{x}) = x$ and $p(\widetilde{y}) = y$ such that $\nabla_{\widetilde{y}\widetilde{x}} f_{\alpha} = 1$ and therefore
    \[
        \kappa(\widetilde{x},\widetilde{y}) \leq \nabla_{\widetilde{x}\widetilde{y}} \Delta f_{\alpha} = 0.
    \]
    This contradicts the assumption that the Ollivier curvature is positive on every edge containing $x$, thereby completing the proof.
\end{proof}

\subsection{Betti number estimate for non-negative Ollivier curvature}
In this subsection, we prove our first main theorem, that is, the Betti number estimate for graphs with non-negative Ollivier curvature in terms of the minimal vertex degree.
\begin{theorem}\label{th:MainRes}
    Let $G=(V,w,\mu,d)$ be a finite graph with non-negative Ollivier curvature. Then
    \begin{equation*}
        \beta_{1}(M_{2}(G)) \leq \frac{\deg_{\min}}{2}.
    \end{equation*}
\end{theorem}

\begin{proof}
    Let $x \in V$ with $\deg(x) = \deg_{\min}$. We consider the map $\psi: H_{1}(M_{2}(G)) \to \R^{\ON(x)}$ given by
    \begin{equation*}
        \psi\alpha(y) = \frac{\alpha(x,y)}{d(x,y)}.
    \end{equation*}
    According to Lemma~\ref{lem:AlphaEins}, $\psi$ is injective, and therefore $\dim\ker\psi = 0$. Furthermore, by Lemma~\ref{lem:AlphaEins}, we have that for every $v \in \psi(H_1(M_2(G)))$, there exist $y,z \in \ON(x)$ such that $\vert v(y) \vert = \vert v(z) \vert  = \lVert v \rVert_{\infty}$. Therefore, we can apply Lemma~\ref{lem:LinAlg} and obtain $\dim(\psi(H_1(M_2(G)))) \leq \frac{\deg_{\min}}{2}$. Using the rank–nullity theorem, we conclude
    \begin{equation*}
        \beta_{1}(M_{2}(G)) = \dim H_{1}(M_{2}(G)) \leq \frac{\deg_{\min}}{2},
    \end{equation*}
    finishing the proof.
\end{proof}

\subsection{Basic facts about Bakry-\'Emery curvature}
We now turn to Bakry-\'Emery curvature and aim to give a similar Betti number estimate as in the case of non-negative Ollivier curvature. We first state a few general facts about Bakry-\'Emery curvature.

The Bakry-\'Emery calculus was introduced in Bakry and \'Emery in \cite{BakryEmery85} to characterize lower Ricci curvature bounds on manifolds via the Bochner formula.
Later, it was extended to continuous metric measure spaces, see e.g. \cite{bakry2006logarithmic,bakry2006diffusions,bakry1996levy,bakry2000some,
savere2014discrete}, and also discrete spaces, see e.g. \cite{schmuckenschlager1998curvature,lin2010ricci}. We now turn to the formal definitions for discrete spaces with a graph Laplacian $\Delta$.
The definition of the Laplacian $\Delta$ gives rise to the carré du champ operator $\Gamma$, a symmetric bilinear form. Namely, for $f,g \in C(V)$, we define
\begin{equation*}
    \Gamma(f,g) = \frac{1}{2}\Big(\Delta(fg) - f\Delta(g) - g\Delta(f)\Big).
\end{equation*} 
By iterating $\Gamma$, we derive another symmetric bilinear form, denoted as $\Gamma_{2}$, which is defined by
\begin{equation*}
    \Gamma_{2}(f,g) = \frac{1}{2}\Big(\Delta\Gamma(f,g) - \Gamma(f,\Delta g) - \Gamma(g, \Delta f)\Big).
\end{equation*}
We introduce the additional notations $\Gamma(f) = \Gamma(f,f)$ and $\Gamma_{2}(f) = \Gamma_{2}(f,f)$.
Using these operators, Bakry and \'Emery \cite{BakryEmery85} defined a Ricci curvature notion, which is motivated by Bochner's formula, a fundamental identity in Riemannian Geometry. Namely, let $\mathcal{K} \in \R$ and $\mathcal{N} \in (0, \infty]$. A vertex $x\in V$ satisfies the curvature-dimension inequality $CD(\mathcal{K},\mathcal{N},x)$, if 
\begin{equation*}
    \Gamma_{2}(f)(x) \geq \frac{1}{\mathcal{N}}(\Delta f(x))^2+\mathcal{K}\Gamma(f)(x)
\end{equation*}
for any $f\in C(V)$. If all vertices $x$ of a graph $G$ satisfy $CD(\mathcal{K}, \mathcal{N},x)$, we say that $G$ satisfies $CD(\mathcal{K}, \mathcal{N})$. We refer to $\mathcal{K}$ as a lower Ricci curvature bound, and $\mathcal{N}$ as the dimension parameter.
For any $x\in V$, we define
\begin{equation*}
    \mathcal{K}_{G,x}(\mathcal{N}) = \sup\{\mathcal{K} \in \R: CD(\mathcal{K},\mathcal{N},x) \},
\end{equation*}
the Bakry-\'Emery curvature function of $x$. Various stronger modifications of this curvature notion can be found in, e.g., \cite{horn2014volume, bauer2015li,dier2021discrete} which were mainly focused on establishing Li-Yau type inequalities and polynomial volume growth. It was later proven in \cite{munch2019li} that a Li-Yau type inequality and polynomial volume growth also follows from the original Bakry-\'Emery curvature condition.

In what follows, a graph $G$ is said to have non-negative Bakry-\'Emery curvature if it satisfies $CD(\mathcal{K}, \mathcal{N})$ for some $\mathcal{K} \geq 0$.

\subsection{Betti number estimates for non-negative Bakry-\'Emery curvature}

In contrast to Ollivier curvature, for Bakry-\'Emery curvature there are already some Betti number estimates discussed in the literature.
As in the case of Ollivier curvature, one first has to extend the graph to a 2-dimensional cell complex such that the universal cover preserves the curvature. As the local dependency of the Bakry curvature is different from the one of the Ollivier curvature, the choice of 2-cells we glue in also has to be different. It turns out that for Bakry-\'Emery curvature, the right choice of 2-cells comes from the path homology \cite{grigor2019homology,grigor2012homologies} which was originally introduced to describe homologies mainly of directed graphs. In the undirected graph case, the two-cells in the corresponding cell complex arising from the path homology are precisely the triangles and quadrilaterals. As discussed in \cite{kempton2021homology}, the universal cover of this cell complex preserves Bakry-\'Emery curvature.
We hence denote by $\MBE(G)$ the 2-dimensional cell complex arising from the graph $G$ when gluing in as 2-cells all cycles of length at most 4. Here by length, we mean the combinatorial graph distance.

While we are interested in the non-negative curvature case, there have been established Betti number estimates for positive curvature and almost positive curvature. More precisely, in \cite{kempton2021homology} it was shown that the first Betti number of $\MBE(G)$ vanishes if $G$ has positive curvature. This was improved in \cite{munch2020spectrally} where a vanishing first Betti number was shown under the strictly weaker condition that the operator $-\frac 1 2 \Delta + K$ has strictly positive spectrum where $K:X_0 \to \R$ is a pointwise lower bound on the Bakry-\'Emery curvature. 

To derive Betti number estimates under non-negative Bakry-\'Emery curvature, we begin by proving that, under the assumption of non-negative Bakry-\'Emery curvature, for a function $f$ with constant Laplacian, $\Gamma f$ must be constant if it attains its supremum norm at some vertex.

\begin{lemma}\label{lem:GammaConstant}
Let $G=(V,w,\mu,d)$ be a connected locally finite graph with non-negative Bakry-\'Emery curvature. Assume $f \in \R^V$ with $\Delta f = C=const$ and $\|\Gamma f\|_\infty = \Gamma f(x)$ for some $x \in V$. Then, $\Gamma f = \|\Gamma f\|_\infty$ everywhere.
\end{lemma}

\begin{proof}
By non-negative Bakry-\'Emery curvature, we have
\[
0 \leq \Delta \Gamma f - 2 \Gamma(f,\Delta f) = \Delta \Gamma f.
\]
Now with $g = \Gamma f$, we have $g(x) = \|g\|_\infty$ and $\Delta g \geq 0$. The maximum principle now implies $g= \|g\|_\infty$ everywhere, finishing the proof.
\end{proof}

For a 1-form $\alpha \in \R^{X_1}$, we define
\[
|\alpha|^2(x) := \frac {1} {2\mu(x)} \sum_y w(x,y)\alpha(x,y)^2 .
\]

Using the previous Lemma, we now show that under the assumption of non-negative Bakry-\'Emery curvature, if $\alpha$ is a harmonic 1-form, then $|\alpha|^2$ must be constant.

\begin{lemma}\label{lem:constAlphaNormBE}
Let $G=(V,w,\mu,d)$ be a finite graph with non-negative Bakry-\'Emery curvature. Let $\alpha \in H_1(\MBE(G))$ be a harmonic 1-form. Then, $|\alpha|^2$ is constant on $X_0$.
\end{lemma}

\begin{proof}
Let $\widetilde{\MBE(G)}$ be the universal cover with covering map $\pr : \widetilde{\MBE(G)} \to \MBE(G)$.
Then, as $\delta\alpha = 0$, there exists $f_\alpha \in \R^{\widetilde X_0}$ with $\widetilde \delta f_\alpha = \alpha \circ \pr$.
Then, $\Gamma(f_\alpha) = |\alpha|^2 \circ \pr$ and hence, $\Gamma(f_\alpha)$ takes its maximum as $\MBE(G)$ is finite.
Moreover, $\widetilde \Delta f_\alpha = 0$ as $\delta^* \alpha =0$.
By Lemma~\ref{lem:GammaConstant}, we see that $\Gamma (f_\alpha) = const$ and hence, $|\alpha|^2$ is constant on $X_0$.
\end{proof}

We are now prepared to prove the main theorem of this subsection, an upper bound for the Betti number under non-negative Bakry-\'Emery curvature.

\begin{theorem}\label{th:BEBetti}
Let $G=(V,w,\mu,d)$ be a finite graph with non-negative Bakry-\'Emery curvature. Then
    \begin{equation*}
        \beta_{1}(M_{4}(G)) \leq \deg_{\min}-1.
    \end{equation*}
\end{theorem}

\begin{proof}
Choose $x\in V$ with $\deg(x)$ minimal.
We consider the restriction map $\psi:H_1 \to \R^{S_1(x)}$ given by 
$\psi \alpha (y) := \alpha(x,y)$.
We notice that assuming $\psi\alpha = 0$, we obtain $|\alpha|^2(x) = 0$
and by Lemma~\ref{lem:constAlphaNormBE}, this implies $|\alpha|^2 = 0$ everywhere giving $\alpha = 0$.
Hence, $\psi$ is injective.
Moreover $\psi(H_1)$ has codimension at least one, as $\sum w(x,y)\psi \alpha (y) = 0$ as $\alpha$ is harmonic.
We conclude
\[
\beta_1 \leq \dim \psi(H_1) \leq \# S_1(x) - 1 = \deg_{\min} - 1
\]
finishing the proof. 
\end{proof}

\section{Sharpness of the Betti number estimates for Ollivier curvature}

The aim of this section is to prove that the upper bound in Theorem~\ref{th:MainRes} is attained if and only if the graph is a potentially twisted torus.
The key idea is to find a suitable basis of the homology group. We will show that each basis element induces a graph automorphism. The group generated by 
these automorphisms turns out to be abelian, and we will prove that its Cayley graph is isomorphic to the underlying graph.

\subsection{Supremum norm rigidity}
As a first step, we will exploit that harmonic one-forms attain their supremum norm twice at each vertex.
We start with an abstract lemma about linear algebra.
\begin{lemma}\label{lem:starrheit}
    Let $E \subseteq \R^n$ be a subspace such that for every $v \in E$, there exist indices $i_{1}< \dots< i_{k}$ satisfying $\vert v_{i_{j}} \vert = \lVert v \rVert_{\infty}$ for $j =1, \dots, k$ and $\dim(E) = \frac{n}{k}$. Then, we can find a partition $\{\mathcal{J}_{1}, \dots, \mathcal{J}_{\frac{n}{k}}\}$ of the set $\{1,\dots, n\}$ such that $\vert \mathcal{J}_{i} \vert = k$ for $i=1, \dots, \frac{n}{k}$, and 
    a basis $(v^1, \dots,v^\frac{n}{k})$ satisfying
    \begin{equation*}
        \vert v^i_j\vert = \begin{cases}
            1&: j \in \mathcal{J}_{i}, \\
            0&: \text{otherwise},
        \end{cases}
    \end{equation*}
    for each $i = 1, \dots, \frac{n}{k}$.
\end{lemma}

\begin{proof}
    We will prove the statement by induction on the dimension of $E$. For simplicity, we write $d=\dim(E)$.\\
    \underline{Base case:} The claim is clear for $d=1$. Choose $v\in E$ with $\lVert v \rVert_\infty = 1$. Then, $(v)$ is a basis with 
    $\vert v_j\vert = 1$ for each $j=1,\dots,n$.\\
    \underline{Inductive step:} Assume the statement is true for $\dim(E)=d-1$. We shall show the statement for $\dim(E)=d$. Choose $v' \in E$ such that there exist exactly $k$ indices $i_1< \dots <i_k$ satisfying $\vert v'_{i_j} \vert = \lVert v' \rVert_\infty=1$ and define $\mathcal{J}_d=\{i_1,\dots,i_k\}$. As shown in Lemma~\ref{lem:LinAlg}, there exists a subspace $W \subset E$ such that  
    \begin{equation*}
        E = \SPAN(\{v'\}) \oplus W,
    \end{equation*} 
    and $w_j = 0$ for $j \in \mathcal{J}_d$, for every $w \in W$.
    We embed $W\subset \R^n$ into $\R^{n-k}$ by the map $\iota: W \to \R^{n-k}$, defined by $\iota(w) = (w_{j})_{j\not\in \mathcal{J}_d}$. Then,
    \begin{equation*}
        \dim(\iota(W)) = d-1 = \frac{n-k}{k}.
    \end{equation*}
    By the induction hypothesis we obtain a partition $\{\mathcal{J}'_{1}, \dots, \mathcal{J}'_{d-1}\}$ of the set $\{1, \dots, n-k\}$ and a basis $(\widetilde{v}^1, \dots,\widetilde{v}^{d-1})$ satisfying
    \begin{equation*}
        \vert \widetilde{v}^i_j\vert = \begin{cases}
            1&: j \in \mathcal{J}'_{i}, \\
            0&: \text{otherwise},
        \end{cases}
    \end{equation*}
    for each $i = 1,\dots,d-1$. Thus, the vectors $v^i = \iota^{-1}(\widetilde{v}^i)$ for $i=1,\dots,d-1$ form a basis of $W$ and it exists a partition $\{\mathcal{J}_1, \dots, \mathcal{J}_d\}$ of the set $\{1,\dots,n\}$ such that
    \begin{equation*}
        \vert v^i_j\vert = \begin{cases}
            1&: j \in \mathcal{J}_{i}, \\
            0&: \text{otherwise},
        \end{cases}
    \end{equation*}
    for each $i = 1, \dots,d-1$. Furthermore, $(v^1, \dots, v^{d-1}, v')$ forms a basis of $E$. Next, we define
    \begin{equation*}
        v^d = v'- \sum_{i=1}^{d-1} v'_{j_i} \sgn(v^{i}_{j_i})v^{i},
    \end{equation*}
    where $j_i$ is an arbitrary but fixed element of $\mathcal{J}_i$. By the Steinitz exchange lemma, $(v^1, \dots,v^d)$ forms a basis of $V$. As $v^i_j=0$ for every $j\in \mathcal{J}_d$ and $i=1,\dots,d-1$, we have $\vert v^d_j \vert = \vert v'_j \vert =1$ for every $j\in \mathcal{J}_d$. Finally, assume there exists an $j\not\in \mathcal{J}_d$ such that $\vert v^d_j \vert > 0$. Then, there exists $\mathcal{J}_k$ such that $j \in \mathcal{J}_k$. Using that $v^d_{j_k}=0$ by construction, we conclude that
    \begin{equation*}
        w = v^k + \frac{1}{2} \cdot \sgn(v^k_j) \cdot\sgn(v^d_j) \cdot v^d
    \end{equation*}
    attains $\lVert w\rVert_\infty$ in at most $k-1$ coordinates, contradicting our assumption. Therefore, the vector $v^d$ also satisfies
    \begin{equation*}
        \vert v^d_j\vert = \begin{cases}
            1&: j \in \mathcal{J}_{d}, \\
            0&: \text{otherwise}.
        \end{cases}
    \end{equation*}
    This concludes the proof.
\end{proof}

\subsection{Finding a suitable basis of the homology group}

The aim of this subsection is to find a basis of the homology group such that for every one-form in the basis and every vertex, there are only two edges containing this vertex, along which the one-form does not vanish.
Using the previous results, it is relatively easy to show this if the basis is allowed to depend on the vertex. The next lemma will be useful to show that the basis can be chosen independently of the vertex.

\begin{lemma}\label{lem:perm}
    Let $A:\R^m \to \R^m$ be a linear map such that $\lVert v \rVert_\infty = \lVert Av \rVert_\infty$ for all $v \in \R^m$. Then the entrywise absolute value $\vert A\vert := (\vert a_{ij} \vert)_{i,j=1,\dots,m}$ of $A$ is a permutation matrix.
\end{lemma}

\begin{proof}
    Denote by $e_k \in \R^m$ the $k$-th standard unit vector, namely, $e_k = (\delta_{ki})_{i=1,\dots,m}$ where $\delta_{ki}$ is the Kronecker delta. Then,
    \begin{equation*}
        1=\lVert e_k \rVert_\infty = \lVert Ae_k \rVert_\infty = \max_{1\leq i\leq m} \vert A_{ik} \vert.
    \end{equation*}
    Thus, each column contains at least one entry with absolute value equal to one. Furthermore, it is not possible to have two entries with absolute value equal to one in the same row. Namely, assume there exists an index $i\in \{1,\dots m\}$ and indices $k\not= k'$ such that $\vert A_{ik} \vert = \vert A_{ik'}\vert = 1$. For $v = \sgn(A_{i_kk}) e_k + \sgn(A_{i_{k'}k'})e_k'$ we obtain
    \begin{equation*}
        1=\lVert v \rVert_\infty = \lVert Av \rVert_\infty \geq 2,
    \end{equation*}
    a contradiction. Hence, each row contains exactly one entry with absolute value equal to one. Namely, for every $i\in \{1,\dots, m\}$, there exists exactly one $k_i \in \{1,\dots, m\}$ such that $\vert A_{ik_i} \vert = 1$. Finally, assume there exists a row with two non-zero entries. Namely, assume there exists an index $i\in \{1,\dots, m\}$ and $l \not = k_i$ such that $\vert A_{il} \vert > 0$. For $v = \sgn(A_{ik_i}) e_{k_i} + \sgn( A_{il})e_{l}$ we obtain
    \begin{equation*}
        1= \lVert v \rVert_\infty =  \lVert Av \rVert_\infty > 1,
    \end{equation*}
    a contradiction. Thus, $\vert A \vert$ is a permutation matrix.
\end{proof}

We are now prepared to prove that there exists a basis of the homology group such that for every one-form in the basis and every vertex, there are only two edges containing this vertex, along which the one-form does not vanish.

\begin{lemma}\label{lem:alphauni}
    Let $G=(V,w,\mu,d)$ be a finite graph with non-negative Ollivier curvature satisfying 
    \begin{equation*}
        \beta_1(M_2(G)) = \frac{\deg_{\max}}{2}=: m.
    \end{equation*}
    Then, there exists a basis $(\alpha_1, \dots, \alpha_{m})$ of $H_1(M_2(G))$, such that for every $i\in\{1,\dots,m\}$ and $x\in X_0$ there exist $y=y(i,x),\;z=z(i,x)\in X_0$ with $x\sim y$, $x\sim z$ and 
    \begin{equation*}
        \alpha_i(x,x') = d(x,x')\big(1_y(x') - 1_z(x')\big),
    \end{equation*} 
    for every $x'\in X_0$.
\end{lemma}

\begin{proof}
    Note that Theorem~\ref{th:MainRes} implies $\deg_{\max} = \deg_{\min}$, and consequently,  
    \begin{equation*}
        \deg(x) = \deg_{\max} \quad \forall x\in V.
    \end{equation*} Let $x\in V$ be arbitrary. We consider the linear map $\psi_{x}: H_1(M_2(G)) \to \R^{\ON(x)}$ given by 
    \begin{equation*}
        \psi_{x}\alpha(y) = \frac{\alpha(x,y)}{d(x,y)}.
    \end{equation*}
    According to Lemma~\ref{lem:AlphaEins}, the map $\psi_{x}$ satisfies 
    \begin{equation*}
        \lVert \psi_{x}\alpha\rVert_\infty = \lVert \alpha\rVert_\infty,
    \end{equation*}
    for every $\alpha \in H_1(M_2(G))$, and it exist $y\not=z\in \ON(x)$ such that 
    \begin{equation*}
        \vert \psi_{x}\alpha(y)\vert = \vert \psi_{x}\alpha(z)\vert = \lVert \alpha \rVert_\infty.
    \end{equation*}
    Furthermore, according to Theorem~\ref{th:MainRes} and by assumption, we have 
    \begin{equation*}
        \dim(\psi_x(H_1(M_2(G)))) = \beta_1(M_2(G)) = \frac{\deg_{\max}}{2}.
    \end{equation*}
    Therefore, we can apply Lemma~\ref{lem:starrheit}, and obtain a partition $\{\mathcal{J}^x_{1}, \dots, \mathcal{J}^x_{m}\}$ of the set $\{1,\dots,\deg_{\max}\}$ such that $\vert \mathcal{J}^x_i\vert =2$ for $k=1,\dots,m$ as well as a basis $(v^x_1, \dots, v^x_m)$ of $\psi_x(H_1(M_2(G)))$ satisfying
    \begin{equation}\label{eq:2}
        \vert v^x_i(j)\vert = \begin{cases}
            1&: j \in \mathcal{J}^x_{i}, \\
            0&: \text{otherwise},
        \end{cases}
    \end{equation}
    for each $i=1,\dots,m$. We can identify $\psi_x(H_1(M_2(G)))$ with $\R^m$ via the isomorphism $\iota_x: \psi_x(H_1(M_2(G))) \to \R^m$ defined by $\iota_x(v^x_i)=e_i$, where $e_i$ denotes the $i$-th standard unit vector in $\R^m$. Observe that 
    \begin{equation*}
        \lVert \iota_x(v)\rVert_\infty = \lVert v \rVert_\infty \quad \forall v\in \psi_x(H_1(M_2(G))).
    \end{equation*}
    
    Now, we fix an arbitrary $x_0 \in V$. Define $\alpha_i = \psi^{-1}_{x_0}(v^{x_0}_i)$ for $i=1,\dots, m$. Using the linear independence of the family $(\alpha_1,\dots,\alpha_m)$ and the assumption that $\dim H_1(M_2(G)) = m$, we conclude that 
    \begin{equation*}
        H_1(M_2(G)) = \SPAN(\alpha_1, \dots, \alpha_m).
    \end{equation*}
    Finally, let $x\in V$ be arbitrary. The map 
    \begin{equation*}
        A = \iota_x \circ \psi_{x} \circ \psi^{-1}_{x_0} \circ \iota^{-1}_{x_0}: \R^m \to \R^m 
    \end{equation*}
    is linear since it is a composition of linear maps. Furthermore, $A$ satisfies $\lVert v \rVert_\infty = \lVert Av \rVert_\infty$ for all $v\in \R^m$, as it is a composition of maps that preserve the $\lVert \cdot \rVert_\infty$ norm. Therefore, we can apply Lemma~\ref{lem:perm} and conclude that $\vert A \vert$ is a permutation matrix. Thus, there exists $k\in \{1,\dots,m\}$ and $s \in \{1,-1\}$ such that
    \begin{equation*}
        \psi_x(\alpha_i) = \iota^{-1}_x \circ A \circ \iota_{x_0} \circ \psi_{x_0}(\alpha_{i}) = \iota^{-1}_x \circ A(e_i) = \iota^{-1}_x (s\cdot e_k) =s\cdot v^x_k.
    \end{equation*} 
    We conclude that there exist $y,z\in V$ such that $x\sim y$, $x\sim z$ and 
    \begin{equation*}
        \frac{\alpha_i(x,x')}{d(x,x')} = 1_y(x') -1_z(x').
    \end{equation*}
    Here, we used Lemma~\ref{lem:AlphaEins}, to guarantee that $\frac{\alpha_i(x,y)}{d(x,y)} = -\frac{\alpha_i(x,z)}{d(x,z)}$. This concludes the proof.
\end{proof}

\subsection{Automorphisms induced by 1-forms}
Using the basis of the homology group from the previous subsection, we will construct certain graph automorphisms. The Cayley graph of the group generated by these graph automorphisms will later be shown to be abelian and isomorphic to the underlying graph.

We recall a finite graph $G=(V,w,\mu,d)$ is Ollivier-Betti sharp if it has non-negative Ollivier curvature and 
\begin{equation*}
    \beta_1(M_2(G)) = \frac{\deg_{\max}}{2} =: m.
\end{equation*}

Recall that, according to Lemma~\ref{lem:alphauni}, there exists a basis $(\alpha_1,\dots, \alpha_m)$ of $H_1(M_2(G))$, such that for every $i\in \{1,\dots,m\}$ and $x\in X_0$, there exist $y=y(i,x), \; z=z(i,x)\in X_0$ with $x\sim y, x\sim z$ and 
\begin{equation*}
    \alpha_i(x,x')=d(x,x')\big(1_y(x') -1_z(x')\big),
\end{equation*}
for every $x'\in X_0$. Using this, we define the maps $\phi_i: X_0 \to X_0$ by
\begin{equation*}
    \phi_i(x) = y(i,x),
\end{equation*}
for $i=1,\dots,m$. We aim to show that $\phi_i$ and $\phi_j$ commute for arbitrary $i,j\in\{1,\dots,m\}$ and that these maps are graph automorphisms. We start by proving that the maps are bijective and that they map the same nodes to distinct images.

\begin{lemma}\label{lem:phiBijective}
    Let $G=(V,w,\mu,d)$ be an Ollivier-Betti sharp graph.
    Then, for $i\not =j$, we have $\phi_i(x) \neq \phi_j(x)$ and $\phi^{-1}_i(x) \neq \phi_j(x)$  for all $x\in X_0$. Furthermore, the maps $\phi_i$ are bijective.
\end{lemma}

\begin{proof}
    We argue by contradiction. Assume there exist $i\not=j$ and $x\in X_0$ such that $\phi_i(x) = \phi_j(x)$. If $z(i,x) = z(j,x)$, we obtain
    \begin{equation*}
        \alpha_i(x,x') -\alpha_j(x,x') = 0\quad \forall x'\in X_0.
    \end{equation*}
    Lemma~\ref{lem:AlphaEins} implies that $\alpha_i -\alpha_j=0$, contradicting the linear independence of $\alpha_i$ and $\alpha_j$. Thus, $z(i,x) \not= z(j,x)$ and
    \begin{equation*}
        \alpha(x,x') := \frac{1}{2}\Big(\alpha_i(x,x') + \alpha_j(x,x')\Big) = d(x,x') \Big[(1_{\phi_i(x)}(x') - \frac{1}{2}\Big(1_{z(i,x)}(x') + 1_{z(j,x)}(x')\Big)\Big]. 
    \end{equation*}
    This contradicts Lemma~\ref{lem:AlphaEins}, as $\lVert \alpha \rVert_\infty = 1$ and for all $x'\in X_0$, we have $\alpha(x,x')\not=-d(x,x')$. Thus, our assumption is false, and $\phi_i(x) \not=\phi_j(x)$ must hold.

    Next, assume there exist $i\not=j$ and $x\in X_0$ such that $\phi^{-1}_i(x) = \phi_j(x)$. If $\phi_i(x) = z(j,x)$, we obtain
    \begin{equation*}
        \alpha_i(x,x') +\alpha_j(x,x') = 0\quad \forall x'\in X_0.
    \end{equation*}
    Lemma~\ref{lem:AlphaEins} implies that $\alpha_i +\alpha_j=0$, contradicting the linear independence of $\alpha_i$ and $\alpha_j$. Thus, $\phi_i(x) \not= z(j,x)$ and 
    \begin{equation*}
        \alpha'(x,x') := \frac{1}{2}\Big(\alpha_j(x,x') - \alpha_i(x,x')\Big) = d(x,x')\Big[1_{\phi_j(x)}(x') - \frac{1}{2}\Big(1_{\phi_i(x)}(x') + 1_{z(j,x)}(x')\Big)\Big]. 
    \end{equation*}
    This contradicts Lemma~\ref{lem:AlphaEins}, as $\lVert \alpha' \rVert_\infty = 1$ and for all $x'\in X_0$, we have $\alpha'(x,x')\not=-d(x,x')$. Thus, our assumption is false, and $\phi^{-1}_i(x) \not=\phi_j(x)$ must hold.

    It remains to show that the maps $\phi_i$ are bijective. Let $i\in \{1,\dots,m\}$ be arbitrary. Since $\phi_i:X_0 \to X_0$ and $X_0$ is a finite set, it is sufficient to prove that $\phi_i$ is injective. Assume this is not the case; namely, there exist $x\not=z\in X_0$ such that $\phi_i(x) = \phi_i(z)=:y$. Thus, $\alpha_i(y, x) = -d(y,x)$ and  $\alpha_i(y, z) = -d(y,z)$. This contradicts our assumptions on $\alpha_i$.
\end{proof}

Next, we calculate the supremum norm of one-forms.

\begin{lemma}\label{lem:supNormAlphaLinKomb}
    Let $G=(V,w,\mu,d)$ be an Ollivier-Betti sharp graph. Then, for all $\lambda_1,\dots, \lambda_m\in \R$,
    \begin{equation*}
        \lVert \lambda_1 \alpha_1 + \dots +  \lambda_m \alpha_m \rVert_\infty = \max_{1\leq k\leq m} \vert \lambda_k\vert.
    \end{equation*}
\end{lemma}

\begin{proof}
    Let $k\in \{1,\dots,m\}$ and $x\in X_0$ be arbitrary. Then $\alpha_k(x,\phi_k(x))=d(x,\phi_k(x))$, and, according to Lemma~\ref{lem:phiBijective}, $\alpha_j(x,\phi_k(x))=0$ for $j\not=k$. Thus, 
    \begin{equation*}
        \vert \lambda_k \vert = \abs*{\frac{\lambda_1 \alpha_1(x,\phi_k(x)) + \dots +  \lambda_1 \alpha_m(x,\phi_k(x)) }{d(x,\phi_k(x))}} \leq \lVert \lambda_1 \alpha_1 + \dots +  \lambda_m \alpha_m \rVert_\infty.
    \end{equation*}
    Since $k\in\{1,\dots,m\}$ was chosen arbitrarily , we obtain 
    \begin{equation*}
        \max_{1\leq k\leq m} \vert \lambda_k\vert \leq \lVert \lambda_1 \alpha_1 + \dots +  \lambda_m \alpha_m \rVert_\infty.
    \end{equation*}
    To prove the other inequality, let $x\sim y$ be an edge such that
    \begin{equation*}
        \abs*{\frac{\lambda_1 \alpha_1(x,y) + \dots +  \lambda_1 \alpha_m(x,y)}{d(x,y)}} = \lVert \lambda_1 \alpha_1 + \dots +  \lambda_m \alpha_m \rVert_\infty.
    \end{equation*}
    There exists $i\in \{1,\dots,m\}$ such that $\vert \alpha_i(x,y)\vert = d(x,y)$, and according to Lemma~\ref{lem:phiBijective}, we have $\alpha_j(x,y)=0$ for all $j\not=i$. Hence,
    \begin{equation*}
        \lVert \lambda_1 \alpha_1 + \dots +  \lambda_m \alpha_m \rVert_\infty =  \vert \lambda_i \vert \leq \max_{1\leq k\leq m} \vert \lambda_k \vert.
    \end{equation*}
    Putting together the upper and lower estimates yields
    \begin{equation*}
        \lVert \lambda_1 \alpha_1 + \dots +  \lambda_m \alpha_m \rVert_\infty = \max_{1\leq k\leq m} \vert \lambda_k\vert.
    \end{equation*}
    as desired.
\end{proof}

We next show that the maps $\phi_i$ preserve distances in most directions.
\begin{lemma}\label{lem:LaengeSeiten}
    Let $G=(V,w,\mu,d)$ be an Ollivier-Betti sharp graph. Let $x,y\in V$ with $y=\phi_j(x)$. Then, for $i\not=j$, we have
    \begin{equation*}
        d(x,y) = d(\phi_i(x), \phi_i(y)) \quad \mbox{ and } \quad d(x,\phi_i(x)) = d(y,\phi_i(y)).
    \end{equation*}
\end{lemma}

\begin{proof}
    We first lift $\phi_i$ to a map $\widetilde \phi_i$ on the universal cover by requiring that  $\phi_i \circ \pr = \pr \circ \widetilde \phi_i$ and that $\widetilde \phi_i(\widetilde x) \sim \widetilde x$ for all $\widetilde x \in \widetilde X_0$. Let $x\in V$ be arbitrary and let $\widetilde{x} \in \widetilde{X}_0$ with $\pr(\widetilde{x})=x$. Additionally, let $\widetilde{y} = \widetilde{\phi}_j(\widetilde{x})$ and note that $\pr(\widetilde{y}) = \phi_j(x) = y$. We assume without loss of generality that $f_{\alpha_k}(\widetilde x) = 0$ for all $k$.  We notice that for all $\widetilde x' \in \CN(\widetilde x)$,
    \begin{equation*}
        f_{\alpha_k}(\widetilde x') - f_{\alpha_k}(\widetilde x) =
        \begin{cases}
        d(\widetilde{x}, \widetilde{x}')&: \widetilde x' = \widetilde \phi_k(\widetilde x) \\
        -d(\widetilde{x}, \widetilde{x}')&: \widetilde x = \widetilde \phi_k(\widetilde x') \\
        0&: \mbox{ otherwise}.
        \end{cases}
    \end{equation*}
    for all $k$. Define $g= f_{\alpha_i} + f_{\alpha_j}$. By Lemma~\ref{lem:supNormAlphaLinKomb}, we have $\|\nabla g\|_\infty = 1$. Moreover, on $\CN(\widetilde x) \setminus \{\widetilde \phi_i(\widetilde x)\}$
    \begin{equation*}
        g= f_{\alpha_i} + f_{\alpha_j} \leq f_{\alpha_j}.
    \end{equation*}
    For all $\widetilde x' \in \CN(\widetilde x) \setminus \{\widetilde \phi_i(\widetilde x)\}$,
    \begin{align*}
        f_{\alpha_j}(\widetilde{\phi_i}(\widetilde y)) - f_{\alpha_j}(\widetilde x') &=  g(\widetilde{\phi_i}(\widetilde y)) - \widetilde{d}(\widetilde{y}, \widetilde{\phi_i}(\widetilde y)) - f_{\alpha_j}(\widetilde x')\\
        &\leq g(\widetilde{\phi_i}(\widetilde y)) - \widetilde{d}(\widetilde{y}, \widetilde{\phi_i}(\widetilde y)) - g(\widetilde x') \\
        &\leq \widetilde{d}(\widetilde{\phi_i}(\widetilde y), \widetilde{x}') - \widetilde{d}(\widetilde{y}, \widetilde{\phi_i}(\widetilde y)) < \widetilde{d}(\widetilde{\phi_i}(\widetilde y), \widetilde{x}') .
    \end{align*}
    By Lemma~\ref{lem:LipFun} $(ii)$ applied to $f_{\alpha_j}$, 
    this implies
    \begin{equation*}
        \widetilde{d}(\widetilde \phi_i (\widetilde x),\widetilde{\phi_i}(\widetilde y)) = f_{\alpha_j}(\widetilde{\phi_i}(\widetilde y)) - f_{\alpha_j}(\widetilde \phi_i(\widetilde x)).
    \end{equation*}
    Using $f_{\alpha_j}(\widetilde \phi_i(\widetilde x)) = 0$ and $f_{\alpha_j}(\widetilde{\phi_i}(\widetilde y)) = \widetilde{d}(\widetilde{x}, \widetilde{y})$, we conclude
    \begin{equation*}
        d(x,y) = \widetilde{d}(\widetilde x, \widetilde y) = \widetilde{d}(\widetilde \phi_i (\widetilde x),\widetilde{\phi_i}(\widetilde y)) = d(\phi_i(x), \phi_i(y)),
    \end{equation*}
    where we used Lemma~\ref{lem:Propertiesd} for the last equality. For the second part of the proof, let $\lambda \in \{-1,1\}$. By Lemma~\ref{lem:supNormAlphaLinKomb}, we have $\|\nabla (f_{\lambda \alpha_i} + f_{\alpha_j})\|_\infty = 1$. Hence,
    \begin{align*}
        \widetilde{d}(\widetilde{x},\widetilde{y}) = \widetilde{d}(\widetilde \phi_i (\widetilde x),\widetilde{\phi_i}(\widetilde y)) &\geq (f_{\lambda \alpha_i} + f_{\alpha_j})(\widetilde{\phi}_i(\widetilde{y})) - (f_{\lambda \alpha_i} + f_{\alpha_j})(\widetilde{\phi}_i(\widetilde{x})) \\
        &= \lambda \widetilde{d}(\widetilde{y}, \widetilde{\phi}_i(\widetilde{y})) + \widetilde{d}(\widetilde{x}, \widetilde{y}) - \lambda \widetilde{d}(\widetilde{x}, \widetilde{\phi}_i(\widetilde{x})),
    \end{align*}
    and we conclude $d(x,\phi_i(x)) = \widetilde{d}(\widetilde{x}, \widetilde{\phi}_i(\widetilde{x})) = \widetilde{d}(\widetilde{y}, \widetilde{\phi}_i(\widetilde{y})) = d(y,\phi_i(y))$.
\end{proof}

We are now prepared to prove that the maps $\phi_i$ commute.

\begin{lemma}\label{lem:phiCommute}
    Let $G=(V,w,\mu,d)$ be an Ollivier-Betti sharp graph. Then $\phi_i\phi_j=\phi_j\phi_i$ for all $i,j=1,\dots,m$.
\end{lemma}

\begin{proof}
    Let $x\in V$ be arbitrary. According to Lemma~\ref{lem:LaengeSeiten}, we have 
    \begin{equation}\label{eq:3}
        d(\phi_i(x), \phi_i\phi_j(x)) = d(x,\phi_j(x)) = d(\phi_i(x), \phi_j\phi_i(x)).
    \end{equation}
    Choose $\widetilde{y}\sim \widetilde{\phi_i}(\widetilde{x})$ such that $\widetilde{d}(\widetilde{\phi_i}(\widetilde{x}),\widetilde{y}) + \widetilde{d}(\widetilde{y},\widetilde{\phi_i}\widetilde{\phi_j}(\widetilde{x})) = \widetilde{d}(\widetilde{\phi_i}(\widetilde{x}),\widetilde{\phi_i}\widetilde{\phi_j}(\widetilde{x}))$. As shown in Lemma~\ref{lem:LaengeSeiten}, we have 
    \begin{equation*}
        \widetilde{d}(\widetilde \phi_i (\widetilde x),\widetilde{\phi_i}\widetilde{\phi}_j(\widetilde{x})) = f_{\alpha_j}(\widetilde{\phi_i}\widetilde{\phi}_j(\widetilde{x})) - f_{\alpha_j}(\widetilde \phi_i(\widetilde x)).
    \end{equation*}
    Using that $\| \nabla f_{\alpha_j}\|_\infty=1$, we obtain
    \begin{equation*}
        f_{\alpha_j}(\widetilde{y}) -f_{\alpha_j}(\widetilde{\phi_i}(\widetilde{x})) = \widetilde{d}(\widetilde{y},\widetilde{\phi_i}(\widetilde{x})) = d(\pr(\widetilde{y}), \phi_i(x)).
    \end{equation*}
    By the definition of $f_{\alpha_j}$, it follows that $\alpha_j(\phi_i(x), \pr(\widetilde{y}))= d(\phi_i(x), \pr(\widetilde{y}))$ and consequently, $\pr(\widetilde{y}) = \phi_j\phi_i(x)$ and $\widetilde{y} = \widetilde{\phi_j}\widetilde{\phi_i}(\widetilde{x})$. 
    
    Since $\widetilde{y}\sim \widetilde{\phi_i}(\widetilde{x})$ , we have $\widetilde{d}(\widetilde{\phi_i}(\widetilde{x}),\widetilde{y})= d(\phi_i(x),\phi_j\phi_i(x))$, and by Lemma~\ref{lem:Propertiesd}, we have $\widetilde{d}(\widetilde{\phi_i}(\widetilde{x}),\widetilde{\phi_i}\widetilde{\phi_j}(\widetilde{x})) = d(\phi_i(x), \phi_i\phi_j(x))$. Furthermore, let $\widetilde{y} = \widetilde{y}_0 \sim \widetilde{y}_1 \sim \dots \sim \widetilde{y}_n =  \widetilde{\phi_i}\widetilde{\phi_j}(\widetilde{x})$ be a geodesic from $\widetilde{y}$ to $\widetilde{\phi_i}\widetilde{\phi_j}(\widetilde{x})$. Then,
    \begin{equation*}
        d(\pr(\widetilde{y}), \phi_i\phi_j(\widetilde{x})) \leq \sum_{k=1}^{n} d(\pr(\widetilde{y}_k), \pr(\widetilde{y}_{k-1})) = \sum_{k=1}^{n} \widetilde{d}(\widetilde{y}_k, \widetilde{y}_{k-1}) = \widetilde{d}(\widetilde{y},\widetilde{\phi_i}\widetilde{\phi_j}(\widetilde{x})).
    \end{equation*} 
    
    Hence,
    \begin{equation*}
        d(\phi_i(x), \phi_i\phi_j(x)) \geq d(\phi_i(x), \phi_j\phi_i(x)) + d(\phi_j\phi_i(x), \phi_i\phi_j(x)),
    \end{equation*}
    and using Equation~\ref{eq:3}, we conclude $\phi_j\phi_i(x) = \phi_i\phi_j(x)$.
\end{proof}

We will now consider graph isomorphisms of weighted graphs which only need to preserve the topology. Particularly, we say two weighted graphs $G_i=(V_i,w_i,\mu_i)$ for $i=1,2$ are isomorphic, if there exists a bijective $\Psi:V_1 \to V_2$ such that $w_1(x_1,y_1) >0$ if and only if $w_2(\Psi(x_1),\Psi(y_1))>0$. Using this definition, we prove that the maps $\phi_i$ are automorphisms.

\begin{lemma}
    Let $G=(V,w,\mu,d)$ be an Ollivier-Betti sharp graph. Then, the maps $\phi_i$ are graph automorphisms.
\end{lemma}
  
\begin{remark}
Note that these graph automorphisms need not preserve the weights.
\end{remark}

\begin{proof}
    Let $i\in \{1,\dots,m\}$ be arbitrary. According to Lemma~\ref{lem:phiBijective}, the map $\phi_i$ is bijective. To complete the proof, we have to show that $x\sim y$ if and only if $\phi_i(x) \sim \phi_i(y)$. Since $G$ has only finitely many edges, it suffices to prove the implication $x\sim y \implies \phi_i(x)\sim \phi_i(y)$. To this end, let $x\sim y$ be an arbitrary edge in $G$. Choose $j\in\{1,\dots,m\}$ such that $\vert \alpha_j(x,y)\vert = d(x,y)$. Thus, we have either $\phi_j(x) = y$ or $\phi_j(y) = x$. If $\phi_j(x)=y$, we obtain
    \begin{equation*}
        \phi_i(y) = \phi_i\phi_j(x) = \phi_j\phi_i(x) \sim \phi_i(x),
    \end{equation*}
    where we used the commutativity of $\phi_i$ and $\phi_j$, as stated in Lemma~\ref{lem:phiCommute}. Similarly, if $\phi_j(y)=x$, we obtain
    \begin{equation*}
        \phi_i(x) = \phi_i\phi_j(y) = \phi_j\phi_i(y) \sim \phi_i(y).
    \end{equation*}
\end{proof}

Denote by $\SG=\langle \phi_1,\dots,\phi_m\rangle$ the subgroup of the graph's automorphism group $\AUT(G)$ generated by the automorphisms $\phi_i$. We show that the Cayley graph of $\SG$ with respect to the generating set $\{\phi_1,\dots,\phi_m\}$, denoted by $\CAY(\SG,\{\phi_1,\dots,\phi_m\})$, is isomorphic to $G$.

\subsection{Ollivier-Betti sharp implies abelian Cayley}

Having established that the maps $\phi_i$ commute and are graph automorphisms, we now show that the Cayley graph generated by these maps is isomorphic to the underlying graph.

\begin{theorem}\label{th:GisomorphCay}
    Let $G=(V,w,\mu,d)$ be an Ollivier-Betti sharp graph. Then $G$ is isomorphic to the Cayley graph $\CAY(\SG,\{\phi_1,\dots,\phi_m\})$.
\end{theorem}
\begin{remark}
Note that this graph isomorphism need not preserve the weights.
\end{remark}

\begin{proof}
    Fix an arbitrary vertex $x\in V$, and define the map $\Psi:\SG \to V$ by
    \begin{equation*}
        \Psi(\phi_1^{a_1}\phi_2^{a_2}\dots\phi_m^{a_m}) = \phi_1^{a_1}\phi_2^{a_2}\dots\phi_m^{a_m}(x).
    \end{equation*}
    We will prove that $\Psi$ is a graph isomorphism.
    
    For surjectivity, let $y\in V$ be an arbitrary vertex. As $G$ is connected, there exists a path $x=x_0\sim x_1\sim \dots \sim x_n=y$ from $x$ to $y$. For $j=1,\dots, n$ there exists an $i_j\in\{1,\dots,m\}$ such that $\vert \alpha_{i_j}(x_{j-1},x_j)\vert = d(x_{j-1},x_j)$. Thus, there exists $\eps_{j}\in \{-1,1\}$ such that $\phi_{i_j}^{\eps_{j}}(x_{j-1}) = x_j$. Using the commutativity of the maps $\phi_k$, we obtain
    \begin{equation*}
        \Psi(\phi_1^{a_1}\phi_2^{a_2}\dots\phi_m^{a_m}) = \phi_1^{a_1}\phi_2^{a_2}\dots\phi_m^{a_m}(x) = y,
    \end{equation*}
    where $a_k = \sum_{j=1}^{n}\eps_{j} \delta_k(i_j)$, and $\delta_k(i_j)$ denotes the Kronecker delta.

    For injectivity, assume there exists $\phi_1^{a_1}\phi_2^{a_2}\dots\phi_m^{a_m}\in \SG$ and $\phi_1^{b_1}\phi_2^{b_2}\dots\phi_m^{b_m}\in \SG$, such that
    \begin{equation*}
        \Psi(\phi_1^{a_1}\phi_2^{a_2}\dots\phi_m^{a_m}) = \Psi(\phi_1^{b_1}\phi_2^{b_2}\dots\phi_m^{b_m}).
    \end{equation*}
    We will show that $\phi_1^{a_1}\phi_2^{a_2}\dots\phi_m^{a_m}=\phi_1^{b_1}\phi_2^{b_2}\dots\phi_m^{b_m}$. To this end, let $y\in V$ be arbitrary. By surjectivity, there exists $\phi_1^{c_1}\phi_2^{c_2}\dots \phi_m^{c_m}\in \SG$, such that
    \begin{equation*}
        \Psi(\phi_1^{c_1}\phi_2^{c_2}\dots\phi_m^{c_m}) = \phi_1^{c_1}\phi_2^{c_2}\dots\phi_m^{c_m}(x) = y.
    \end{equation*}
    Thus, using the commutativity, we obtain
    \begin{align*}
        \phi_1^{a_1}\phi_2^{a_2}\dots\phi_m^{a_m}(y) &= \phi_1^{a_1}\phi_2^{a_2}\dots\phi_m^{a_m}(\phi_1^{c_1}\phi_2^{c_2}\dots\phi_m^{c_m}(x)) \\
        &= \phi_1^{c_1}\phi_2^{c_2}\dots\phi_m^{c_m}\phi_1^{a_1}\phi_2^{a_2}\dots\phi_m^{a_m}(x)\\
        &= \phi_1^{c_1}\phi_2^{c_2}\dots\phi_m^{c_m}\phi_1^{b_1}\phi_2^{b_2}\dots\phi_m^{b_m}(x) \\
        &= \phi_1^{b_1}\phi_2^{b_2}\dots\phi_m^{b_m}(\phi_1^{c_1}\phi_2^{c_2}\dots\phi_m^{c_m}(x)) \\
        &= \phi_1^{b_1}\phi_2^{b_2}\dots\phi_m^{b_m}(y).
    \end{align*}
    Since $y\in V$ was chosen arbitrarily, we conclude that $\phi_1^{a_1}\phi_2^{a_2}\dots\phi_m^{a_m}=\phi_1^{b_1}\phi_2^{b_2}\dots\phi_m^{b_m}$.

    It remains to show that $\Psi$ is edge-preserving. To this end, let $\phi_1^{a_1}\phi_2^{a_2}\dots\phi_m^{a_m}$ be an arbitrary vertex of $\CAY(\SG,\{\phi_1,\dots,\phi_m\})$, $j\in \{1,\dots,m\}$ and $\eps\in \{-1,1\}$. We obtain 
    \begin{equation*}
        \Psi(\phi_j^{\eps}\phi_1^{a_1}\phi_2^{a_2}\dots\phi_m^{a_m}) = \phi_j^{\eps}\phi_1^{a_1}\phi_2^{a_2}\dots\phi_m^{a_m}(x) 
        = \phi_j^{\eps}(\Psi(\phi_1^{a_1}\phi_2^{a_2}\dots\phi_m^{a_m})) \sim \Psi(\phi_1^{a_1}\phi_2^{a_2}\dots\phi_m^{a_m}).
    \end{equation*}
    Conversely, assume $\Psi(\phi_1^{a_1}\phi_2^{a_2}\dots\phi_m^{a_m}) \sim \Psi(\phi_1^{b_1}\phi_2^{b_2}\dots\phi_m^{b_m})$ for some $\phi_1^{a_1}\phi_2^{a_2}\dots\phi_m^{a_m}, \; \phi_1^{b_1}\phi_2^{b_2}\dots\phi_m^{b_m}\in \SG$. Then there exist $i\in \{1,\dots,m\}$ and $\eps \in \{-1,1\}$ such that 
    \begin{equation*}
        \phi_i^{\eps}(\Psi(\phi_1^{a_1}\phi_2^{a_2}\dots\phi_m^{a_m})) = \Psi(\phi_1^{b_1}\phi_2^{b_2}\dots\phi_m^{b_m}). 
    \end{equation*}
    Since $\Psi$ is injective and 
    \begin{equation*}
        \phi_i^{\eps}(\Psi(\phi_1^{a_1}\phi_2^{a_2}\dots\phi_m^{a_m})) = \phi_i^{\eps}\phi_1^{a_1}\phi_2^{a_2}\dots\phi_m^{a_m}(x) = \Psi(\phi_i^{\eps}\phi_1^{a_1}\phi_2^{a_2}\dots\phi_m^{a_m}),
    \end{equation*}
    we conclude that $\phi_i^{\eps}\phi_1^{a_1}\phi_2^{a_2}\dots\phi_m^{a_m} = \phi_1^{b_1}\phi_2^{b_2}\dots\phi_m^{b_m}$. Thus, we conclude $\phi_1^{a_1}\phi_2^{a_2}\dots\phi_m^{a_m} \sim \phi_1^{b_1}\phi_2^{b_2}\dots\phi_m^{b_m}$ in $\CAY(\SG,\{\phi_1,\dots,\phi_m\})$. This concludes the proof.
\end{proof}

Hence, we have shown that an Ollivier-Betti sharp graph is isomorphic to an abelian Cayley graph. We can also say more about the edge- and vertex-weights of an Ollivier-Betti sharp graph.

\begin{lemma}\label{lem:GraphsWeights}
    Let $G=(V,w,\mu,d)$ be an Ollivier-Betti sharp graph. Then, for an edge $x\sim y$ with $y=\phi_i(x)$, we obtain
    \begin{itemize}
        \item[$(i)$] $w(x,y)\cdot d(x,y) = w(y, \phi_i(y))\cdot d(y,\phi_i(y)) $,
        \item[$(ii)$] $\frac{w(x,y)}{\mu(x)} = \frac{w(\phi_j(x), \phi_j(y))}{\mu(\phi_j(x))}$ for all $j \in\{1,\dots,m\}$ with $j\not=i$.
    \end{itemize}
\end{lemma}

\begin{proof}
    $(i)$ Since $\alpha_i\in H_1(M_2(G))$, it follows that
    \begin{equation*}
        0 = \delta^*\alpha_i(y) = \sum_{z: y\sim z} \frac{w(y,z)}{\mu(z)}\alpha_i(y,z) = \frac{w(y,\phi_i(y))}{\mu(y)}d(y,\phi_i(y)) - \frac{w(y,x)}{\mu(y)}d(y,x).
    \end{equation*}
    Using $y=\phi_i(x)$, we conclude $w(x,y)\cdot d(x,y)= w(\phi_i(x),\phi_i(y))\cdot d(\phi_i(x),\phi_i(y)) $.

    $(ii)$ We first lift $\phi_i$ to a map $\widetilde{\phi_i}$ on the universal cover by requiring that $\phi_i\circ \pr = \pr \circ \widetilde{\phi_i}$ and that $\widetilde{\phi_i}(\widetilde{z}) \sim \widetilde{z}$ for all $\widetilde{z} \in \widetilde{X_0}$. Next, we fix an $\widetilde{x} \in \pr^{-1}(x)$. Observe that $\pr(\widetilde{\phi_i}(\widetilde{x})) = \phi_i(\pr(\widetilde{x})) = y$. Let $j\not = i$ be arbitrary. We assume without loss of generality that $f_{\alpha_i}(\widetilde{x}) = f_{\alpha_j}(\widetilde{x}) = 0$ and define
    \begin{equation*}
        g := \max\{f_{\alpha_j},f_{(\alpha_j + \alpha_i)}\} = f_{\alpha_j} + (f_{\alpha_i})_+.
    \end{equation*}
    Observe that $g$ is the pointwise maximum of two 1-Lipschitz functions, and hence it is itself 1-Lipschitz. Furthermore, we have $g(\widetilde{x}) = 0$ and $g(\widetilde{\phi_j}(\widetilde{x})) = d(\widetilde{x},\widetilde{\phi_j}(\widetilde{x}) )$. 
    By the non-negative Ollivier curvature of the universal cover, we obtain
    \begin{equation*}
        0 \leq \kappa(\widetilde{x},\widetilde{\phi_j}(\widetilde{x})) \leq \Delta g(\widetilde{x}) - \Delta g(\widetilde{\phi_j}(\widetilde{x})) = \Delta f_{\alpha_j}(\widetilde{x}) + \Delta(f_{\alpha_i})_+(\widetilde{x}) - \Delta f_{\alpha_j}(\widetilde{\phi_j}(\widetilde{x})) - \Delta(f_{\alpha_i})_+(\widetilde{\phi_j}(\widetilde{x})).
    \end{equation*}
    Using $\Delta f_{\alpha_j}=0$, we conclude
    \begin{equation*}
        \frac{w(x,y)}{\mu(x)}d(x,y) 
        =\Delta(f_{\alpha_i})_+(\widetilde{x})
        \geq \Delta(f_{\alpha_i})_+(\widetilde{\phi_j}(\widetilde{x})) 
        = \frac{w(\phi_j(x),\phi_j(y))}{\mu(\phi_j(x))}d(\phi_j(x), \phi_j(y)).
    \end{equation*}
    Finally, observe that $d(x,y) = d(\phi_j(x), \phi_j(y))$ by Lemma~\ref{lem:LaengeSeiten}. Applying the same reasoning to the function 
    \begin{equation*}
        g' = \max\{f_{-\alpha_j}, f_{(\alpha_i - \alpha_j)}\}
    \end{equation*}
    yields the other inequality. This concludes the proof.
\end{proof}

We now introduce the notion of a discrete weighted flat torus which will characterize Ollivier-Betti sharpness.

\begin{definition}\label{def:Torus}
    We say $T=((V,\cdot), (a_1,\dots,a_m), w,\mu,d)$ is a weighted, flat Torus if $(V,\cdot)$ is an abelian group with generating set $\{a_1,\dots,a_m\} \subseteq V$ 
    satisfying $a_i \neq a_j^{\eps}$ for all $i\neq j$ and $\eps \in \{-1,1\}$, and $d: V\times V \to [0,\infty)$ is a general path metric, 
    such that 
    \begin{itemize}
        \item[$(i)$] $w(x,y) > 0$ if and only if $y \in\{a_i^{\eps}\cdot x:  i=1,\dots,m;\;  \eps =\pm 1\}$ for all $x,y \in V$,
        \item[$(ii)$] $w(x,a_i\cdot x)d(x,a_i\cdot x) = w(x, a_i^{-1}\cdot x) d(x, a_i^{-1}\cdot x)$ for all $x\in V$ and $i\in \{1,\dots,m\}$,
        \item[$(iii)$] $p(x,a_i^{\eps_1}\cdot x) = p(a_j^{\eps_2}\cdot x, a_i^{\eps_1}\cdot a_j^{\eps_2} \cdot x)$, where $p(x,y) = \frac{w(x,y)}{\mu(x)}$, for all $x\in V$, $i,j\in \{1,\dots,m\}$ with $i\not=j$ and $\eps_1,\eps_2\in\{-1,1\}$,
        \item[$(iv)$] $d(x,a_i\cdot x) = d(a_j\cdot x, a_i\cdot a_j\cdot x)$ for all $x\in V$ and $i,j\in \{1,\dots,m\}$ with $i\not=j$,
        \item[$(v)$] $\beta_1(M_2(T)) = m$. 
    \end{itemize}
\end{definition}

Next, we prove that a graph satisfying the properties $(i)-(iv)$ of the previous definition has non-negative Ollivier curvature.

\begin{lemma}\label{lem:nonnegORC}
    Let $G=(V,w,\mu,d)$ be a finite graph. Let $"\cdot"$ be an abelian group structure on $V$ with generating set $\{a_1,\dots,a_m\}\subseteq V$ with $a_i\neq a_j^{\eps}$ for all $i\neq j$ and $\eps \in \{-1,1\}$ such that properties $(i)-(iv)$ of Definition~\ref{def:Torus} are satisfied. Then $G$ has non-negative Ollivier curvature.
\end{lemma}

\begin{proof}
    Let $x_0\sim y_0$ be an arbitrary edge in $G$. By property $(i)$ there exists $i\in \{1,\dots,m\}$ and $\eps \in\{-1,1\}$ such that $y_0 = a_i^{\eps}\cdot x_0$. Without loss of generality, we assume $r := p(x_0,a_i^{-\eps}x_0)- p(y_0,a_i^{\eps}y_0) \geq 0$.
     
     We define the following transport plan $\rho: \CN(x_0) \times \CN(y_0)\to [0,\infty)$ by

     \begin{equation*}
        \rho(x,y) = \begin{cases}
            p(y_0, x_0)
            &: x=y=x_0,\\
            p(x_0, y_0)
            &: x=y=y_0,\\
            p(y_0,y) &: (x,y) \in \{(a_j^{\eps'}x_0, a_j^{\eps'}y_0): \eps' \in \{-1,1\}, j\neq i \} \mbox{ or } (x,y) = (a_i^{-\eps}x_0,  a_i^{\eps}y_0) \\
            r &: (x,y)= (a_i^{-\eps}x_0, y_0),\\
            0 &: \mbox{otherwise}.
            \end{cases}
    \end{equation*}
    By property $(iii)$, we have $p(x_0, a_j^{\eps'}x_0) = p(y_0, a_j^{\eps'}y_0)$ for all $j\not=i$ and $\eps'\in \{-1,1\}$, and we conclude that $\rho$ is a transport plan.  
    Using that $d(a_i^{-\eps}x_0,a_i^{\eps}y_0) \leq d(a_i^{-\eps}x_0,x_0) + d(x_0,y_0) + d(y_0,a_i^{\eps}y_0)$, we obtain
    \[
        p(y_0,a_i^{\eps}y_0)\left[d(x_0,y_0) - d(a_i^{-\eps}x_0, a_i^{\eps}y_0) \right] \geq -p(y_0,a_i^{\eps}y_0) \left[d(a_i^{-\eps}x_0,x_0) + d(y_0,a_i^{\eps}y_0)\right].
    \]
    Similarly, we obtain
    \[
        r\left[d(x_0,y_0) - d(a_i^{-\eps}x_0,y_0)\right] \geq d(a_i^{-\eps}x_0,x_0)\left[p(y_0,a_i^{\eps}y_0) - p(x_0,a_i^{-\eps}x_0)\right].
    \]
    Putting these inequalities together, we obtain
    \begin{align*}
        d(x_0,y_0) \kappa(x_0,y_0) &\geq \sum_{\substack{x\in \CN(x_0)\\y\in \CN(y_0)}} \rho(x,y) \left[d(x_0,y_0)- d(x,y)\right] \\
        &\geq d(x_0,y_0)\left[p(y_0,x_0) + p(x_0y_0)\right] -p(y_0,a_i^{\eps}y_0)d(y_0,a_i^{\eps}y_0) - p(x_0,a_i^{-\eps}x_0)d(x_0,a_i^{-\eps}x_0)\\
        &= 0,
    \end{align*}
    where we used property $(iv)$ for the second inequality and property $(ii)$ for the last equality. Since $d(x_0,y_0) >0$, we conclude that $\kappa(x_0,y_0) \geq 0$ must holds. This concludes the proof.
\end{proof}

\subsection{Characterization of Ollivier Betti sharpness}

In this subsection, we prove our second main theorem, that is, a graph is Ollivier Betti sharp if and only if it is a weighted flat torus.

\begin{theorem}\label{th:CharacBS}
    Let $G=(V,w,\mu,d)$ be a finite graph. Then $G$ is Ollivier-Betti sharp with $\beta_1(M_2(G)) = m$ if and only if $((V,\cdot), (a_1,\dots,a_m), w, \mu,d)$ is a weighted, flat Torus for some abelian group structure $"\cdot"$ on $V$ and some generators $a_1, \dots,a_m$.
\end{theorem}

\begin{proof}
    $"\implies"$ As before, let $\SG=\langle\phi_1, \dots,\phi_m\rangle$ denote the subgroup of the graph's automorphism group $\AUT(G)$, generated by the automorphisms $\phi_i$. 
    By Theorem~\ref{th:GisomorphCay}, the graph is isomorphic to the abelian Cayely graph of $\SG$ with generators $\phi_i$. In particular,    
    the group structure on $\SG$ induces the group structure on $V$.
    Note that by Lemma~\ref{lem:phiBijective} the $\phi_i$ and $\phi_j^{-1}$ are mutually distinct.   
    Property $(i)$ is an immediate consequence of Theorem~\ref{th:GisomorphCay}. Properties $(ii)$ and $(iii)$ follow from Lemma~\ref{lem:GraphsWeights}. Lemma~\ref{lem:LaengeSeiten} implies property $(iv)$ and $(v)$ follows by assumption.

    $"\impliedby"$ By Lemma~\ref{lem:nonnegORC}, the graph has non-negative Ollivier curvature. Furthermore, $G$ is $2m$-regular by property $(i)$ of Definition~\ref{def:Torus}. Using property $(v)$, we conclude
    \[
        \beta_1(M_2(G)) = m = \frac{\deg_{\max}}{2}.
    \] 
    Hence, the graph is Ollivier-Betti sharp.
\end{proof}

Next, we turn to the special case of the combinatorial path distance, namely
\[
    d(x,y) = \inf\{n : x = x_0\sim \dots \sim x_n = y\}.
\]  

In this case, we derive the following equivalent characterization of a weighted flat torus.

\begin{corollary}\label{cor:BettiSharpCombDist}
    Let $G=(V,w,\mu,d)$ be a finite graph, where $d$ denotes the combinatorial path distance. Then $G$ is Ollivier-Betti sharp with $\beta_1(M_2(G))=m$ if and only if there exists some abelian group structure $"\cdot"$ on $V$ and some generators $a_1,\dots,a_m$, such that
    \begin{itemize}
        \item[$(i)$] $w(x,y) > 0$ if and only if $y \in\{a_i^{\eps}\cdot x:  i=1,\dots,m;\;  \eps =\pm 1\}$ for all $x,y \in S$,
        \item[$(ii)$] $w(x,a_i\cdot x) = w(x, a_i^{-1}\cdot x)$ for all $x\in V$ and $i\in \{1,\dots,m\}$,
        \item[$(iii)$] $p(x,a_i^{\eps_1}\cdot x) = p(a_j^{\eps_2}\cdot x, a_i^{\eps_1}\cdot a_j^{\eps_2} \cdot x)$, where $p(x,y) = \frac{w(x,y)}{\mu(x)}$, for all $x\in V$, $i,j\in \{1,\dots,m\}$ with $i\not=j$ and $\eps_1,\eps_2\in\{-1,1\}$,
        \item[$(iv)$] $\forall v \in \Z^m$ with $0 < \lVert v \rVert_1 \leq 5$, we have $a_1^{v_1}\cdots a_m^{v_m} \not = e$.
    \end{itemize}
\end{corollary}

\begin{proof}
    $"\implies"$ By Theorem~\ref{th:GisomorphCay}, $G$ is isomorphic to the Cayley graph $\CAY(S,\{\phi_1,\dots,\phi_m\})$. Denote by $(\alpha_1,\dots,\alpha_m)$ the corresponding basis of $H_1(M_2(G))$. As in the proof of the general case in Theorem~\ref{th:CharacBS}, the abelian group structure on $V$ is induced by the abelian group structure on $\SG$ and $(i)$ follows trivially. Properties $(ii)$ and $(iii)$ follow again from Lemma~\ref{lem:GraphsWeights} and Lemma~\ref{lem:LaengeSeiten}. Finally, assume there exists an $v\in \Z^m$ with $0<\lVert v\rVert_1\leq 5$ and $a_1^{v_1}\dots a_m^{v_m} =e$. For an arbitrary $x_0\in V$, this induces a cycle $(x_0\sim x_1 \sim \dots \sim x_n)$, where $n=\lVert v\rVert_1-1$, by iteratively multiplying the factors of $a_1^{v_1}\dots a_m^{v_m}$ to $x_0$. This can be done in arbitrary order, since the graph is abelian. Since $\lVert v\rVert_1\leq 5$, this cycle is an element of $X_2$. Let $i = \min\{j: v_j\not=0\}$, then
    \[
        \delta\alpha_i((x_0\sim \dots\sim x_n)) = \sum_{j=0}^{n-1} \alpha_i(x_{j}, x_{j+1\mod n})= v_i \not=0.
    \]
    Here, we used that 
    \begin{equation*}
        \alpha_i(x_j, x_{j+1\mod n}) = \begin{cases}
            1 &: x_{j+1\mod n} = a_ix_j,\\
            -1 &: x_{j+1\mod n} = a_i^{-1}x_j, \\
            0 &: \mbox{otherwise},
        \end{cases}
    \end{equation*}
    Therefore, $\delta\alpha_i((x_0\sim \dots\sim x_n)) \neq 0$, a contradiction to $\alpha_i\in H_1(M_2(G))$.

    $"\impliedby"$ Assume $G$ satisfies properties $(i)-(iv)$. Lemma~\ref{lem:nonnegORC} implies non-negative Ollivier curvature. Thus, it suffices to show that $\beta_1(M_2(G)) =m$. According to $(i)$ and $(iv)$, $G$ is $2m$-regular. By Theorem~\ref{th:MainRes}, we obtain
    \[
        \beta_1(M_2(G)) \leq m.
    \]
    For the other inequality, we define $\alpha_i\in C(X_1)$ by
    \begin{equation*}
        \alpha_i(x,y) = \begin{cases}
            1 &: y = a_ix,\\
            -1 &: y = a_i^{-1}x, \\
            0 &: \mbox{otherwise},
        \end{cases}
    \end{equation*}
     for $i=1\dots,m$. Then, 
     \begin{equation*}
        \delta^*\alpha_i(x) = \sum_{y:x\sim y} \frac{w(x,y)}{\mu(x)}\alpha_i(x,y) = \frac{w(x,a_ix)}{\mu(x)} - \frac{w(x,a_i^{-1}x)}{\mu(x)} = 0
     \end{equation*}
     for all $x\in V$ and $i\in \{1,\dots,m\}$. Next, let $(x_0\sim x_1 \sim \dots \sim x_{k-1})$ be a cycle of length at most $5$. Using property $(iv)$ and that $G$ is an abelian Cayley graph, all cycles of length at most five must have length exactly four and correspond to the commutators of $a_j$ and $a_k$. More precisely, we obtain $k=4$ and there exist $j,k\in \{1,\dots,m\}$ and $\eps_1, \eps_2 \in \{-1,1\}$ such that $x_1 = a_j^{\eps_1}x_0,\; x_2 = a_k^{\eps_2}x_1, \; x_3 = a_j^{-\eps_1}x_2$ and $x_0 = a_k^{-\eps_2}x_3$. Thus,
     \begin{equation*}
        \delta \alpha_i((x_0\sim x_1 \sim x_2 \sim x_3)) = \alpha_i(x_0, a_j^{\eps_1}x_0) + \alpha_i(x_2, a_j^{-\eps_1}x_2) + \alpha_i(x_1, a_k^{\eps_2}x_1) + \alpha_i(x_3, a_k^{-\eps}x_3) = 0,
     \end{equation*}
     for all $i\in \{1,\dots,m\}$. Therefore, we conclude $\alpha_i \in H_1(M_2(G))$ for all $i=1,\dots,m$. The family of maps $(\alpha_1, \dots, \alpha_m)$ are linearly independent, since $\alpha_j(x,a_ix)= 0$ for all $j\not =i$ and $x\in V$. Hence, we conclude
     \begin{equation*}
        \beta_1(M_2(G)) = \dim H_1(M_2(G)) \geq m.
     \end{equation*}
     This shows that $G$ is Ollivier-Betti sharp and thus, the proof of the corollary is finished.
\end{proof}

\subsection{The non-reversible case}\label{Sec:NonRevCase}

In this subsection we generalize the previous results to the non-reversible case. To this end, let $(V,p)$ be a Markov chain, i.e., $V$ is a finite or countably infinite set and the map $p:V\times V\to [0,1]$ satisfies
\[
    \sum_{y\in V}p(x,y) =1 \quad \mbox{for all } x\in V.
\]
$V$ is called the state space and $p$ is the transition probability function. In the following, we will primarily consider the case where $V$ is finite. However, we introduce the notion of a Markov chain $(V,p)$ in this generality, as it will also be relevant when working with the universal cover of the induced graph.

We always assume that the Markov chain is irreducible, that is, for every $x,y\in V$, there exist $x_0,\ldots,x_n\in V$ such that $x_0=x, x_n=y$ and $p(x_i,x_{i+1}) > 0$ for $i=0,\ldots,n-1$. Furthermore, we assume that $p$ has symmetric support, i.e., for $x,y\in V$, we have $p(x,y) > 0$ if and only if $p(y,x) >0$.

We call a probability measure $\pi$ on $V$ an invariant distribution if 
\[
    \sum_{x\in V}p(x,y) \pi(x) = \pi(y).
\]
It is a well-known fact that an irreducible Markov chain on a finite state space $V$ has a unique and positive invariant distribution.

As the graph Laplacian is not self-adjoint anymore, it cannot be expressed as $\delta^* \delta$. Therefore, we replace $\delta^*$ by $\dst : C(X_1) \to C(X_0)$,
\[
\dst \alpha (x) := \sum_{y\in V} p(x,y)\alpha(x,y).
\]
Then, the graph Laplacian is given by
\[
\Delta f(x) = \sum_{y\in V}p(x,y)(f(y)-f(x)) = -\dst \delta f(x).
\]
The graph Laplacian satisfies the following properties.

\begin{lemma}\label{lem:laplace}
    Let $(V,p)$ be a Markov chain with finite state space. Then, the graph Laplacian $\Delta: C(V) \to C(V)$ satisfies
    \begin{align*}
        \ker(\Delta) &= \{f \in C(V) : \exists c\in \R \;s.t. \; f(x) = c \;\forall x\in V \}, \\
        \Imm(\Delta) &= \{f \in C(V): \langle f, \pi \rangle = 0\},
    \end{align*}
    where $\pi$ denotes the invariant distribution of $(V,p)$.
\end{lemma}

\begin{proof}
    Assume $f\in \ker(\Delta)$. Choose $x\in V$ such that $f(x) =\max_{y\in V} f(y)$. Since
    \[
        0 = \Delta f(x) = \sum_{y \in V} p(x,y)(f(y) -f(x)),
    \]
    we conclude that $f(y) = f(x)$ for every $y\in V$ with $p(x,y) > 0$. Using the irreducibility of the Markov chain, we conclude that $f$ must be constant. Since every constant function $f$ satisfies $\Delta f \equiv 0$, we conclude
    \[
        \ker(\Delta) = \{f \in C(V) : \exists c\in \R \;s.t. \; f(x) = c \;\forall x\in V \}.
    \]

    Assume $f\in \Imm(\Delta)$. Hence, there exists $g\in C(V)$ such that $f = \Delta g$. Therefore, we obtain
    \begin{align*}
        \langle f,\pi \rangle &= \sum_{x\in V} f(x) \pi(x) \\
        &= \sum_{x\in V} \sum_{y\in V} p(x,y)(g(y)-g(x)) \pi(x)\\
        &= \sum_{y \in V}g(y)\sum_{x\in V} p(x,y)\pi(x) - \sum_{x\in V}\pi(x)g(x)\sum_{y\in V}p(x,y) \\
        &= \sum_{y\in V}g(y)\pi(y) - \sum_{x\in V}g(x)\pi(x) \\
        &=0.
    \end{align*}
    Thus, $\Imm(\Delta) \subseteq \{f \in C(V): \langle f, \pi \rangle = 0\}$. By the rank-nullity theorem, we obtain
    \[
        \dim(\Imm(\Delta)) = \dim(C(V)) - \dim(\ker(\Delta)) = \vert V \vert -1.
    \]
    Using, that $\dim(\{f \in C(V): \langle f, \pi \rangle = 0\}) = \vert V \vert -1$, we conclude
    \[
        \Imm(\Delta) = \{f \in C(V): \langle f, \pi \rangle = 0\}.
    \] 
\end{proof}

We define
\[
H_1(M_2(G)) := \{ \alpha \in C(X_1):\delta \alpha = 0, \dst \alpha = const \}.
\]
For convenience, we write $\delta_k := \delta|_{C(X_k)}$. Using this, we define the first Betti number by
\[
    \beta_1(M_2(G)) := \dim \left(\frac{\ker \delta_1}{\Imm \delta_0} \right).
\]
Then, we obtain the following characterization of the Betti number.
\begin{lemma}
    Let $(V,p)$ be a Markov chain with finite state space, and let $d:V \times V \to [0,\infty)$ be a general path distance on $V$. Then
    \[
        \beta_1(M_2(G)) = \dim(H_1(M_2(G))).
    \]
\end{lemma}

\begin{proof}
    First, observe that every $\alpha \in H_1(M_2(G))$ satisfies $\alpha \in \ker\delta_1$. Therefore, we can define the map 
    \[
        \Phi:H_1(M_2(G)) \to \frac{\ker\delta_1}{\Imm\delta_0}
    \]
    by $\alpha \mapsto [\alpha]$, where $[\alpha]$ denotes the equivalence class of $\alpha$ in $\frac{\ker\delta_1}{\Imm\delta_0}$. Note that $\Phi$ is a linear map. We first show that $\Phi$ is injective. To this end, let $\alpha \in H_1(M_2(G))$ such that $\Phi(\alpha) = [\alpha] = 0$. Therefore, there exists $f\in C(X_0)$ such that $\alpha = \delta_0 f$. Using that $\alpha \in H_1(M_2(G))$, we conclude that there exists a constant $c\in \R$ such that
    \[
        c = \dst\alpha(x) = \dst\delta f(x) = -\Delta f(x),
    \]
    for every $x\in V$. Assume $c > 0$. Choose $x \in V$ such that $f(x) = \min_{y\in V}f(y)$. Then,
    \[
        c = -\Delta f(x) = \sum_{y\in V}p(x,y)(f(x)-f(y)) \leq 0,
    \]
    a contradiction. Analogously, it can be demonstrated that $c < 0$ is not possible. Therefore, we conclude that $f \in \ker(\Delta)$.  By Lemma~\ref{lem:laplace}, $f$ must be constant, and consequently, 
    \[
        \alpha(x,y) = \delta f(x,y) = f(x) -f(y) = 0,
    \] 
    for every edge $x\sim y$. This proves that $\Phi$ is injective.

    For surjectivity, let $[\alpha_0] \in \frac{\ker \delta_1}{\Imm\delta_0}$ be arbitrary and denote by $\pi$ the invariant distribution of $(V,p)$. Define $c := \langle \dst \alpha_0 , \pi \rangle$ and $g := \dst\alpha_0 - c$. Since $\langle g, \pi \rangle = 0$, Lemma~\ref{lem:laplace} ensures the existence of $f \in C(V)$ such that $\Delta f = g$. Note that $[\alpha_0] = [\alpha_0 + \delta_0 f]$ and 
    \[
        \dst(\alpha_0 + \delta_0 f)(x)
        = \dst \alpha_0(x) - \Delta f(x) = \dst \alpha_0(x) - g(x) = c
    \]
    for every $x \in V$. Therefore, we conclude $\alpha_0 + \delta_0 f \in H_1(M_2(G))$ and 
    \[
        \Phi(\alpha_0 + \delta_0 f) = [\alpha_0].
    \]
    Since $[\alpha_0]$ was chosen arbitrarily, we obtain that $\Phi$ is surjective. Thus, $\Phi$ is an isomorphism, and we conclude
    \[
        \beta_1(M_2(G)) = \dim\left(\frac{\ker \delta_1}{\Imm \delta_0} \right)=\dim(H_1(M_2(G)))
    \]
\end{proof}

Using this characterization of the Betti number, we can generalize our previous results to the non-reversible case. As before, we obtain that the first Betti number vanishes if the Markov chain has non-negative Ollivier curvature everywhere and some positive curvature.

\begin{theorem}\label{th:Betti_vanish_non_rev}
    Let $G=(V,p)$ be a Markov chain with finite state space, and let $d:V \times V \to [0,\infty)$ be a general path distance on $V$. Assume that $G$ has non-negative Ollivier-curvature and that there exists an $x\in V$ such that the Ollivier curvature $\kappa(x,y) > 0$ for all $y \in V$ with $p(x,y)> 0$. Then, the first Betti number satisfies $\beta_{1}(M_{2}(G)) = 0$.   
\end{theorem}

Similarly, the upper bound for the Betti number derived in Theorem~\ref{th:MainRes} generalizes to the non-reversible case.

\begin{theorem}\label{th:Main_res_non_rev}
Let $G=(V,p)$ be a Markov chain with finite state space, and let $d:V \times V \to [0,\infty)$ be a general path distance on $V$.  If $G$ has non-negative Ollivier curvature, then
    \begin{equation*}
        \beta_{1}(M_{2}(G)) \leq \frac{\deg_{\min}}{2}.
    \end{equation*}
\end{theorem}

The proofs of these theorems are analogous to those in the reversible case. We remark, that for $\alpha \in H_1(M_2(G))$ the associated function $f_{\alpha}$ on the universal cover is no longer harmonic. However, there exists a constant $c\in \R$ such that
\[
    \Delta f_{\alpha}(x) = c 
\]
for every $x\in V$. Since Lemma~\ref{lem:LipFun} is applicable in this case, the proofs proceed in exactly the same manner as the proofs in the reversible case.

For an analogue of the rigidity result in Theorem~\ref{th:CharacBS}, we remark that only Lemma~\ref{lem:GraphsWeights} does not extend directly to the case of Markov chains. However, we arrive at the following generalization.

\begin{lemma}
    Let $G=(V,p)$ be a Markov chain with finite state space, and let $d:V \times V \to [0,\infty)$ be a general path distance on $V$. Assume $G$ has non-negative Ollivier curvature and satisfies 
    \[
        \beta_{1}(M_{2}(G)) = \frac{\deg_{\min}}{2}.
    \]
    Then, for every $i\in \left\{1,\ldots,\frac{\deg_{\min}}{2}\right\}$, there exists a constant $c_i \in \R$, such that
    \begin{itemize}
        \item[$(i)$] $p(x, \phi_i(x))d(x,\phi_i(x)) - p(x, \phi_i^{-1}(x)) d(x,\phi_i^{-1}(x)) = c_i$, for every $x \in V$,
        \item[$(ii)$] $p(x,\phi_i(x)) = p(\phi_j(x), \phi_j(\phi_i(x)))$ for all $j\not=i$.
    \end{itemize}
\end{lemma}

\begin{proof}
    $(i)$ Let $i\in \left\{1,\ldots,\frac{\deg_{\min}}{2}\right\}$ be arbitrary. Since $\alpha_i\in H_1(M_2(G))$, there exists a constant $c_i\in \R$, such that $\dst \alpha_i = c_i$. Now, let $x\in V$ be arbitrary. We obtain
    \[
        c_i = \dst\alpha_i(x) = \sum_{y \in V} p(x,y)\alpha_i(x,y) = p(x, \phi_i(x))d(x,\phi_i(x)) - p(x, \phi_i^{-1}(x)) d(x,\phi_i^{-1}(x)).
    \]

    $(ii)$ We first lift $\phi_i$ to a map $\widetilde{\phi_i}$ on the universal cover by requiring that $\phi_i\circ \pr = \pr \circ \widetilde{\phi_i}$ and that $\widetilde{\phi_i}(\widetilde{z}) \sim \widetilde{z}$ for all $\widetilde{z} \in \widetilde{X_0}$. Next, we fix an $\widetilde{x} \in \pr^{-1}(x)$. Observe that $\pr(\widetilde{\phi_i}(\widetilde{x})) = \phi_i(\pr(\widetilde{x})) = \phi_i(x)$. Let $j\not = i$ be arbitrary. We assume without loss of generality that $f_{\alpha_i}(\widetilde{x}) = f_{\alpha_j}(\widetilde{x}) = 0$ and define
    \begin{equation*}
        g := \max\{f_{\alpha_j},f_{(\alpha_j + \alpha_i)}\} = f_{\alpha_j} + (f_{\alpha_i})_+.
    \end{equation*}
    Observe that $g$ is the pointwise maximum of two 1-Lipschitz functions, and hence it is itself 1-Lipschitz. Furthermore, we have $g(\widetilde{x}) = 0$ and $g(\widetilde{\phi_j}(\widetilde{x})) = d(\widetilde{x},\widetilde{\phi_j}(\widetilde{x}) )$. 
    By the non-negative Ollivier curvature of the universal cover, we obtain
    \begin{equation*}
        0 \leq \kappa(\widetilde{x},\widetilde{\phi_j}(\widetilde{x})) \leq \Delta g(\widetilde{x}) - \Delta g(\widetilde{\phi_j}(\widetilde{x})) = \Delta f_{\alpha_j}(\widetilde{x}) + \Delta(f_{\alpha_i})_+(\widetilde{x}) - \Delta f_{\alpha_j}(\widetilde{\phi_j}(\widetilde{x})) - \Delta(f_{\alpha_i})_+(\widetilde{\phi_j}(\widetilde{x})).
    \end{equation*}
    Using that $\Delta f_{\alpha_j}$ is constant, we conclude
    \begin{equation*}
        p(x,\phi_i(x))d(x,\phi_i(x)) 
        =\Delta(f_{\alpha_i})_+(\widetilde{x})
        \geq \Delta(f_{\alpha_i})_+(\widetilde{\phi_j}(\widetilde{x})) 
        = p(\phi_j(x),\phi_j(\phi_i(x))) d(\phi_j(x), \phi_j(\phi_i(x))).
    \end{equation*}
    Finally, observe that $d(x,\phi_i(x)) = d(\phi_j(x), \phi_j(\phi_i(x)))$ by Lemma~\ref{lem:LaengeSeiten}. Applying the same reasoning to the function 
    \begin{equation*}
        g' = \max\{f_{-\alpha_j}, f_{(\alpha_i - \alpha_j)}\}
    \end{equation*}
    yields the other inequality. This concludes the proof.
\end{proof}

Hence, we obtain the following rigidity result for the case of Markov chains.

\begin{theorem}\label{th:Rigidity_Non_Rev}
Let $G=(V,p)$ be a Markov chain with finite state space, and let $d:V \times V \to [0,\infty)$ be a general path distance on $V$.
The following are equivalent.
\begin{enumerate}[(i)]
\item $G$ has non-negative Ollivier-curvature and $\beta_1(M_2(G)) = \deg_{\min}/2$.
\item There exists some abelian group structure $"\cdot"$ on $V$ with generating set $\{a_1, \ldots,a_m\}$ with $a_i \neq a_j^{\eps}$ for all $i\neq j$ and $\eps \in \{-1,1\}$ such that
\begin{enumerate}[(a)]
    \item $p(x,y) > 0$ if and only if $y\in \{a_i^{\eps}\cdot x :x=1,\ldots,m;\; \eps \in \{-1,1\} \}$ for all $x,y \in V$,
    \item For every $i=1,\ldots,m$, there exists $c_i\in \R$ such that $p(x,a_i\cdot x)d(x,a_i\cdot x) - p(x,a_i^{-1}\cdot x)d(x,a_i^{-1}\cdot x) = c_i$ for all $x\in V$,
    \item $p(x,a_i^{\eps_1}\cdot x) = p(a_j^{\eps_2}\cdot x, a_j^{\eps_1} \cdot a_i^{\eps_2}\cdot x)$ for all $x \in V, \; i,j\in \{1,\ldots,m\}$ with $i \neq j$ and $\eps_1, \eps_2 \in \{-1,1\}$,
    \item $d(x,a_i \cdot x) = d(a_j \cdot x, a_j \cdot a_i \cdot x)$ for all $x \in V$ and $i,j \in \{1,\ldots,m\}$ with $i\neq j$, 
    \item $\beta_1(M_2(G)) = m$. 
\end{enumerate}
\end{enumerate}

\end{theorem}

\section{Betti numbers beyond non-negative Ollivier curvature}

In this section, we have two goals. First, we investigate how the rigidity result changes if we strengthen our curvature assumption to even require bone-idleness explained below.
Second, we will give Betti number estimates for the case that we allow some negative curvature.

\subsection{Bone-idleness}

Bone-idleness was first introduced in \cite{bourne2018ollivier} in the context of the normalized Laplacian, i.e., for the case where $w:V\times V\to \{0,1\}$ and $\mu(x) = \deg(x)$ for every $x\in V$, and the combinatorial path distance was used as the underlying metric. A classification of all regular bone-idle graphs with small vertex degrees is provided in \cite{hehl2024graphs}. We generalize this concept to weighted graphs that satisfy
\[
    \sum_{y:x\sim y}w(x,y) = \mu(x),
\]
for every $x\in V$, equipped with a general path distance $d:V\times V \to [0,\infty)$.
To this end, following \cite{munch2017ollivier}, we define the finitely supported probability measures
\[
    m^{\eps}_x(y) = \begin{cases}
        1-\eps &: y = x,\\
        \eps p(x,y) &: \mbox{otherwise},
    \end{cases}
\]
for $\eps \in [0,1]$. Following the standard definition, we define, for $x\neq y$
\[
    \kappa_{\eps}(x,y) = 1- \frac{W(m^{\eps}_x, m^{\eps}_y)}{d(x,y)}
\]
where $W$ denotes the Wasserstein distance. In particular, the Wasserstein distance $W(\nu_1,\nu_2)$ for two probability measures $\nu_1$ and $\nu_2$ on $V$ is given by 
\[
    W(\nu_1,\nu_2) = \inf_{\rho} \sum_{x,y\in V}\rho(x,y)d(x,y)
\]
where the infimum is taken over all $\rho: V\times V\to [0,1]$ which satisfy $\sum_{y\in V}\rho(x,y)=\nu_1(x)$ and $\sum_{x\in V}\rho(x,y)=\nu_2(y)$ for all $x,y\in V$. We call such a $\rho$ a coupling between $\nu_1$ and $\nu_2$. Couplings are also referred to as transport plans; however, we avoid using this terminology here to clearly distinguish them from the transport plans defined in Proposition~\ref{Prop:TransPlan}.

According to \cite[Theorem 2.1]{munch2017ollivier}, the Ollivier curvature is given by
\[
    \kappa(x,y) = \lim_{\eps \to 0^+}\frac{1}{\eps}\kappa_{\eps}(x,y).
\]
Observe that $\kappa(x,y) = \kappa'_0(x,y)$, where the derivative is taken with respect to the idleness parameter. Analogous to \cite[Lemma 2.1]{lin2011ricci}, one can prove that the function $\kappa_{\eps}(x,y)$ is concave in $\eps \in [0,1]$. Since the graph of a concave function lies below its tangent line at each point, we obtain 
\[
    \kappa_{\eps}(x,y) \leq \kappa_0(x,y) + \kappa_0'(x,y)\eps = \eps \kappa(x,y).
\]

We are now ready to introduce the concept of bone-idleness.
\begin{definition}
    Let $G=(V,w,\mu,d)$ be a locally finite graph. We say that $G$ is bone-idle if
    \[
        \sum_{y: x\sim y} w(x,y) = \mu(x) \quad \forall x\in V \quad \mbox{and} \quad \kappa_{\eps}(x,y) =0,
    \]
    for every $\eps \in [0,1]$ and every edge $x\sim y$ of $G$.
\end{definition}

\subsection{Bone-idleness -- Colinear edges have same length}

To give a characterization of bone-idle Ollivier-Betti sharp graphs, we need the following technical lemma.

\begin{lemma}\label{lem:constdist}
    Let $G=((V,\cdot), (a_1,\dots,a_m)w,\mu,d)$ be a weighted, flat Torus and assume $G$ is bone-idle. Let $i\in \{1,\dots,m\}$ be arbitrary and assume $x\in V$ satisfies
    \[
        d(x,a_ix) = \min_{x'\in V}d(x',a_ix').
    \]
    Then, $d(a_i^{-1}x,x) = d(x,a_ix) = d(a_ix, a_i^{2}x)$ and $\mu(x) = \mu(a_ix)$.
\end{lemma}

\begin{proof}
    Denote by $(\alpha_1,\dots, \alpha_m)$ the corresponding basis of $H_1(M_2(G))$, i.e.,
    \[
        \alpha_i(x,a_jx) = \delta_{ij} d(x,a_jx) \quad \mbox{and} \quad \alpha_i(x,a^{-1}_jx) = -\delta_{ij} d(x,a_j^{-1}x).
    \]
    Let $i\in \{1,\dots,m\}$ be arbitrary and $x \in V$ such that 
    \[
        d(x,a_ix) = \min_{x'\in V}d(x',a_ix').
    \]
    For ease of reading, we introduce the notation $y=a_ix$ and $z=a_i^{-1}x$. We assume, without loss of generality, that $\mu(x) \geq \mu(y)$. If this is not the case, we can replace the generator $a_i$ by $a_i^{-1}$.
    Since $G$ is bone-idle, the equality
    \[
        W(m_x^1, m_y^1) = d(x,y)
    \]
    holds true.
    Denote by $f_{\alpha_i}$ the associated 1-Lipschitz function on the universal cover that satisfies $f_{\alpha_i}(x) = 0$. Recall that the functions $f_{\alpha_i}$ are harmonic. Let $g$ be the minimal Lipschitz extension of 
    $(f_{\alpha_i} + \eps 1_x)|_{\CN(y)}$ given by 
    \[
        g(x') = \max_{y' \in \CN(y)} (f_{\alpha_i} + \eps 1_x)(y') - d(y',x')
    \]
    for every $x' \in V$. Since $G$ is a flat Torus, it follows that 
    \[
        (g-f_{\alpha_i})|_{\CN(x)} = \eps 1_x + \eps 1_z,
    \]
    for $\eps$ small enough.
    Hence, we obtain by the Kantorovich duality 
    \begin{align*}
        d(x,y) &= W(m_y^1, m_x^1) \\
        &\geq \sum_{y': y\sim y'} g(y') m_y^1(y') - \sum_{x': x\sim x'} g(x') m_x^1(x') \\
        &= \sum_{y': y\sim y'} (f_{\alpha_i} + \eps 1_x)(y') m_y^1(y') - \sum_{x': x\sim x'} (f_{\alpha_i} + \eps1_x + \eps 1_z)(x') m_x^1(x') \\ 
        &= \sum_{y': y\sim y'} f_{\alpha_i}(y') m_y^1(y') + \eps m_y^1(x) - \sum_{x': x\sim x'} f_{\alpha_i}(x') m_x^1(x') - \eps m_x^1(z) \\
        &= d(x,y) + \eps (m_y^1(x) - m_x^1(z)),
    \end{align*}
    where we used that $f_{\alpha_i}$ is harmonic, $f_{\alpha_i}(x) = 0$ and therefore
    \[
        \sum_{y': y\sim y'} f_{\alpha_i}(y') m_y^1(y') - \sum_{x': x\sim x'} f_{\alpha_i}(x') m_x^1(x') = \Delta f_{\alpha_i}(y) + f_{\alpha_i}(y) - \Delta f_{\alpha_i}(x) = f_{\alpha_i}(y) = d(x,y).
    \]
    Thus, we obtain
    \[
        p(y,x) =m_y^1(x) \leq m_x^1(z) = p(x,z).
    \]
    On the other hand, we have 
    \[
        p(x,z)d(x,y) \leq p(x,z)d(x,z) = p(x,y)d(x,y),
    \]
    where we used $w(x,z)d(x,z) =w(x,y)d(x,y)$ for the last equality. By assumption, we have $\mu(x) \geq \mu(y)$ and therefore
    \[
        p(y,x) \leq p(x,z) \leq p(x,y) \leq p(y,x).
    \]
    We conclude $p(y,x) = p(x,y)= p(x,z)$, which implies $\mu(x) = \mu(y)$ and 
    \[
        p(x,y)d(x,z) = p(x,z)d(x,z) = p(x,y)d(x,y),
    \]
    leading to $d(x,y) = d(x,a_i^{-1}x)$. 
    Using $\mu(x) = \mu(y)$, we can interchange the roles of $x$ and $y$ to obtain $d(x,y) = d(y,a_iy)$. Thus, we have shown
    \[
        d(a_i^{-1}x,x) = d(x,a_ix) = d(a_ix, a_i^2x).
    \]
\end{proof}

\subsection{Characterization of bone-idle Ollivier-Betti sharp graphs}

In this subsection we will give a characterization of bone-idle Ollivier-Betti sharp graphs. Notably, under the assumption of Ollivier-Betti sharpness, bone-idleness is equivalent to non-negative curvature for the non-lazy random walk.

\begin{theorem}\label{th:OBSandBI}
    Let $G=(V,w,\mu,d)$ be a finite graph that satisfies $\sum_{y:x\sim y}w(x,y) =\mu(x)$ for every $x\in V$. The following are equivalent.
    \begin{itemize}
        \item[$(i)$] $G$ is Ollivier-Betti sharp and bone-idle.
        \item[$(ii)$] $G$ is Ollivier-Betti sharp and $\kappa_1(x,y) \geq 0$ for every edge $x\sim y$.
        \item[$(iii)$] $G$ is a weighted flat Torus, $\mu$ is constant, and $d(x, a_ix) = d(x',a_ix')$ for every $x,x' \in V$ and $i\in \{1,\dots,m\}$, where $(a_1,\dots,a_m)$ is the generating set from the Torus structure. 
    \end{itemize}
\end{theorem}

\begin{proof}
    $(i) \implies (iii)$ Theorem~\ref{th:CharacBS} implies that $G$ is a weighted, flat Torus. Let $i\in \{1,\dots,m\}$ and $x \in V$ be arbitrary. Choose $y_0\in V$ such that 
    \[
        d(y_0,a_iy_0) = \min_{x'\in V}d(x',a_ix').
    \]
    There exist $j_1, \dots, j_m\in \Z$ such that $x = a_1^{j_1}\ldots a_m^{j_m}y_0$. Define $y = a_1^{j_1} \ldots a_{i-1}^{j_{i-1}} a_{i+1}^{j_{i+1}} \ldots a_m^{j_m}y_0$. Using property $(iv)$ from the definition of a weighted, flat Torus, we obtain
    \[
        d(y,a_iy) = d(y_0,a_iy_0) = \min_{x'\in V}d(x',a_ix').
    \]
    By Lemma~\ref{lem:constdist}, we obtain
    \[
        d(a_i^{-1}y, y) = d(a_iy, a_i^{2}y) = d(y,a_iy) = \min_{x'\in V}d(x',a_ix').
    \]
    Iterating the lemma yields,
    \[
        \min_{x'\in V}d(x',a_ix') = d(y,a_iy) = d( a_i^{j_i}y, a_i^{j_i+1}y).
    \]
    Using $a_i^{j_i}y = x$, we conclude
    \[
        \min_{x'\in V}d(x',a_ix') = d(x,a_ix).
    \]
    Since $x\in V$ and $i\in\{1,\dots,m\}$ was chosen arbitrarily, we conclude
    \[
        d(x,a_ix) = d(x',a_ix'),
    \]
    for every $x,x'\in V$ and $i\in\{1,\dots,m\}$.
    It remains to show that $\mu$ is constant. To this end, let $x\sim a_ix$ be an arbitrary edge in $G$. Since $G$ is bone-idle, the equality
    \[
        \sum_{y: x \sim y} p(x,y) = 1 =  \sum_{y: a_ix \sim y} p(a_ix,y)
    \]
    holds. Using that $G$ is a weighted, flat Torus, this implies
    \begin{align*}
        \frac{1}{\mu(x)}\left(w(x,a_i^{-1}x) + w(x,a_ix)\right) &= p(x,a_i^{-1}x) + p(x,a_ix) \\
        &=p(a_ix,x) + p(a_ix,a_i^2x) \\
        &= \frac{1}{\mu(a_ix)}\left(w(a_ix,x) + w(a_ix,a_i^2x)\right).
    \end{align*}
    Since $d(x,a_i^{-1}x) = d(x,a_ix) = d(a_ix,a_i^2x)$, we obtain
    $w(x,a_i^{-1}x) = w(x,a_ix) = w(a_ix,a_i^2x)$ by property $(ii)$ from the definition of a weighted, flat Torus. Therefore, we conclude $\mu(x) = \mu(a_ix)$. Since the edge $x \sim a_ix$ was chosen arbitrarily, we conclude that $\mu$ must be constant.

    $(iii) \implies (ii)$ Since $G$ is a weighted, flat Torus, it follows from Theorem~\ref{th:CharacBS} that $G$ is Ollivier-Betti sharp. 
    Let $x_0\sim y_0$ be an arbitrary edge in $G$. We assume, without loss of generality, that $y_0 =a_ix_0$ for some $i\in \{1,\dots,m\}$. Define the transport plan $\rho: \ON(x_0) \times \ON(y_0) \to [0,1]$ transporting $m_{x_0}^1$ to $m_{y_0}^1$ by
    \[
        \rho(x,y) = \begin{cases}
            p(x_0,x) &: (x,y) \in \{(a_j^{\eps}x_0,a_j^{\eps}y_0): j \leq m, \eps \in \{-1,1\}\},\\
            0 &: \mbox{otherwise}.
            \end{cases}
    \]
    Since $p(x_0, a_j^{\eps}x_0) = p(y_0, a_j^{\eps}y_0)$ for $j\neq i$ and $\eps \in \{-1,1\}$, and since $p(x_0,a_i^{-1}x_0) = p(y_0,x_0) =p(x_0,y_0) = p(y_0,a_iy_0)$, we conclude that $\rho$ is indeed a transport plan, transporting $m_{x_0}^1$ to $m_{y_0}^1$. Hence,
    \begin{align*}
        W(m_{x_0}^1, m_{y_0}^1) &\leq \sum_{x,y\in V}\rho(x,y)d(x,y)\\
        &= \sum_{\substack{j\leq m\\ \eps\in \{-1,1\}}}p(x_0,a_j^{\eps}x_0)d(a_j^{\eps}x_0,a_j^{\eps}y_0) \\
        &= d(x_0,y_0) \sum_{y: x_0\sim y} p(x_0,y)\\
        &= d(x_0,y_0),
    \end{align*}
    where we used $d(a_j^{\eps}x_0,a_j^{\eps}y_0) = d(x_0,y_0)$ for every $j\neq i$ and $\eps\in \{-1,1\}$, as well as $d(a_i^{-1}x_0,x_0) = d(x_0,y_0) = d(y_0,a_iy_0)$. Hence,
    \[
        \kappa_1(x_0,y_0) = 1- \frac{W(m^{1}_{x_0}, m^{1}_{y_0})}{d(x_0,y_0)} \geq 0.
    \]

    $(ii) \implies (i)$ Since $G$ is Ollivier-Betti sharp, there exists a basis $(\alpha_1, \dots, \alpha_m)$, with $m = \beta_1(M_2(G))$, of $H_1(M_2(G))$ satisfying the properties from Lemma~\ref{lem:alphauni}. 
    
    Let $x\sim y$ be an arbitrary edge in $G$. There exists $i\in \{1,\dots,m\}$ such that, without loss of generality,
    \[
        \alpha_i(x,y)  = d(x,y).
    \]
    Denote by $f_{\alpha_i}$ the associated 1-Lipschitz function on the universal cover that satisfies $f_{\alpha_i}(x) = 0$. Recall that the function $f_{\alpha_i}$ is harmonic since $\alpha_i \in H_1(M_2(G))$. Furthermore,
    \[
        \nabla_{yx}f_{\alpha_i} = \frac{f_{\alpha_i}(y) - f_{\alpha_i}(x)}{d(x,y)} = 1.
    \] 
    Thus, we obtain
    \[
    \kappa(x,y) = \inf_{\substack{\nabla_{yx}f = 1\\ \lVert \nabla f \rVert_\infty = 1}} \nabla_{xy} \Delta f \leq \nabla_{xy} \Delta f_{\alpha_i} = 0.
    \]
    On the other hand, $G$ is Ollivier-Betti sharp, which implies that $G$ has non-negative Ollivier curvature. Therefore, we conclude $\kappa(x,y) = 0$. Using $\kappa_1(x,y) \leq \kappa(x,y) = 0$ and our assumption, we obtain 
    \[
        \kappa_1(x,y) = \kappa(x,y)= 0. 
    \]
    Since the function $\kappa_{\eps}(x,y)$ is concave in $\eps$, it follows that $\kappa_{\eps}(x,y) = 0$ for every  $\eps \in [0,1]$. Hence, the graph is bone-idle. This concludes the proof.
\end{proof}

As an immediate consequence we obtain that for a graph equipped with the combinatorial path distance, in the case of the normalized Laplacian, Ollivier-Betti sharp implies bone-idleness.

\begin{corollary}\label{Cor:OBSimpliesBI}
    Let $G=(V,w,\mu,d)$ be a finite graph, where $d$ denotes the combinatorial path distance. In the case of the normalized Laplacian, that is, when $w:V\times V \to \{0,1\}$ and $\mu(x) = \deg(x)$, the following implication holds true. If $G$ is Ollivier-Betti sharp, then $G$ is bone-idle.
\end{corollary}

\begin{remark}
    The converse implication does not hold true. A counterexample illustrating this is provided in Subsection~\ref{sec:OBSnotBI}.
\end{remark}

\subsection{Betti number estimates allowing some negative Ollivier curvature
}

In this section, we study graphs with non-negative Ollivier curvature outside a finite subset. Let $G=(V,w,\mu,d)$ be a weighted graph and let $W\subseteq V$ be a finite subset of $V$ such that $\kappa(x,y) \geq 0$ if $x,y \in V\setminus W$. Denote by
\[
    \eW = \{ x \sim y: x \in W\}
\]
the set of edges with at least one endpoint in $W$. In this setting, we obtain the following Betti-number estimate.

\begin{theorem}\label{th:BettiNumEstSomeNegCurv}
    Let $G=(V,w,\mu,d)$ be a finite graph and let $W \subseteq V$ be a non-empty set such that $\kappa(x,y) \geq 0$ if $x,y\in V\setminus W$. Then
    \[
        \beta_1(M_2(G)) \leq \vert \eW \vert.
    \]
\end{theorem}

\begin{proof}
    We will show that if $\alpha \in H_1(M_2(G))$ satisfies $\alpha |_{\eW} = 0$, then it follows that $\alpha = 0$. To this end, assume there exists a non-zero $\alpha \in H_1(M_2(G))$ satisfying $\alpha |_{\eW} = 0$. We assume without loss of generality that $\lVert \alpha \rVert_{\infty} = 1$. Therefore, the associated function $f_{\alpha}$ on the universal cover is 1-Lipschitz. Choose $\widetilde{x}_0, \widetilde{y}_0 \in \widetilde{X}_0$ such that
    \[
        f_{\alpha}(\widetilde{x}_0) - f_{\alpha}(\widetilde{y}_0) = \widetilde{d}(\widetilde{x}_0, \widetilde{y}_0). 
    \]
    According to Lemma~\ref{lem:LipFun}, we can construct two paths, starting from $\widetilde{x}_0$ and $\widetilde{y}_0$, respectively, as described below.
    \begin{algorithmic}
        \State $k=0$
        \While{$\nexists\widetilde{w} \in \pr^{-1}(W)$ on a geodesic from $\widetilde{x}_k$ to $\widetilde{y}_k$}
        \If{$\max_{i \leq k} \widetilde{d}(\widetilde{x}_i,\widetilde{y}_i) < 3\diam(G) \cdot \frac{d_{\max}}{d_{\min}}$} 
            \State Choose  $\widetilde{x}_{k+1} \in \CN(\widetilde{x}_k)$ and $\widetilde{y}_{k+1} \in \CN(\widetilde{y}_k)$ such that 
            \[
            f(\widetilde{x}_{k+1}) - f(\widetilde{y}_{k+1}) = \widetilde{d}(\widetilde{x}_{k+1}, \widetilde{y}_{k+1}) > \widetilde{d}(\widetilde{x}_{k}, \widetilde{y}_{k})
            \]
        \Else
           \State Choose $\widetilde{y}_{k+1}$ on a geodesic from $\widetilde{y}_{k}$ to $\pr^{-1}(W)$ and $\widetilde{x}_{k+1}$ such that 
           \[
            f(\widetilde{x}_{k+1}) - f(\widetilde{y}_{k+1}) = \widetilde{d}(\widetilde{x}_{k+1}, \widetilde{y}_{k+1})
           \]
        \EndIf 
        \State $k \gets k+1$
        \EndWhile
    \end{algorithmic}
    where $d_{\max} = \max_{x\sim y} d(x,y)$ and $d_{\min} = \min_{x\sim y} d(x,y)$. Observe that, by construction, the algorithm must terminate after a finite number of steps and denote by $\widetilde{x}_n$ and $\widetilde{y}_n$ the final vertices of the paths. We claim that $\widetilde{x}_n  \neq \widetilde{y}_n$. Suppose not, then there exists $k \leq n$ such that 
    \[
        \widetilde{d}(\widetilde{x}_k,\widetilde{y}_k) > 3\diam(G) \cdot \frac{d_{\max}}{d_{\min}}.
    \]
    Since $\widetilde{d}(\widetilde{x}, q^{-1}(W)) \leq \diam(G)$ for every $\widetilde{x} \in \widetilde{X}_0$, we conclude
    \[
        \widetilde{d}(\widetilde{y}_k, \widetilde{y}_n) \leq \diam(G).
    \]
    Thus, we obtain
    \[
        \widetilde{d}(\widetilde{x}_k, \widetilde{x}_n) \leq \frac{d_{\max}}{d_{\min}}\diam (G).
    \]
    Using the triangle inequality, we conclude 
    \[
        0 = \widetilde{d}(\widetilde{x}_n, \widetilde{y}_n) \geq \widetilde{d}(\widetilde{x}_k, \widetilde{y}_k) - \widetilde{d}(\widetilde{y}_k, \widetilde{y}_n) - \widetilde{d}(\widetilde{x}_k, \widetilde{x}_n) > 0.
    \]
    Hence, our assumption leads to a contradiction, implying that $\widetilde{x}_n  \neq \widetilde{y}_n$ must hold and therefore
    \[
        f_{\alpha}(\widetilde{x}_n) - f_{\alpha}(\widetilde{y}_n) = d(\widetilde{x}_n, \widetilde{y}_n) > 0.
    \]
    Thus, there exists $\widetilde{w} \in \pr^{-1}(W)$ on a geodesic from  $\widetilde{x}_n$ to $\widetilde{y}_n$ and a neighboring vertex $\widetilde{z}$ of $\widetilde{w}$ on the geodesic such that
    \[
        \vert f_{\alpha}(\widetilde{w}) - f_{\alpha}(\widetilde{z}) \vert = \vert \alpha(\pr(\widetilde{w}), \pr(\widetilde{z})) \vert > 0.
    \]
    This contradicts $\alpha |_{\eW} = 0$ since $q(\widetilde{w}) \in W$.

    Therefore, we have established the equality $\alpha_1 = \alpha_2$ for any $\alpha_1, \alpha_2 \in H_1(M_2(G))$ satisfying $\alpha_1 |_{\eW} = \alpha_2|_{\eW}$, which implies
    \[
        \beta_1(M_2(G)) \leq \vert \eW \vert.
    \]
    This completes the proof.
\end{proof}

\section{Examples}\label{sec:Examples}

The rigidity result in Theorem~\ref{th:CharacBS} refers to the maximal and not the minimal vertex degree. The reason is that there is a variety of examples for which $\beta_1 = \deg_{\min}/2$ with quite different combinatorial structures which seem out of reach to characterize. In the following we will give two such examples. 

Before that, we show that for a potentially non-reversible Markov chain on a cycle of at least five vertices, there always exists a path metric with constant Ollivier curvature. This metric has non-negative Ollivier curvature and the curvature is equal to zero if and only if the graph is Ollivier-Betti sharp.

\subsection{Markov chain on a cycle}
\label{sec:CycleGraphs}

We define the cycle graph $C_n=(V=\{0, \ldots, n-1\},p)$ with an arbitrary Markov kernel $p:V^2 \to [0,\infty)$, satisfying $p(i,j) > 0$ if and only if $j \in \{(i+1) \mod n, (i-1) \mod n\}$.
Particularly, we allow the Markov chain to be non-reversible. For ease of notation, in the following, all indices will be considered modulo $n$, though this will not be explicitly stated again.

While the related case of birth death chains and their curvature is well investigated in the literature
\cite{hua2023ricci,smith2017cutoff,
johnson2017discrete,joulin2007poisson,
mielke2013geodesic}, the case of arbitrarily weighted cycles still provides some interesting questions which we want to address now.
Specifically, the aim of this subsection is to show the following:
\begin{itemize}
\item There exists a constant curvature metric which is unique up to scaling. 
\item The corresponding curvature is non-negative.
\item Constant zero curvature is equivalent to Ollivier Betti sharpness.
\end{itemize}
A key tool for establishing these results is the Ollivier Ricci flow 
which is given by
\[
d_{k+1}(x,y) = d_k(x,y) - \alpha \kappa_k(x,y) d_k(x,y)
\]
for $x\sim y$, and this metric is extended to non-neighbors by taking path distance. Here, $d_k$ is the metric after the $k-$the step of the Ricci flow, and $\kappa_k$ is the curvature corresponding to $d_k$. Moreover, $0<\alpha<1$ is a constant.
We will use the following convergence result from 
\cite[Theorem~4]{li2024convergence}.
\begin{theorem}\label{thm:RicciFlow}
Let the Ollivier Ricci flow be as above. Then, 
\begin{enumerate}[(i)]
\item Assume $\frac{\max_{x,y} d_k(x,y)}{\min_{x\neq y}d_k(x,y)}$ is bounded in $k$. Then,
\[
\frac{d_k}{\max_{x,y}d_k(x,y)} \stackrel{k\to \infty}{\longrightarrow} d_\infty
\]
for some constant curvature metric $d_\infty$.
\item The curvature $\kappa_k$ is bounded in $k$.
\end{enumerate}
\end{theorem}
We remark that we use a slightly more general version than in \cite{li2024convergence} where only the reversible case was considered, however the proof carries over verbatim to the non-reversible case.
Furthermore, the formulation is slightly different, as in \cite{li2024convergence}, surgeries are used to enforce convergence if the distance quotients are unbounded.
Also, statement $(ii)$ is not mentioned explicitly in \cite{li2024convergence}, but follows easily from the proof of \cite[Theorem~2]{li2024convergence}, where it is shown that the maximum curvature is decreasing in $k$ and the minimum curvature is increasing in $k$.

To prove the existence of a constant curvature metric, we show that the assumption of $(i)$ in Theorem~\ref{thm:RicciFlow} is satisfied on cycles of length at least five.

\begin{lemma}\label{lem:RicciFlowNoSurgery}
On the cycle $C_n$ for $n \geq 5$, the distance quotient $\max_{x,y} d_k(x,y)/\min_{x\neq y} d_k(x,y)$ is bounded in $k$.
\end{lemma}

\begin{proof}
Suppose not. W.l.o.g. and by taking subsequence, $d_k(x,x-1)/d_k(x,x+1) \to \infty$ for $k \to \infty$ for some fixed vertex $x$.
As $\kappa_k$ is bounded by Theorem~\ref{thm:RicciFlow}(ii), we infer that for all vertices $i$, the distance quotient
\[
\frac{\min(d_k(i-1,i),d_k(i-1,i+2))}{d_k(i,i+1)}
\]
is bounded in $k$.
Particularly, $\frac{\min(d_k(x-1,x+2))}{d_k(x,x+1)}$ is bounded in $k$. As $d_k(x,x-1)/d_k(x,x+1) \to \infty$, we see that all the quotients $d_k(y,y+1)/d_k(x,x+1)$ are bounded in $k$ for $y \notin \{x-1,x+1\}$.
We want to show boundedness also for $y=x+1$.
To this end, as $n \geq 5$, we notice that $d_k(x,x-1)/d_k(x-2,x-1) \to \infty$ for $k\to \infty$.
Now, we are in the same situation as in the beginning, and by relabeling vertices and doing the same argument again, we obtain
that $d_k(y,y+1)/d_k(x-2,x-1)$ is bounded in $k$ for all $y \notin \{x-3,x-1\}$.
Specifically, $d_k(x+1,x+2)/d_k(x-2,x-1)$ is bounded as $n \geq 5$, and thus also $d_k(x+1,x+2)/d_k(x,x+1)$ is bounded in $k$.
In summary, we see that $d_k(y,y+1)/d_k(x,x+1)$ are bounded in $k$ for $y \notin \{x-1\}$. This however implies by the triangle inequality that  $d_k(x-1,x)/d_k(x,x+1)$ is also bounded in $k$ which is a contradiction and finishing the proof.
\end{proof}

We are now prepared to prove the existence of a constant curvature metric on cycles of length at least five. Furthermore, we show that this constant curvature metric is unique up to scaling, and the corresponding curvature is non-negative.

\begin{theorem}\label{th:UniqueConstCurvMetr}
On the cycle $C_n$ for $n \geq 5$, there exists a  constant curvature metric which is unique up to multiplication by a positive constant. Moreover, this constant curvature metric has non-negative Ollivier curvature.
\end{theorem}

\begin{proof}
The existence follows directly from Lemma~\ref{lem:RicciFlowNoSurgery} and Theorem~\ref{thm:RicciFlow}.
Suppose this constant curvature metric has constant curvature $-K$ for some $K>0$.
Taking the universal cover of the graph (without any 2-cells) gives a weighted graph on the line with curvature upper bounded by $-K$. Considering the Busemann function 
$$f:= \lim_{m\to -\infty}\widetilde d(m,\cdot)-\widetilde d(m,0),$$
with $\widetilde d$ the inherited distance on the universal cover,
 we see that this function is an optimizer in the Lipschitz characterization of the curvature and thus,
\[
-K \geq \widetilde \kappa(x,y) = \frac {\widetilde \Delta f(x)-\widetilde \Delta f(y)}{\widetilde d(x,y)}
\]
where $\widetilde \kappa$ is the curvature, and $\widetilde \Delta$ the Laplacian  in the universal cover. But this is a contradiction as $\widetilde\Delta f$ is periodic, and therefore, $\widetilde \Delta f(x)$ would have to tend to infinity as $x \to \infty$. Hence, there exists a $K > 0$ such that $\kappa(x,y) = K$ for every edge $x\sim y$.

It remains to prove the uniqueness. To this end, assume there exists another constant curvature metric $d'$ with constant curvature $K'$. We denote by $d_n'$ the metric obtained after the $k-$the step of the 
Ricci flow. We first show that $K = K'$  must hold. Assume this is not the case and $K > K'$. We may assume without loss of generality, that $d_0 \geq d_0'$. Hence $d_n \geq d_n'$ must hold for every $n \in \N$. Using that $d_n = (1-\alpha K)^nd_0$, we obtain
\begin{equation*}
    \frac{d_n}{d_n'} = \frac{d_0}{d_0'} \left(\frac{1-\alpha K}{1-\alpha K'}\right)^n.
\end{equation*}
Using that $K>K'$, we conclude that there exists an $n \in N$ such that $d_n < d_n'$, a contradiction. The case $K < K'$ can be proven in a similar manner, and we conclude that $K=K'$ must hold.

Next, we prove that $d = d'$ holds true. After scaling by an appropriate constant, we may assume, without loss of generality, that $d(x_1,x_2) = d'(x_1,x_2)$ and $d \leq d'$. Let $\rho'$ be an optimal transport plan, such that
\begin{equation*}
    \kappa'(x_1,x_2) = \sum_{\substack{x \in \CN(x_1) \\ y \in \CN(x_2)}}\rho'(x,y) \left[1 - \frac{d'(x,y)}{d'(x_1,x_2)}\right].
\end{equation*}
If mass is transported over the edge $x_2 \sim x_3$, we conclude that $d(x_2,x_3) = d'(x_2,x_3)$ must hold, as otherwise
\begin{equation*}
    \kappa'(x_1,x_2) = \sum_{\substack{x \in \CN(x_1) \\ y \in \CN(x_2)}}\rho'(x,y) \left[1 - \frac{d'(x,y)}{d'(x_1,x_2)}\right] <  \sum_{\substack{x \in \CN(x_1) \\ y \in \CN(x_2)}}\rho'(x,y) \left[1 - \frac{d(x,y)}{d(x_1,x_2)}\right] \leq \kappa(x_1,x_2).
\end{equation*}
This contradicts $\kappa(x_1,x_2) = \kappa'(x_1,x_2)$. Hence, we may assume without loss of generality that mass is transported along the path $(x_0 \sim x_{n-1} \sim \ldots \sim x_3)$. Again, this implies $d(x_i, x_{(i+1) \mod n}) = d'(x_i,x_{(i+1) \mod n})$ for every $i \in \{3,\ldots,n-1\}$.

By applying the same arguments to the edge $x_0 \sim x_{n-1}$, we obtain that $d(x_2,x_3) = d'(x_2,x_3)$ or $d(x_0,x_1)=d'(x_0,x_1)$. By symmetry, we may assume without loss of generality that $d(x_2,x_3) = d'(x_2,x_3)$. It remains to prove that $d(x_0,x_1) = d'(x_0,x_1)$. Assume, for the sake of contradiction, that $d(x_0,x_1) < d'(x_0,x_1)$. Let $\rho$ be an optimal transport plan, such that
\begin{equation*}
    \kappa(x_0,x_1) = \sum_{\substack{x \in \CN(x_0) \\ y \in \CN(x_1)}}\rho(x,y) \left[1 - \frac{d(x,y)}{d(x_0,x_1)}\right].
\end{equation*}
Using that $d(x_i, x_{i+1}) = d'(x_i, x_{i+1})$ for every $i \neq 0$, we conclude
\begin{equation*}
    \kappa(x_0,x_1) = \sum_{\substack{x \in \CN(x_0) \\ y \in \CN(x_1)}}\rho(x,y) \left[1 - \frac{d(x,y)}{d(x_0,x_1)}\right] <  \sum_{\substack{x \in \CN(x_1) \\ y \in \CN(x_2)}}\rho(x,y) \left[1 - \frac{d'(x,y)}{d'(x_0,x_1)}\right] \leq \kappa'(x_0,x_1),
\end{equation*}
contradicting $\kappa(x_0,x_1) = \kappa'(x_0,x_1)$. Hence, $d(x_0,x_1) = d'(x_0,x_1)$ must hold, completing the proof.
\end{proof}

Next, we prove the main result of this subsection. Namely, the unique constant curvature metric on a cycle of length at most five is Ollivier-Betti sharp if and only if it has zero Ollivier curvature.

\begin{theorem}
    Let $C_n$ be a weighted cycle of length $n \geq 5$, equipped with the unique constant curvature metric. Then, the following are equivalent:
    \begin{itemize}
        \item[$(i)$] $C_n$ is Ollivier-Betti sharp.
        \item[$(ii)$] $C_n$ has zero Ollivier curvature.
        \item[$(iii)$] $M_2(C_n)$ does not contain any 2-cells, i.e., $X_2 = \emptyset$.
    \end{itemize}
\end{theorem}

\begin{proof}
    $"(i) \implies (ii)"$ We prove this by contraposition. To this end, assume the constant curvature metric has strictly positive curvature. By Theorem~\ref{th:Betti_vanish_non_rev}, the first Betti number vanishes and therefore $C_n$ is not Ollivier-Betti sharp.

    $"(ii) \implies (iii)"$ Denote by $d$ the constant curvature metric. Suppose the constant curvature metric has curvature zero, and there is a two-cell for which we will find a contradiction. Taking the universal cover of the graph without any 2-cells gives a weighted graph on the line. Since $X_2 \neq \emptyset$, there exists an edge $(x_i,x_{i+1})$ such that
    \[
        d(x_{i-1},x_i) + d(x_i,x_{i+1}) + d(x_{i+1},x_{i+2}) > d(x_{i+2},x_{i+3}) + \ldots + d(x_{i-2},x_{i-1}).
    \]
    This implies that in the universal cover of the graph without any 2-cells, the images of the edge $(x_i,x_{i+1})$ have negative Ollivier curvature.
    Moreover the universal cover has non-positive curvature.
    We recall that all objects on the universal cover are indicated by a tilde.
    As the Busemann function $f:=\lim_{m \to -\infty} \widetilde d(m,\cdot) - \widetilde d(m,0)$ on the universal cover is an optimizer of the Lipschitz characterization of the curvature, we have
    \[
        \widetilde \kappa(x,y) = \frac {\widetilde \Delta f(x)-\widetilde \Delta f(y)}{\widetilde d(x,y)}.
    \]
    Particularly, for $j>i$,
    \[
        \widetilde \Delta f(i) - \widetilde \Delta f(j) = \sum_{k=i}^{j-1} \widetilde d(k,k+1) \widetilde \kappa(k,k+1). 
    \]
    However, this sum converges to minus infinity, as $j \to \infty$ and $\widetilde \Delta f(j)$ stays bounded as $j \to \infty$. This is a contradiction proving $"(ii) \implies (iii)"$. 

    $"(iii) \implies (i)"$ Assume $X_2 = \emptyset$. By solving a homogeneous system of $n$ linear equations with $n+1$ unknowns, we obtain a non-zero $\alpha \in C(X_1)$ and $c \in \R$ such that $\dst \alpha =c$. Since $X_2 = \emptyset$, we conclude that $\alpha \in H_1(M_2(C_n))$ and therefore $\beta_1(M_2(C_n)) \geq 1$. Using Theorem~\ref{th:Main_res_non_rev}, we obtain $\beta_1(M_2(C_n)) = 1$. According to Theorem~\ref{th:UniqueConstCurvMetr}, the cycle $C_n$ has non-negative Ollivier curvature. Hence, we conclude that $C_n$ is Ollivier-Betti sharp. This concludes the proof.
\end{proof}

\subsection{Rope ladder graph}

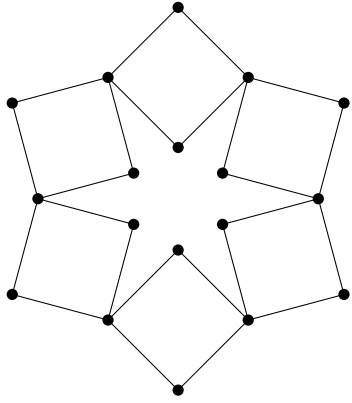
\begin{figure}
    \begin{center}
        \begin{tikzpicture}[x=0.75pt,y=0.75pt,yscale=-1,xscale=1, vertex/.style={
            shape=circle, fill=black, inner sep=1.5pt	
        }]
        \draw   (315,14.64) -- (350.36,50) -- (315,85.36) -- (279.64,50) -- cycle ;
        \node[vertex] (0) at (315,14.64) {};
        \node[vertex] (1) at (350.36,50) {};
        \node[vertex] (2) at (315,85.36) {};
        \node[vertex] (3) at (279.64,50) {};
        \draw   (398.65,62.94) -- (385.71,111.24) -- (337.41,98.3) -- (350.36,50) -- cycle ;
        \node[vertex] (4) at (398.65,62.94) {};
        \node[vertex] (5) at (385.71,111.24) {};
        \node[vertex] (6) at (337.41,98.3) {};
        \draw   (398.65,159.53) -- (350.36,172.47) -- (337.41,124.18) -- (385.71,111.24) -- cycle ;
        \node[vertex] (7) at (398.65,159.53) {};
        \node[vertex] (8) at (350.36,172.47) {};
        \node[vertex] (9) at (337.41,124.18) {};
        \draw   (315,207.83) -- (279.64,172.47) -- (315,137.12) -- (350.36,172.47) -- cycle ;
        \node[vertex] (10) at (315,207.83) {};
        \node[vertex] (11) at (279.64,172.47) {};
        \node[vertex] (12) at (315,137.12)  {};
        \draw   (231.35,159.53) -- (244.29,111.24) -- (292.59,124.18) -- (279.64,172.47) -- cycle ;
        \node[vertex] (13) at (231.35,159.53) {};
        \node[vertex] (14) at (244.29,111.24) {};
        \node[vertex] (15) at (292.59,124.18) {};
        \draw   (231.35,62.94) -- (279.64,50) -- (292.59,98.3) -- (244.29,111.24) -- cycle ;
        \node[vertex] (16) at (231.35,62.94) {};
        \node[vertex] (18) at (292.59,98.3) {};
        \end{tikzpicture}
        \caption{Illustration of the rope ladder graph for $n=6$}
        \label{fig:rope_ladder_graph}  
    \end{center}
\end{figure}

Take $n$ different 4-cycles and denote them by $(x^k_0\sim x^k_1\sim x^k_2\sim x^k_3)$ for $k=0,\dots,n-1$. The rope ladder graph is obtained by identifying $x^k_2$ with $x^{(k+1) \mod n}_0$ for $k = 0, \dots, n-1$. We set the vertex weight $\mu\equiv 1$, $w(x,y) = 1$ if $x\sim y$ in some 4-cycle and $w(x,y) = 0$ otherwise. An illustration of the resulting graph for $n=6$ can be found in Figure~\ref{fig:rope_ladder_graph}. Denote the resulting graph by $G=(V,w,\mu,d)$ where $d$ denotes the combinatorial path distance. It is an easy exercise to check that $\kappa(x,y) = 0$ for every edge $x\sim y$, if $n \geq 3$.

This graph satisfies $\deg_{\min} = 2$ and $\deg_{\max} = 4$. Furthermore, for $n \geq 3$, the first Betti number of the graph is equal to one. This can be seen as follows. According to Theorem~\ref{th:MainRes}, we have 
\[
    \beta_1(M_2(G)) \leq \frac{\deg_{\min}}{2} = 1.
\] 
On the other hand, define $\alpha \in C(X_1)$ by 
\[
    \alpha(x,y) = \begin{cases}
        1 &: (x,y) \in \{(x_3^k, x_0^k), (x_1^k, x_0^k): k=0,\dots,n-1\},\\
        -1 &:(x,y) \in \{(x_3^k, x_2^k), (x_1^k, x_2^k): k=0,\dots,n-1\}.
    \end{cases}
\]
The 1-form $\alpha$ satisfies 
\begin{align*}
    \delta^*\alpha(x_2^k) &= \sum_{y:x_2^k\sim y} \frac{w(x_2^k,y)}{\mu(x_2^k)}\alpha(x_2^k,y) \\
    &= \alpha(x_2^k, x_3^k) + \alpha(x_2^k, x_1^k) + \alpha(x_0^{(k+1)\mod n}, x_3^{(k+1)\mod n}) + \alpha(x_0^{(k+1)\mod n}, x_1^{(k+1)\mod n})\\
    &= 0,
\end{align*}
for every $k=0,\dots,n-1$. Since $x^k_2 = x^{(k+1) \mod n}_0$, we obtain $\delta^*\alpha(x_0^k) = 0$ for every $k=0,\dots,n-1$. Furthermore, we have
\begin{align*}
    \delta^*\alpha(x_1^k) &= \sum_{y:x_1^k\sim y} \frac{w(x_1^k,y)}{\mu(x_1^k)}\alpha(x_1^k,y) \\
    &= \alpha(x_1^k, x_0^k) + \alpha(x_1^k, x_2^k)\\
    &= 0,
\end{align*}
for every $k=0,\dots,n-1$. Similarly, one can show that $\delta^*\alpha(x_3^k) = 0$. Finally, observe that the only cycles of length less than or equal to five are the 4-cycles $(x^k_0\sim x^k_1\sim x^k_2\sim x^k_3)$ for $k=0,\dots,n-1$ and 
\[
    \delta\alpha((x^k_0\sim x^k_1\sim x^k_2\sim x^k_3)) = \alpha(x_0^k,x_1^k) + \alpha (x_1^k, x_2^k) + \alpha(x_2^k, x_3^k) + \alpha(x_3^k, x_0^k) = 0.
\]
Hence, $\alpha \in H_1(M_2(G))$ and therefore $\beta_1(M_2(G)) = 1$. Thus, $G$ is an example for which $\beta_1(M_2(G)) = \deg_{\min}/2$, while not having the combinatorial structure of a flat Torus.

\subsection{Zero-range process}

\begin{figure}
    \begin{center}     
        \begin{tikzpicture}[x=0.75pt,y=0.75pt,yscale=-1,xscale=1,vertex/.style={
            shape=circle, fill=black, inner sep=1.5pt	
        }]

        \draw   (330,96.64) -- (365.36,132) -- (330,167.36) -- (294.64,132) -- cycle ;
        \node[vertex] (0) at (330,96.64) {};
        \node[vertex] (1) at (365.36,132) {};
        \node[vertex] (2) at (330,167.36) {};
        \node[vertex] (3) at (294.64,132) {};
        \draw   (294.64,132) -- (330,167.36) -- (294.64,202.71) -- (259.29,167.36) -- cycle ;
        \node[vertex] (4) at (294.64,202.71) {};
        \node[vertex] (5) at (259.29,167.36) {};
        \draw   (365.36,132) -- (400.71,167.36) -- (365.36,202.71) -- (330,167.36) -- cycle ;
        \node[vertex] (6) at (400.71,167.36) {};
        \node[vertex] (7) at (365.36,202.71) {};
        \draw   (330,167.36) -- (365.36,202.71) -- (330,238.07) -- (294.64,202.71) -- cycle ;
        \node[vertex] (8) at (330,238.07) {};
        \draw   (259.29,167.36) -- (294.64,202.71) -- (259.29,238.07) -- (223.93,202.71) -- cycle ;
        \node[vertex] (9) at (259.29,238.07) {};
        \node[vertex] (10) at (223.93,202.71) {};
        \draw   (294.64,202.71) -- (330,238.07) -- (294.64,273.42) -- (259.29,238.07) -- cycle ;
        \node[vertex] (11) at (294.64,273.42) {};
        \draw   (400.71,167.36) -- (436.07,202.71) -- (400.71,238.07) -- (365.36,202.71) -- cycle ;
        \node[vertex] (12) at (436.07,202.71) {};
        \node[vertex] (13) at (400.71,238.07) {};
        \draw   (365.36,202.71) -- (400.71,238.07) -- (365.36,273.42) -- (330,238.07) -- cycle ;
        \node[vertex] (14) at (365.36,273.42) {};
        \draw   (436.07,202.71) -- (471.42,238.07) -- (436.07,273.42) -- (400.71,238.07) -- cycle ;
        \node[vertex] (15) at (471.42,238.07) {};
        \node[vertex] (16) at (436.07,273.42) {};
        \draw   (223.93,202.71) -- (259.29,238.07) -- (223.93,273.42) -- (188.58,238.07) -- cycle ;
        \node[vertex] (17) at (223.93,273.42)  {};
        \node[vertex] (18) at (188.58,238.07) {};

        \draw    (471.42,238.07) -- (506.78,273.42) ;
        \node[vertex] (19) at (506.78,273.42)  {};
        \draw    (188.58,238.07) -- (153.22,273.42) ;
        \node[vertex] (20) at (153.22,273.42)  {};

        \draw    (153.22,273.42) .. controls (196.5,147) and (196.5,146) .. (330,96.64) ;
        \draw    (330,96.64) .. controls (461.5,145) and (462.5,145) .. (506.78,273.42) ;
        \draw    (188.58,238.07) .. controls (288.5,2) and (289.5,1) .. (365.36,132) ;
        \draw    (223.93,202.71) .. controls (309.5,1) and (309.5,0) .. (400.71,167.36) ;
        \draw    (294.64,132) .. controls (370.5,0) and (370.5,1) .. (471.42,238.07) ;
        \draw    (259.29,167.36) .. controls (351.5,1) and (351.5,1) .. (436.07,202.71) ;
        \end{tikzpicture}
        \caption{Illustration of the zero-range process graph for $l=6$ and $M=2$}
        \label{fig:zero_range_process} 
    \end{center}
\end{figure}
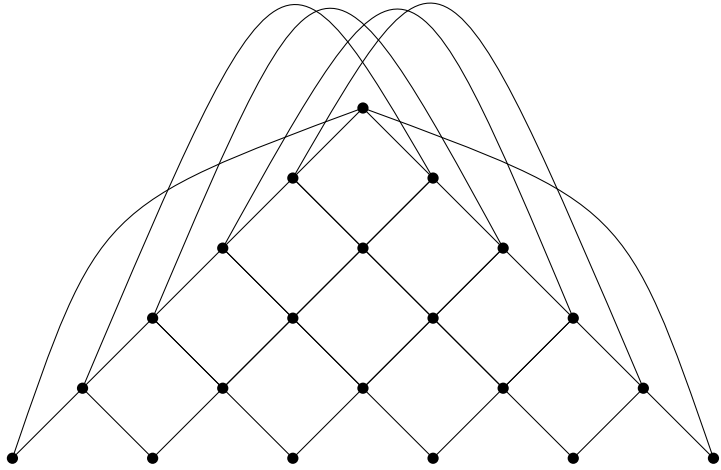

The zero-range process is a stochastic model of interacting particles. It describes the dynamics of particles hopping between the vertices of a graph, with the hopping rates depending only on the number of particles at the departure site. We discuss the case where the underlying graph is $C_l=(x_0 \sim x_1 \sim \dots \sim x_{l-1})$, the cycle of length $l$. If we denote by $M$ the total number of particles, the possible configurations of the system are given by
\[
    V = \left\{n \in \N_0^{C_n} : \sum_{k=0}^{l-1}n(x_k) =M \right\},
\]
where $n(x_k)$ denotes the number of particles at vertex $x_k$. It is possible to transition from one configuration $n$ to another configuration $n'$ if and only if $n-n'= 1_x - 1_y$ for some adjacent vertices $x \sim y$ in $C_l$. In this case, we set $w(n,n') = 1$; otherwise, we set $w(n,n') = 0$. Furthermore, we set $\mu \equiv 1$ and denote the resulting graph by $G=(V,w,\mu,d)$, where $d$ denotes the combinatorial path distance. An illustration of the resulting graph for $l=6$ and $M=2$ can be found in Figure~\ref{fig:zero_range_process}. 

In the following, we study the case where $l\geq6$ and $M=2$. It is an easy exercise to check that $\kappa(x,y) = 0$ for every edge $x\sim y$. The graph satisfies $\deg_{\min} = 2$ and $\deg_{\max} = 4$. Furthermore, the first Betti number of the graph is equal to one. This can be seen as follows. According to Theorem~\ref{th:MainRes}, we have
\[
    \beta_1(M_2(G)) \leq \frac{\deg_{\min}}{2} =1.
\]
Given a configuration $n$, we define
\[
    (p_1(n), p_2(n)) = (\min\{k: n(x_k) > 0\}, \max\{k: n(x_k) > 0\}),
\]
i.e., the locations of the two particles in the current configuration. Observe, that a configuration is fully determined by the quantities $p_1(n)$ and $p_2(n)$. Given a configuration $n$, we define the configuration $n^+$ by 
\[
    p_1(n^+) = p_1(n) \quad \mbox{ and } \quad p_2(n^+) = (p_2(n) + 1) \mod l.
\]
Similarly, define the configuration $n^{-}$ by
\[
    p_1(n^-) = p_1(n) \quad \mbox{ and } \quad p_2(n^-) = (p_2(n) - 1) \mod l,
\]
as well as the configuration $n_+$ by 
\[
    p_1(n_+) = (p_1(n) + 1) \mod l  \quad \mbox{ and } \quad p_2(n_+) = p_2(n),
\]
and the configuration $n_-$ by
\[
    p_1(n_+) = (p_1(n) - 1) \mod l  \quad \mbox{ and } \quad p_2(n_+) = p_2(n),
\]

Define $\alpha \in C(X_1)$ by 
\[
    \alpha(n,n') =  \begin{cases}
        1 &: n' \in \{n^+, n_+\},\\
        -1 &: n' \in \{n^-,n_-\}.
    \end{cases}
\]
If $n$ is a configuration with $p_1(n) = p_2(n)$, then
\begin{align*}
    \delta^*\alpha(n) &= \sum_{n': n \sim n'} \frac{w(n,n')}{\mu(n)} \alpha(n,n') \\
    &= \alpha(n, n^+) + \alpha(n,n^-) \\
    &= 0.
\end{align*}
On the other hand, if $p_1(n) < p_2(n)$, we obtain
\begin{align*}
    \delta^*\alpha(n) &= \sum_{n': n \sim n'} \frac{w(n,n')}{\mu(n)} \alpha(n,n') \\
    &= \alpha(n, n^+) + \alpha(n,n^-) + \alpha(n,n_+) + \alpha(n,n_-) \\
    &= 0.
\end{align*}
Since $l\geq 6$, the only cycles of length less than or equal to five are the 4-cycles of the form
\[
    (n \sim n^+ \sim (n^+)_- \sim ((n^+)_-)^-).
\]
For 4-cycles of this form we obtain
\begin{align*}
    \delta\alpha((n \sim n^+ \sim (n^+)_- \sim (n^+)_-)^-) &= \alpha(n, n^+) + \alpha(n^+, (n^+)_-) + \alpha((n^+)_-,  ((n^+)_-)^-) + \alpha(((n^+)_-)^-, n) \\
    &= 0.
\end{align*}
Hence, $\alpha \in H_1(M_2(G))$ and therefore $\beta_1(M_2(G)) = 1$. Thus, $G$ is another example for which $\beta_1(M_2(G)) = \deg_{\min}/2$, while not having the combinatorial structure of a flat Torus.

\subsection{Bone-idle does not imply Ollivier-Betti sharp}\label{sec:OBSnotBI}

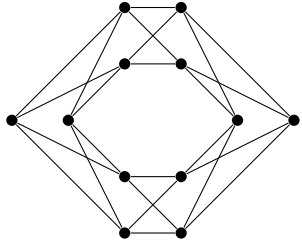
\begin{figure}
    \begin{center}
        \begin{tikzpicture}[x=1.5cm, y=1.5cm,
            vertex/.style={
                shape=circle, fill=black, inner sep=1.5pt	
            }
        ]
        
        \node[vertex] (1) at (0, 0) {};
        \node[vertex] (7) at (-0.5, 0) {};
        \node[vertex] (2) at (0.5, 0.5) {};
        \node[vertex] (8) at (0.5, 1) {};
        \node[vertex] (3) at (1, 0.5) {};
        \node[vertex] (9) at (1, 1) {};
        \node[vertex] (4) at (1.5, 0) {};
        \node[vertex] (10) at (2, 0) {};
        \node[vertex] (5) at (1, -0.5) {};
        \node[vertex] (11) at (1, -1) {};
        \node[vertex] (6) at (0.5, -0.5) {};
        \node[vertex] (12) at (0.5, -1) {};

        \draw (1) -- (2);
        \draw (2) -- (3);
        \draw (3) -- (4);
        \draw (4) -- (5);
        \draw (5) -- (6);
        \draw (6) -- (1);
        \draw (7) -- (8);
        \draw (8) -- (9);
        \draw (9) -- (10);
        \draw (10) -- (11);
        \draw (11) -- (12);
        \draw (12) -- (7);
        \draw (7) -- (2);
        \draw (7) -- (6);
        \draw (8) -- (1);
        \draw (8) -- (3);
        \draw (9) -- (2);
        \draw (9) -- (4);
        \draw (10) -- (3);
        \draw (10) -- (5);
        \draw (11) -- (4);
        \draw (11) -- (6);
        \draw (12) -- (5);
        \draw (12) -- (1);
    
        \end{tikzpicture}
        \caption{Illustration of the bone-idle graph $BI_{6}$}
        \label{fig:bone_idle_graph}
    \end{center}
\end{figure}

Let $G=(V,w,\mu,d)$ be a graph equipped with the combinatorial path distance $d:V\times V \to [0,\infty)$. In the case of the normalized Laplacian, that is, when $w:V\times V \to \{0,1\}$ and $\mu(x) = \deg(x)$, we have seen in Corollary~\ref{Cor:OBSimpliesBI} that an Ollivier-Betti sharp graph is bone-idle. In this section, we present a counterexample to the converse implication.

To this end, we take two different $n$-cycles $(x_0\sim x_1 \sim \dots \sim x_{n-1})$ and $(y_0\sim y_1 \sim \dots \sim y_{n-1})$. We then add the edges $x_k\sim y_{(k-1)\mod n}$ and $x_k\sim y_{(k+1)\mod n}$. Denote the resulting graphs by $BI_n$. An illustration of the graph $BI_6$ can be found in Figure~\ref{fig:bone_idle_graph}. As shown in \cite{hehl2024ollivier}, the resulting graph is bone-idle for $n \geq 6$. But $BI_n$ is not Ollivier-Betti sharp, since the first Betti number $\beta_1(M_2(BI_n)) = 1 < \deg_{\max}/2 = 2$.

\begin{figure}
    \begin{center}
        \begin{tikzpicture}[x=1.5cm, y=1.5cm,
            vertex/.style={
                shape=circle, fill=black, inner sep=1.5pt	
            }
        ]
        \node[vertex] (1) at (0, 0) {};
        \node[vertex] (2) at (1, 0) {};
        \node[vertex] (3) at (2, 0) {};
        \node[vertex] (4) at (3, 0) {};
        \node[vertex] (5) at (4, 0) {};

        \node[vertex] (6) at (0, 1) {};
        \node[vertex] (7) at (1, 1) {};
        \node[vertex] (8) at (2, 1) {};
        \node[vertex] (9) at (3, 1) {};
        \node[vertex] (10) at (4, 1) {};
        
        \node[vertex] (11) at (0, 2) {};
        \node[vertex] (12) at (1, 2) {};
        \node[vertex] (13) at (2, 2) {};
        \node[vertex] (14) at (3, 2) {};
        \node[vertex] (15) at (4, 2) {};

        \node[vertex] (16) at (0, 3) {};
        \node[vertex] (17) at (1, 3) {};
        \node[vertex] (18) at (2, 3) {};
        \node[vertex] (19) at (3, 3) {};
        \node[vertex] (20) at (4, 3) {};

        \node[vertex] (21) at (0, 4) {};
        \node[vertex] (22) at (1, 4) {};
        \node[vertex] (23) at (2, 4) {};
        \node[vertex] (24) at (3, 4) {};
        \node[vertex] (25) at (4, 4) {};

        \draw (1) -- (2);
        \draw (2) -- (3);
        \draw (3) -- (4);
        \draw (4) -- (5);

        \draw (6) -- (7);
        \draw (7) -- (8);
        \draw (8) -- (9);
        \draw (9) -- (10);

        \draw (1) -- (7);
        \draw (2) -- (6);
        \draw (3) -- (9);
        \draw (4) -- (8);

        \draw (1) -- (6);
        \draw (2) -- (7);
        \draw (3) -- (8);
        \draw (4) -- (9);
        \draw (5) -- (10);

        \draw (7) -- (13);
        \draw (8) -- (12);
        \draw (9) -- (15);
        \draw (10) -- (14);

        \draw (11) -- (12);
        \draw (12) -- (13);
        \draw (13) -- (14);
        \draw (14) -- (15);

        \draw (11) -- (17);
        \draw (12) -- (16);
        \draw (13) -- (19);
        \draw (14) -- (18);

        \draw (11) -- (6);
        \draw (12) -- (7);
        \draw (13) -- (8);
        \draw (14) -- (9);
        \draw (15) -- (10);
        
        \draw (17) -- (23);
        \draw (18) -- (22);
        \draw (19) -- (25);
        \draw (20) -- (24);

        \draw (16) -- (17);
        \draw (17) -- (18);
        \draw (18) -- (19);
        \draw (19) -- (20);

        \draw (11) -- (16);
        \draw (12) -- (17);
        \draw (13) -- (18);
        \draw (14) -- (19);
        \draw (15) -- (20);

        \draw (21) -- (22);
        \draw (22) -- (23);
        \draw (23) -- (24);
        \draw (24) -- (25);

        \draw (21) -- (16);
        \draw (22) -- (17);
        \draw (23) -- (18);
        \draw (24) -- (19);
        \draw (25) -- (20);

        \draw[dashed] (1) -- (0,-1);
        \draw[dashed] (2) -- (1,-1);
        \draw[dashed] (3) -- (2,-1);
        \draw[dashed] (4) -- (3,-1);
        \draw[dashed] (5) -- (4,-1);

        \draw[dashed] (1) -- (-1,-1);
        \draw[dashed] (2) -- (2,-1);
        \draw[dashed] (3) -- (1,-1);
        \draw[dashed] (4) -- (4,-1);
        \draw[dashed] (5) -- (3,-1);

        \draw[dashed] (21) -- (0,5);
        \draw[dashed] (22) -- (1,5);
        \draw[dashed] (23) -- (2,5);
        \draw[dashed] (24) -- (3,5);
        \draw[dashed] (25) -- (4,5);

        \draw[dashed] (21) -- (1,5);
        \draw[dashed] (22) -- (0,5);
        \draw[dashed] (23) -- (3,5);
        \draw[dashed] (24) -- (2,5);
        \draw[dashed] (25) -- (5,5);
        
        \draw[dashed] (1) -- (-1,0);
        \draw[dashed] (6) -- (-1,1);
        \draw[dashed] (11) -- (-1,2);
        \draw[dashed] (16) -- (-1,3);
        \draw[dashed] (21) -- (-1,4);

        \draw[dashed] (6) -- (-1,2);
        \draw[dashed] (11) -- (-1,1);
        \draw[dashed] (16) -- (-1,4);
        \draw[dashed] (21) -- (-1,3);

        \draw[dashed] (5) -- (5,0);
        \draw[dashed] (10) -- (5,1);
        \draw[dashed] (15) -- (5,2);
        \draw[dashed] (20) -- (5,3);
        \draw[dashed] (25) -- (5,4);

        \draw[dashed] (5) -- (5,1);
        \draw[dashed] (10) -- (5,0);
        \draw[dashed] (15) -- (5,3);
        \draw[dashed] (20) -- (5,2);
    \end{tikzpicture}
    \caption{Illustration of the chessboard lattice}
    \label{fig:chessboard_lattice}
    \end{center}
\end{figure}
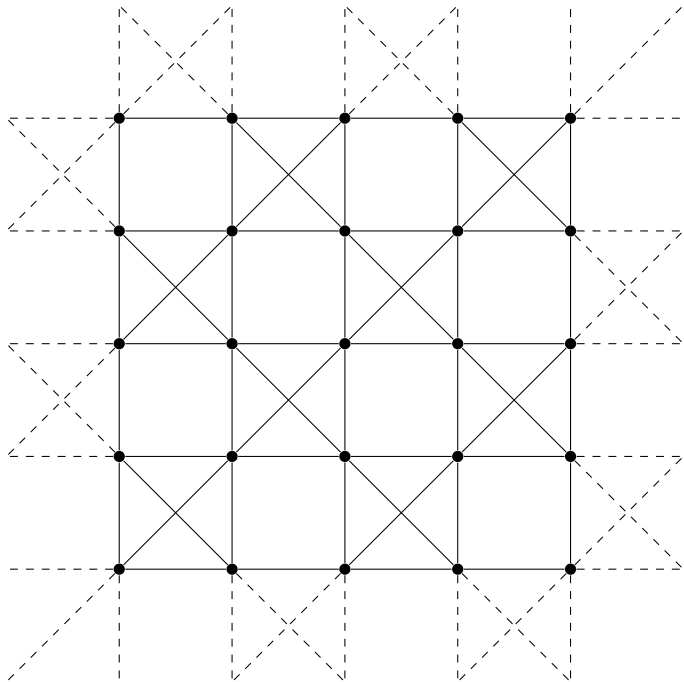

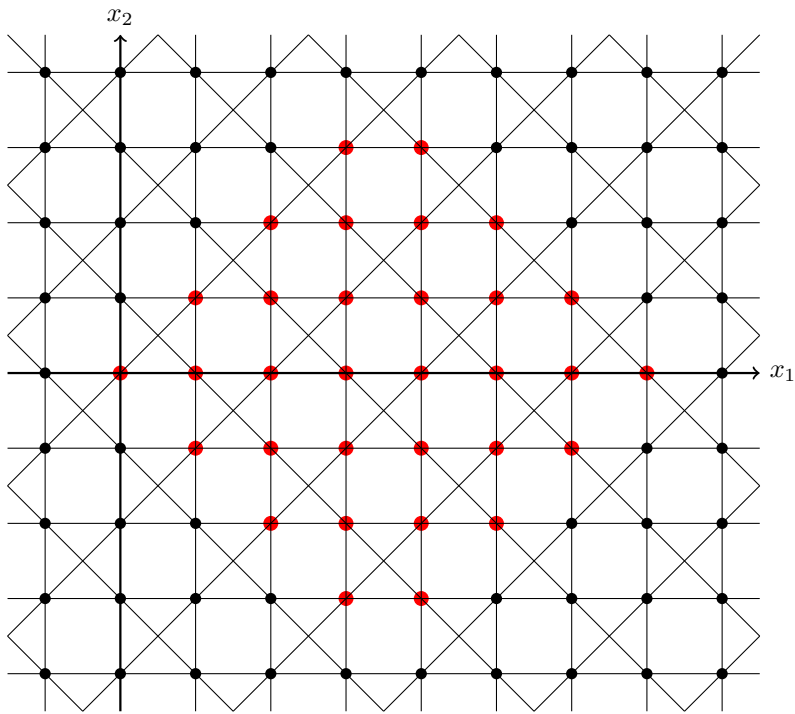
\begin{figure}
    \begin{center}

    \begin{tikzpicture}
        \tikzstyle{default point}=[circle, fill=black, inner sep=1.5pt]
        \tikzstyle{highlight point}=[circle, fill=red, inner sep=2pt]
        
        \foreach \x in {-1,8}{
            \foreach \y in {-4,...,4} {
                \node[default point] at (\x, \y) {};
            }
        }
        \foreach \y in {-4,...,-1} {
            \node[default point] at (0, \y) {};
        }
        \node[highlight point] at (0, 0) {};
        \foreach \y in {1,...,4} {
            \node[default point] at (0, \y) {};
        }
        \foreach \y in {-4,...,-2} {
            \node[default point] at (1, \y) {};
        }
        \foreach \y in {-1,...,1} {
            \node[highlight point] at (1, \y) {};
        }
        \foreach \y in {2,...,4} {
            \node[default point] at (1, \y) {};
        }
        \foreach \y in {-4,...,-3} {
            \node[default point] at (2, \y) {};
        }
        \foreach \y in {-2,...,2} {
            \node[highlight point] at (2, \y) {};
        }
        \foreach \y in {3,...,4} {
            \node[default point] at (2, \y) {};
        }
        
        \node[default point] at (3, -4) {};
        \foreach \y in {-3,...,3} {
            \node[highlight point] at (3, \y) {};
        }
        \node[default point] at (3, 4) {};

        \node[default point] at (4, -4) {};
        
        \foreach \y in {-3,...,3} {
            \node[highlight point] at (4, \y) {};
        }
        \node[default point] at (4, 4) {};
        \foreach \y in {-4,...,-3} {
            \node[default point] at (5, \y) {};
        }
        \foreach \y in {-2,...,2} {
            \node[highlight point] at (5, \y) {};
        }
        \foreach \y in {3,...,4} {
            \node[default point] at (5, \y) {};
        }
        \foreach \y in {-4,...,-2} {
            \node[default point] at (6, \y) {};
        }
        \foreach \y in {-1,...,1} {
            \node[highlight point] at (6, \y) {};
        }
        \foreach \y in {2,...,4} {
            \node[default point] at (6, \y) {};
        }
        \foreach \y in {-4,...,-1} {
            \node[default point] at (7, \y) {};
        }
        \node[highlight point] at (7, 0) {};

        \foreach \y in {1,...,4} {
            \node[default point] at (7, \y) {};
        }

        \draw[line width= 0.3mm, ->] (-1.5, 0) -- (8.5, 0) node[right] {$x_1$};
        \draw[line width= 0.3mm, ->] (0, -4.5) -- (0, 4.5) node[above] {$x_2$};
        
        \draw (-1.5,1) -- (8.5,1);
        \draw (-1.5,2) -- (8.5,2);
        \draw (-1.5,3) -- (8.5,3);
        \draw (-1.5,4) -- (8.5,4);
        \draw (-1.5,-1) -- (8.5,-1);
        \draw (-1.5,-2) -- (8.5,-2);
        \draw (-1.5,-3) -- (8.5,-3);
        \draw (-1.5,-4) -- (8.5,-4);

        \draw (-1,-4.5) -- (-1,4.5);
        \draw (1,-4.5) -- (1,4.5);
        \draw (2,-4.5) -- (2,4.5);
        \draw (3,-4.5) -- (3,4.5);
        \draw (4,-4.5) -- (4,4.5);
        \draw (5,-4.5) -- (5,4.5);
        \draw (6,-4.5) -- (6,4.5);
        \draw (7,-4.5) -- (7,4.5);
        \draw (8,-4.5) -- (8,4.5);

        \draw (-1.5, 2.5) -- (0.5,4.5);
        \draw (-1.5, 0.5) -- (2.5,4.5);
        \draw (-1.5,-1.5) -- (4.5,4.5);
        \draw (-1.5,-3.5) -- (6.5,4.5);
        \draw (-0.5,-4.5) -- (8.5,4.5);
        \draw (1.5,-4.5) -- (8.5,2.5);
        \draw (3.5,-4.5) -- (8.5,0.5);
        \draw (5.5,-4.5) -- (8.5,-1.5);
        \draw (7.5,-4.5) -- (8.5,-3.5);

        \draw (-1.5,4.5) -- (7.5, -4.5);
        \draw (0.5,4.5) -- (8.5, -3.5);
        \draw (2.5,4.5) -- (8.5, -1.5);
        \draw (4.5,4.5) -- (8.5, 0.5);
        \draw (6.5,4.5) -- (8.5, 2.5);

        \draw (-1.5,2.5) -- (5.5, -4.5);
        \draw (-1.5,0.5) -- (3.5, -4.5);
        \draw (-1.5,-1.5) -- (1.5, -4.5);
        \draw (-1.5,-3.5) -- (-0.5, -4.5);
        
    \end{tikzpicture}
    \end{center}
    \caption{The fundamental domain of the finite chessboard graph highlighted in red}
    \label{fig:fundamental_domain}
\end{figure}

Another counterexample can be constructed from the chessboard lattice, illustrated in Figure~\ref{fig:chessboard_lattice}. To do so, we identify the vertices $x,y\in \Z^2$ if and only if 
\[
    x - y \in \Z \begin{pmatrix}
        4 \\
        4
        \end{pmatrix}
        + \Z \begin{pmatrix}
            4 \\
            -4
            \end{pmatrix}.
\]
The fundamental domain of the construction is illustrated in Figure~\ref{fig:fundamental_domain}.
The resulting graph is 6-regular, bone-idle and satisfies
\[
    \beta_1 = 0 < \frac{\deg_{\min}}{2}.
\]
To verify this, a Python implementation for computing the Betti number is provided in a public repository.\footnote{\url{https://github.com/moritzmath/Cell-Complex-Betti-Number}}. Furthermore, this repository contains an efficient Python implementation to calculate the Bakry-\'Emery curvature on graphs.

\bigskip

\textbf{Acknowledgements.} We want to thank Jürgen Jost, Mark Kempton, Gabor Lippner, Christian Rose and Shing-Tung Yau for fruitful discussions. M.H. wants to thank Max von Renesse for his invaluable support. M.H. is partly supported by BMBF (Federal Ministry of Education and Research) in DAAD project 57616814 (\href{https://secai.org/}{SECAI, School of Embedded Composite AI}).

\printbibliography

\end{document}